\documentclass{article}
\usepackage[utf8]{inputenc}

\usepackage{amssymb}
\usepackage{amsthm}
\usepackage{amsmath,amscd}
\usepackage[mathscr]{euscript}
\usepackage[all]{xy}
\usepackage{lmodern}
\usepackage[T1]{fontenc}
\usepackage[textwidth=14cm,hcentering]{geometry}
\usepackage[colorlinks=true,linkcolor=red,citecolor=blue]{hyperref}
\usepackage{cleveref}
\usepackage{mathtools}
\usepackage{tikz}
\usepackage{fnpct}
\usetikzlibrary{cd}
\usetikzlibrary{arrows,positioning}
\crefname{equation}{Figure}{Figure}

\title{Separability in homotopical algebra}
\author{Maxime Ramzi}
\date{}

\newtheorem{thm}{Theorem}[section]
\newtheorem{lm}[thm]{Lemma}
\newtheorem{prop}[thm]{Proposition}
\newtheorem{cor}[thm]{Corollary}

\newtheorem*{thm*}{Theorem}

\theoremstyle{definition}
\newtheorem{defn}[thm]{Definition}
\newtheorem{cons}[thm]{Construction}
\newtheorem{assu}[thm]{Assumption}
\newtheorem{nota}[thm]{Notation}
\newtheorem{conv}[thm]{Convention}
\newtheorem{ex}[thm]{Example}

\newtheorem{rmk}[thm]{Remark}
\newtheorem{ques}[thm]{Question}
\newtheorem{conj}[thm]{Conjecture}
\newtheorem{warn}[thm]{Warning}
\newtheorem{obs}[thm]{Observation}
\newtheorem{var}[thm]{Variant}

\newtheorem{thmx}{Theorem}

\newtheorem{corx}[thmx]{Corollary}

\newcommand{\op}{^{\mathrm{op}}}
\newcommand{\cat}{\mathbf}
\newcommand{\Cat}{\cat{Cat}}
\newcommand{\on}{\operatorname}
\newcommand{\id}{\mathrm{id}}
\newcommand{\Fun}{\on{Fun}}
\newcommand{\map}{\on{map}}
\newcommand{\Map}{\on{Map}}

\newcommand{\Sph}{\mathbb S}

\newcommand{\Set}{\cat{Set}}
\newcommand{\Ab}{\cat{Ab}}
\newcommand{\Ss}{\cat S}
\newcommand{\Sp}{\cat{Sp}}

\newcommand{\Fin}{\mathrm{Fin}}

\newcommand{\PrL}{\cat{Pr}^\mathrm{L} }

\newcommand{\Alg}{\mathrm{Alg}}
\newcommand{\CAlg}{\mathrm{CAlg}}
\newcommand{\LMod}{\cat{LMod}}
\newcommand{\RMod}{\cat{RMod}}
\newcommand{\Mod}{\cat{Mod}}
\newcommand{\BiMod}{\cat{BiMod}}
\newcommand{\idem}{\mathrm{Idem}}
\newcommand{\HH}{\mathrm{HH}}

\newcommand{\tr}{\mathrm{tr}}
\newcommand{\THH}{\mathrm{THH}}
\newcommand{\Perf}{\mathbf{Perf}}

\newcommand{\Ind}{\mathrm{Ind}}
\newcommand{\End}{\mathrm{End}}
\newcommand{\Idem}{\mathrm{Idem}}
\newcommand{\pt}{\mathrm{pt}}
\newcommand{\colim}{\mathrm{colim}}

\newcommand{\heart}{\heartsuit}
\newcommand{\CSep}{\mathrm{CSep}}
\newcommand{\KO}{\mathrm{KO}}
\newcommand{\KU}{\mathrm{KU}}
\newcommand{\Span}{\mathrm{Span}}
\newcommand{\Corr}{\mathrm{Corr}}
\newcommand{\ho}{\mathrm{ho}}
\newcommand{\C}{\cat{C}}
\newcommand{\D}{\cat{D}}
\newcommand{\E}{\cat{E}}

\newcommand{\M}{\cat{M}}
\newcommand{\one}{\mathbf{1}}
\newcommand{\Aut}{\mathrm{Aut}}
\newcommand{\Syn}{\mathrm{Syn}}
\newcommand{\Proj}{\mathrm{Proj}}
\newcommand{\Oo}{\mathcal{O}}
\newcommand{\Projsep}{\mathrm{ProjSep}}
\newcommand{\cn}{\mathrm{cn}}
\newcommand{\QCoh}{\mathbf{QCoh}}
\newcommand{\Mil}{\mathrm{Mil}}
\newcommand{\coMod}{\mathbf{coMod}}
\newcommand{\Grp}{\mathrm{Grp}}
\newcommand{\Mon}{\mathrm{Mon}}
\newcommand{\Rep}{\mathbf{Rep}}
\newcommand{\Spec}{\mathrm{Spec}}

\newcommand{\Pic}{\mathrm{Pic}}
\newcommand{\Tel}{\mathrm{Tel}}
\newcommand{\MU}{\mathrm{MU}}

\newcommand{\category}{$\infty$-category}
\newcommand{\categories}{$\infty$-categories}
\newcommand{\operad}{$\infty$-operad}
\newcommand{\operads}{$\infty$-operads}

\DeclareFontFamily{U}{min}{}
\DeclareFontShape{U}{min}{m}{n}{<-> udmj30}{}
\newcommand{\pend}{\unskip\nobreak\hfill$\triangleleft$}

\begin{document}

\maketitle
\begin{abstract}
We study the notion of \emph{separable algebras} in the context of symmetric monoidal stable $\infty$-categories. In the first part of this paper, we compare this context to that of tensor-triangulated categories and show that separable algebras and their modules in a symmetric monoidal stable $\infty$-category are, in large parts, controlled by the (tensor-triangulated) homotopy category. We also study a variant of this notion, which we call ind-separability. Among other things, this provides a partially new proof of the Goerss--Hopkins--Miller theorem about the uniqueness of $\mathbb E_\infty$-structures on Morava $E$-theory.

We later initiate a study of separable algebras \textit{à la} Auslander-Goldman by relating them to Azumaya algebras, and prove in some restrictive cases that centers of separable algebras are separable. Finally, we study the Hochschild homology of separable algebras and prove some descent results in topological Hochschild homology.
\end{abstract}
\tableofcontents
\setcounter{secnumdepth}{0}
\section{Introduction}
\subsection{Overview}
\subsubsection*{Separable algebras}
In classical algebra, separable algebras, introduced by Auslander and Goldman in \cite{AuslanderGoldman}, are a generalization to arbitrary commutative rings of the classical notion of separable field extensions. They are $R$-algebras $A$ for which the multiplication map $A\otimes_R A\op\to A$ admits an $(A,A)$-bimodule section. Commutative separable algebras are closely related to étale algebras, while separable algebras whose center is the base commutative ring are also known as Azumaya algebras, introduced in the context of the Brauer group.

A typical trend in homotopical algebra is the attempt to mirror constructions and notions from classical algebra to ``derived'' contexts, and see what parts of the theory carry over, and what changes. For example, in \cite{balmerseparability}, Balmer initiated the study of separable algebras in tensor-triangulated categories (henceforth, tt-categories). A surprising feature of these algebras in this context is that they admit a good notion of module categories, even at the unstructured level of triangulated categories. In many situations, these module categories recover the ``expected'' result. For example, if $A$ is an étale $R$-algebra, then modules over $A$ in the derived category of $R$ recover the derived category of $A$, a surprising result which is known to be wrong when $A$ is a general $R$-algebra. 

A natural source of tt-categories (to some extent, the only source of ``natural'' tt-categories) is stably symmetric monoidal $\infty$-categories: given any such gadget $\C$, its homotopy category $\ho(\C)$ has a natural structure of a tt-category (and in fact, all the ``enhancements'' that appear in \cite[Section 5]{balmerseparability} - which, from this perspective, are trying to encode as extra structure on $\ho(\C)$ the homotopical data contained in $\C$ - arise in this way). From this point of view, it makes sense to wonder what parts of the tensor-triangulated story extend to the $\infty$-categorical case, and also, what parts of the tt-story are (at least morally) \emph{explained} by the $\infty$-story.

This project started out with a desire to answer these questions, and to clarify the connection between separable algebras in $\C$, and separable algebras in $\ho(\C)$. The first half of this paper is devoted to exactly this; there, we argue that separable algebras in $\C$ and their modules are mostly controlled by the homotopy category $\ho(\C)$, and even more so in the commutative setting. In particular, we answer a folk question by proving that the distinction between separable algebras and homotopy separable homotopy algebras is mild in the associative case, and inexistent in the commutative case (\Cref{thm : e1lift} and \Cref{thm : commlift}). 

We further introduce a variant of separability that works in the commutative case in more ``infinitary'' situations, which we call ind-separability, and prove similar results about this variant; among other things leading to a somewhat new proof of the Goerss--Hopkins--Miller theorem (\Cref{cor:GHM} - see \Cref{rmk:isitanewproof?} for a discussion of the sense of the word ``new''). 

\subsubsection*{The higher algebra of separable algebras}
Beyond their nice behaviour with respect to homotopy categories (or more generally, tt-categories), separable algebras are interesting in their own right: as we mentioned before, separability can be seen as an analogue of étale-ness. For example, Balmer proves in \cite{balmerNT} that étale maps of schemes induce separable algebras, and Neeman proves in \cite{Neeman} that over a noetherian scheme, this is not far from an exhaustive list of commutative separable algebras. See also the recent work of Naumann and Pol \cite{NikoLuca} for another comparison of separable commutative algebras and another notion of ``(finite) étale'' due to Mathew \cite{Akhilgalois}. 

Now, classically, separable algebras can be neatly organized in the following way: if $A$ is a separable algebra over $R$, its center $C$ is separable over $R$ and commutative, hence ``étale'', and $A$ is separable over $C$ and central, and hence Azumaya. Thus, separable algebras can be studied by studying separately the commutative, slightly more geometric case, and the central case, closely related to Brauer groups. In a later part of this paper, we try to replicate this story, originally due to Auslander and Goldman \cite{AuslanderGoldman} in the case of ring spectra. Along the way, we correct a mistake in \cite{BRS}, namely we prove that not all Azumaya algebras are separable by giving a number of examples, and formulate a criterion for when a given Azumaya algebra \emph{is} separable. 

\subsubsection*{The Hochschild homology of separable algebras}
Separable algebras are defined in terms that have a clear connection to (topological) Hochschild homology. This connection was already explored to some extent in \cite[Section 9]{rognes} and \cite[Section 1]{BRS}. We close this paper by studying the absolute Hochschild homology of separable extensions under a strengthening of the separability condition (see \Cref{defn:abssep}), which allows us to study some descent results in topological Hochschild homology, as many Galois extensions satisfy this strengthening of separability. 
 
\subsection{Main results}
We now describe our main results in more detail.

As explained earlier, our first goal is to relate separable algebras in $\C$ and in $\ho(\C)$. In the associative case, we prove the following, which we state as a single theorem, but is actually spread across \Cref{section:homotopycat}:
\begin{thmx}\label{thmx:assoc}[{\Cref{thm : e1lift}, \Cref{prop:homsepimpliessep}, \Cref{thm : homod=mod}, and \Cref{thm : morsep}}]
    Let $\C$ be an additively symmetric monoidal $\infty$-category. Given any algebra $A$ in $\ho(\C)$, which is separable as an algebra in $\ho(\C)$, the moduli space $\Alg(\C)^\simeq\times_{\Alg(\ho(\C))^\simeq}\{A\}$ is simply-connected, and any algebra $\tilde A$ lifting $A$ is separable in $\C$.

    Furthermore, in this situation, the canonical functor $\ho(\LMod_{\tilde A}(\C))\to \LMod_A(\ho(\C))$ is an equivalence. 

    Finally, given another algebra $R$ in $\C$, the canonical map $\map_{\Alg(\C)}(\tilde A,R)\to \hom_{\Alg(\ho(\C))}(A,R)$ is a $\pi_0$-isomorphism. 
\end{thmx}
\begin{rmk}
    The simple-connectedness of the moduli space of lifts cannot be improved to a contractibility statement, cf. \Cref{ex:modulinotcontr}. Similarly, the final statement about $\pi_0$ cannot be improved to a space-level statement, even if the source is commutative and the target separable, cf. \Cref{ex:spacenotequiv}. See below for the case where the target is (homotopy) commutative. 
\pend \end{rmk}
We further prove that in the commutative case, the obstructions to contractibility vanish, namely:
\begin{thmx}\label{thmx:comm}[{\Cref{prop : e1comm}, \Cref{thm : commlift}, and \Cref{cor:hocommtargetdiscrete}}]
    Let $\C$ be an additively symmetric monoidal $\infty$-category, and let $A$ be an algebra in $\ho(\C)$ which is separable therein, and homotopy commutative. In this case, the moduli space of lifts to an associative algebra in $\C$ is contractible. 

    More generally, for any $d\geq 1$, the moduli space $\Alg_{\mathbb E_d}(\C)^\simeq\times_{\Alg(\ho(\C))^\simeq}\{A\}$ is contractible - including for $d=\infty$. In particular, $A$ admits an essentially unique lift $\tilde A$ to a commutative algebra in $\C$.  

    Furthermore, for any algebra $R$ in $\C$ which is homotopy commutative, the canonical map $\map_{\Alg(\C)}(\tilde A,R)\to \hom_{\Alg(\ho(\C))}(A,R)$ is an equivalence - the source is discrete. If $R$ is a commutative algebra, then these two spaces are also equivalent (via the canonical map) to $\map_{\CAlg(\C)}(\tilde A,R)$. 
\end{thmx}
As a sample application of these results, in \Cref{subsection:SW}, we use Ravenel and Wilson's computations from \cite{ravenelwilson} to prove (a corrected version of) Sati and Westerland's main results from \cite{satiwesterland}. In some sense, our proof is simpler as it does not involve any obstruction theory. 
\newline 

Although Morava $E$-theory is not separable in $K(n)$-local spectra, we prove that it is close enough to being separable that some of our results still apply to it. In more detail, we introduce the notion of an (homotopy) \emph{ind-separable} algebra in \Cref{section:indsep}, and prove an analogue of \Cref{thmx:comm} for ind-separable algebras. We further prove, using as only input a computation of $\pi_*(L_{K(n)}(E\otimes E))$, that Morava $E$-theory is homotopy ind-separable, and we thus recover the Goerss--Hopkins--Miller theorem (we prove a more precise version, also for morphisms, cf. \Cref{cor:GHM}):
\begin{corx}
Let $E=E(k,\mathbf G)$ be a Morava $E$-theory, where $k$ is a perfect field of characteristic $p$ and $\mathbf G$ a formal group over $k$. For any $d\geq 1$, the moduli space $\Alg_{\mathbb E_d}(\Sp)^\simeq\times_{\Alg(\ho(\Sp))^\simeq}\{E\}$ is contractible. 
\end{corx}
\begin{rmk}
    When $d=1$, this is the Hopkins-Miller theorem, and when $d=\infty$, this is its extension to the Goerss-Hopkins-Miller theorem. This result, for intermediary values of $d$, is well-known to experts, but does not seem to have been recorded in the literature. 
\pend \end{rmk}
\begin{rmk}
    Our proof of this theorem is also based on obstruction theory - we refer to \Cref{rmk:isitanewproof?} for a discussion of the difference between our proof and previous proofs.
\pend \end{rmk}
We then study Auslander-Goldman theory. In this direction, our results are only partial. A special case of what we prove is:
\begin{thmx}[{\Cref{thm:Brauer} and \Cref{prop:Azsep}}]
    Let $R$ be a commutative ring spectrum satisfying the assumptions of \Cref{thm:Brauer}. In this case, any dualizable central separable algebra over $R$ is Azumaya. 

    Conversely, in any additive presentably symmetric monoidal $\infty$-category $\C$, an Azumaya algebra $A$ is separable if and only if its unit $\one_\C\to A$ admits a retraction. 
\end{thmx}
\begin{rmk}
    The assumptions of \Cref{thm:Brauer} cover all connective ring spectra, and all ring spectra ``coming from chromatic homotopy theory'', but they are nonetheless a bit restrictive. 
\end{rmk}
We also study the question of whether centers of separable algebras are separable, although we only reach results in more restricted generality:
\begin{thmx}[{\Cref{thm:center}}]
    Let $R$ be a connective commutative ring spectrum and let $A$ be an almost perfect $R$-algebra. If $A$ is separable, then so is its center.

    The same holds for separable algebras in $K(n)$-local $E$-modules, where $E$ is Morava $E$-theory, and for separable algebras in $K(n)$-local spectra.
\end{thmx}

We finally study the Hochschild homology of separable algebras. Under a technical assumption which strenghtens separability in a relative context, we obtain descent results in \Cref{section:descent}. Special cases worth recording are:
\begin{corx}[{\Cref{cor:Galoisdescent}}]
    The $C_2$-Galois extension $\KO\to \KU$ and the $(\mathbb Z/p)^\times$-Galois extension $L_p\to \KU_p$ (the inclusion of the Adams-summand in $p$-completed $K$-theory) satisfy descent in topological Hochschild homology ($p$-completed for the latter). 
\end{corx}
\subsection{Outline}
Let us describe the contents of this paper linearly. 

In \Cref{section:generalities}, we start by gathering generalities about separable algebras and their module categories. 

In \Cref{section:homotopycat}, we study in more depth the relation between separable algebras in $\C$ and in $\ho(\C)$ in the associative case. This is where we prove \Cref{thmx:assoc}. Most of the proof is relatively elementary, except for the proof that the moduli space of lifts is non-empty, where we use some deformation theory of \categories. In \Cref{section:comm}, we move on to the commutative case, where we prove \Cref{thmx:comm} - in contrast to the previous section, there is no deformation theory in these proofs, and they are relatively elementary. In this section, we also study the analogy between separable and étale algebras in the sense of Lurie.  

In \Cref{section:indsep}, we study a variant of separability, namely ind-separability, and use it to recover the Goerss--Hopkins--Miller theorem. 

In \Cref{section:examples}, we gather a number of examples of separable algebras, to indicate the wealth of examples despite how strong this condition is. 

\Cref{section:brauer} is where we investigate Auslander-Goldman theory in homotopical algebra. We raise a number of questions we were not able to answer in this direction. 

Finally, in \Cref{section:descent}, we study the Hochschild homology of separable algebras and related descent properties. 

We have two short appendices: in \Cref{app:cyctrace}, we give a proof that the trace pairing of a dualizable algebra is $C_2$-equivariant in a coherent sense (this is used for a single result in the paper, namely in \Cref{prop:cohstrongsep}); and in \Cref{app:epi} we compare the condition that a map of algebras be an epimorphism to a condition involving tensor products (this is \emph{not} used in the paper, but is related to \Cref{item:epi} in \Cref{lm:general}).
\begin{rmk}
    We emphasize that \Cref{section:generalities,section:homotopycat,section:comm} are mostly elementary. The only part requiring some nontrivial technology is \Cref{thm : e1lift}, and in turn this relies only on deformation theory of \categories{} in the guise of \Cref{lm:Postnikovcats}. In other words, all of the ``rigidification'' results are completely elementary if we start with an $\mathbb E_1$-algebra, and not obstruction-theoretic. 

    On the other hand, the version involving ind-separabiltiy is more involved and we prove it using the obstruction theory fromm \cite{PVK}. We recommend the reader have a look at that paper if they want to go in depth in the proof of \Cref{thm:indsepBrown}. 
\end{rmk}
\subsection{Conventions}
We freely use the language of \categories{} as extensively developed by Lurie in \cite{HTT,HA}. We view ordinary categories as \categories{} via the nerve, and typically suppress the nerve from the notation. When we want to stress that a \category{} is (the nerve of) an ordinary category, we say ``$1$-category''. Unless explicitly specified, all our categorical notions such as co/limits, adjunctions, etc. are to be understood in the sense of \categories. 
\begin{enumerate}
\item $\Ss,\Sp$ denote, respectively, the \categories{} of spaces\footnote{aka $\infty$-groupoids, homotopy types or anima.} and of spectra.
    \item Throughout this paper, $\C$ will denote a symmetric monoidal \category, satisfying various extra conditions. When we say that an \category{} is ``(semi)additively'' (resp. ``stably'', ``presentably'') symmetric monoidal, we mean that it is a symmetric monoidal \category{} whose underlying \category{} is (semi)additive (resp. stable, presentable), and where the tensor product is compatible with this structure, that is, commutes with coproducts in each variable (resp. finite colimits, all colimits). 
    \item If there is no specified \operad{}, the word ``algebra'' (resp. the notation $\Alg$) means ``associative or equivalently $\mathbb E_1$-algebra'' (resp. denotes the \category{} of associative algebras). 
    \item We use $\ho(C)$ to denote the homotopy category of an \category. When $X$ is an object of $C$ (possibly with some extra structure), we write $hX$ for the same object viewed as an object of $\ho(C)$ (with the appropriate extra structure, in $\ho(C)$). We append the word ``homotopy'' to a type of structure to mean ``that type of structure, considered in $\ho(C)$''. For example, a ``homotopy algebra'' is an algebra in $\ho(C)$. 
    \item We say ``geometric realization'' for ``(homotopy) colimit indexed by $\Delta\op$''.
    \item We write $\map$ for mapping spaces, $\Map$ for mapping spectra in stable \categories, and $\hom$ for hom-sets in $1$-categories. We also use $\hom$ when talking about internal homs of closed symmetric monoidal \categories. 
    \item A common trick consists in embedding a small \category{} (possibly with some extra structure) in its presheaf \category{} to reduce to proving statements about presentable \categories{}, or simply \categories{} with suitable colimits. This is usually compatible with multiplicative structures, essentially by \cite[Section 4.8.1]{HA} (see, e.g., \cite[Proposition 4.8.1.10]{HA}). We will usually simply say ``up to adding enough colimits'' to mean ``without loss of generality, assume $\C$ has these colimits'', i.e., to refer to this trick.
    \item Modules are usually left modules. For commutative algebras, we simply write $\Mod$, while for associative algebras we always specify and write $\LMod$ or $\RMod$. For an algebra $A$, we shorten ``$(A,A)$-bimodule'' to ``$A$-bimodule'' and write $\BiMod_A$, while for two algebras $A$ and $B$ we write $_A\BiMod_B$ for $(A,B)$-bimodule (a left $A$-action and a right $B$-action). If $\C$ is clear from context, we sometimes write $\LMod_A$ for $\LMod_A(\C)$. 
    \item By default, everything is derived. If we write $\Mod_R$, for an ordinary ring $R$, we mean the \category{} modules \emph{in} $\Sp$, equivalently the derived \category{} of $R$. We use $\heart$ to denote hearts of t-structures, so $\Mod_R^\heart$ denotes the abelian category of discrete $R$-modules - we also use it for variants, such as $\CAlg_R^\heart$, which means $\CAlg(\Mod_R^\heart)$. 
\end{enumerate}
\subsection{Acknowledgements}
It is my pleasure to acknowledge the help of my advisors, Jesper Grodal and Markus Land, for their support in my projects, their helpful feedback and comments, and finally their listening to my numerous rants about separable algebras for more than a year.
A big chunk of this project would not have been possible without the help of Robert Burklund, who explained to me what synthetic spectra were. He is also responsible for my not giving up on Morava $E$-theory once I realized it was not separable. I heartily thank him for his help and patience. I also wish to thank Piotr Pstrągowski for patiently answering many questions.  

I thank Haldun \"Ozg\"ur Bay{\i}nd{\i}r for suggesting I think about \Cref{thm:Moravaindsepperf}, and for asking whether I could prove the results from \Cref{subsection:SW} with the tools developed in this paper. 

 I had several interesting discussions related to this project with Anish Chedalavada and Luca Pol, and I also want to thank Luca Pol and Niko Naumann for sharing a draft of their paper \cite{NikoLuca} which led to improvements in \Cref{section:etale,section:center}. While reading eachother's drafts, we realized that there was some overlap between our preliminaries on separable algebras, and I have tried to record it accordingly. 

Jan Steinebrunner pointed out the simple proof of \Cref{prop:C2trace} that now features in the appendix, which meant I did not have to go through too many hoops to include its consequence, \Cref{prop:cohstrongsep}. He also helped me with the drawings from \Cref{app:cyctrace}. 

Finally, I have had helpful conversations related to this project with Itamar Mor and Andy Baker . 

This work was supported by the Danish National Research Foundation through the Copenhagen Centre for Geometry and Topology (DNRF151), and a portion of this work was completed while the author was in residence at the Institut Mittag-Leffler in Djursholm, Sweden in 2022 as part of the program ‘Higher algebraic structures in algebra, topology and geometry’ supported by the Swedish Research Council under grant no. 2016-06596. 
\setcounter{secnumdepth}{2}
\section{Generalities}\label{section:generalities}
The goal of this section is to set the stage: we define separable algebras, and gather some of their basic properties. 
\begin{nota}
Throughout this section, $\C$ is a symmetric monoidal \category, with unit $\one$ and tensor product denoted by $\otimes$. 
\pend \end{nota}
Following \cite{balmerseparability}, we define: 
\begin{defn}\label{defn:sep}
An algebra $A\in \Alg(\C)$ is said to be \emph{separable} if the multiplication map, $A\otimes A\op\to A$ admits a section $s$, as a map of $A$-bimodules. 

In this case, we call the composite $\one \to A\xrightarrow{s}A\otimes A\op$, or sometimes the section $s$ itself, a \emph{separability idempotent}. Equivalently, this is an $A\otimes A\op$-linear idempotent map $A\otimes A\op\to A\otimes A\op$. 
\pend \end{defn}

\begin{var}
An algebra $A\in \Alg(\C)$ is said to be \emph{homotopy separable} if it is separable, as an algebra in $\ho(\C)$. 

If we start with an algebra $A\in \Alg(\ho(\C))$ directly, we will say that we have a \emph{homotopy separable homotopy algebra}. 
\end{var}
\begin{var}
Suppose $\C$ admits geometric realizations which are compatible with the tensor product. In particular, $\C$ admits relative tensor products. 

Let $R\in \CAlg(\C)$ be a commutative algebra, and $A\in \Alg(\Mod_R(\C))$ be an $R$-algebra. We say $A$ is separable over $R$ if $A$ is separable as an algebra in $\Mod_R(\C)$. If $A\in\CAlg(\Mod_R(\C))$, we will say that it is a \emph{separable extension} of $R$.
\end{var}

\begin{rmk}
Recall that an algebra is said to be \emph{smooth} if $A$ is right or left dualizable over $A\otimes A\op$ \cite[Definition 4.6.4.13.]{HA}. If $\C$ is idempotent-complete, dualizable objects are closed under retracts, and so a separable algebra is smooth. 

One can therefore think of separability as a strenghtening of smoothness. 
\pend \end{rmk}
\subsection{Basic properties}
Separable algebras enjoy a number of closure properties:
\begin{lm}\label{lm:general}
Let $A,B\in \Alg(\C)$ be algebras. 
\begin{enumerate}
    \item The unit of $\C$, $\one$, is separable; more generally if $A$ is an idempotent algebra \cite[Definition 4.8.2.8.]{HA}, then it is separable.

    \item\label{item:epi} Suppose $\C$ admits geometric realizations, compatible with the tensor product, and suppose that the map $B\otimes_A B\to B$ is an equivalence. If $A$ is separable, then so is $B$. 
    \item If $A$ and $B$ are separable, then so is $A\otimes B$.
    \item Suppose $\C$ is semiadditively symmetric monoidal. The product $A\times B$ is separable if and only if both $A$ and $B$ are.
    \item If there is a retraction $A\to B\to A$ in $\Alg(\C)$, and $B$ is separable, then so is $A$. 
    \item If $A$ is separable, then so is $A\op$.
    \item\label{item : symmonsep} If $f: \C\to \D$ is a symmetric monoidal functor, and $A$ is separable, then so is $f(A)$.
\end{enumerate}
\end{lm}
\begin{proof}
1. is clear, as the multiplication map $A\otimes A\op\to A$ of an idempotent algebra is an equivalence (on underlying objects, and hence as bimodules).

For 2., observe that basechange along $f$ induces a functor $\BiMod_A\to \BiMod_B$ that sends $A\otimes A\op$ to $B\otimes B\op$ by design, and the bimodule $A$ to the bimodule $B\otimes_A A\otimes_A B \simeq B\otimes_A B\simeq B$, where the last equivalence is by assumption. Further, the multiplication map is sent to the multiplication map, and the existence of a section in the source guarantees the existence of a section in the target.

3. follows from the existence of a functor $\BiMod_A\times\BiMod_B\to \BiMod_{A\otimes B}$ sending $(A\otimes A\op, B\otimes B\op)$ to $(A\otimes B)\otimes (A\otimes B)\op$ and $(A,B)$ to $A\otimes B$.  Similarly as above, the multiplication map is sent to the multiplication map, and the existence of a section in the source guarantees the existence of a section in the target. 

For one direction of 4., we note that the multiplication map $(A\times B)\otimes (A\times B)\op\to A\times B$ factors as $(A\times B)\otimes (A\times B)\op \to (A\otimes A\op)\times (B\otimes B\op)\to A\times B$. If $A,B$ are separable, then the second map has a bimodule section, and so the claim follows from the fact that, in the semiadditive case, with no assumption on $A,B$, the first map also has a bimodule section. 

Conversely,  notice that in the semiadditive case, the projections $A\times B\to A$ (resp. $B$) are strong epimorphisms, so that by 2., if $A\times B$ is separable, then so are $A$ and $B$.
 
For 5., let $A\xrightarrow{i}B\xrightarrow{r}A$ denote a retraction diagram. We view $B$-bimodules as $A$-bimodules via restriction along $i$. Then, $A\xrightarrow{i}B\xrightarrow{s}B\otimes B\op \xrightarrow{r\otimes r\op} A\otimes A\op$ is an $A$-bimodule map, where $s$ is the separability idempotent of $B$. Furthermore, the composite $$A\xrightarrow{i}B\xrightarrow{s}B\otimes B\op \xrightarrow{r\otimes r\op} A\otimes A\op\to A$$ is equivalent to 
$$A\xrightarrow{i}B\xrightarrow{s}B\otimes B\op\to B \xrightarrow{r}A$$
because $r$ is an algebra map, and thus, because $s$ is a separability idempotent, to $A\to B\to A$ and thus, because we started with a retraction diagram, to $\id_A$, and so we are done.

6. and 7 are clear (in fact, 6. follows from 7. as the equivalence $\C\simeq \C^{\mathrm{rev}}$ sends $A$ to $A\op$). 
\end{proof}
\begin{rmk}
It follows from 1. and 3. that if $A$ is separable, then so is $A\otimes A\op$. 
\pend \end{rmk}
\begin{rmk}
\Cref{item : symmonsep} really requires a symmetric monoidal functor, and not just a lax symmetric monoidal one. 
\pend \end{rmk}
\begin{rmk}
We will see in \Cref{prop:centerdescent}, that \Cref{item : symmonsep} has a form of converse in the commutative case, if we assume that $f$ is more than conservative, rather part of a limit decomposition of $\C$. 
\pend \end{rmk}
\begin{rmk}
    The condition in \Cref{item:epi} implies that $A\to B$ is an epimorphism in $\Alg(\C)$. We do not know whether being an epimorphism is sufficient. Note that if $\C$ is furthermore stable, then this condition is \emph{equivalent} to being an epimorphism.  
\pend \end{rmk}
\begin{lm}
Suppose $\C$ is semi-additively symmetric monoidal. Then an algebra $A\in\Alg(\C)$ is separable if and only if $A$ is projective as an $A\otimes A\op$-module, i.e. if and only if there exists some finite $n$ and a retraction of $(A\otimes A\op)^n$ onto $A$.
\end{lm}
\begin{proof}
Clearly separability implies the projectivity condition, with $n=1$. 

For the converse, fix a retraction diagram $A\xrightarrow{i} (A\otimes A\op)^n \xrightarrow{p}A$. Write $p$ as $(p_i)_i$, where each $p_i: A\otimes A\op\to A$ is $p$ on the $i$th summand.

Observe that the unit map $\one\to A$ lifts through $A\otimes A\op\xrightarrow{\mu}A$, so that each $p_i: A\otimes A\op\to A$ lifts as well, now as a bimodule map. Fix a lift $\tilde p_i$ to get $\tilde p: (A\otimes A\op)^n\to A\otimes A\op$. Then $\tilde p\circ i$ is a section of $\mu$. 
\end{proof}
\begin{rmk}\label{rmk : notMorita}
A consequence of this characterization is that in the setting of classical algebra, as projectivity is Morita invariant, separability also is. 

This is \emph{wrong} in the generality that we are in. The following counterexample was pointed out to me by Robert Burklund: one can show that if $X$ is a finite spectrum which generates $\Sp^\omega$ as a thick subcategory, then $\End(X)$ is separable if and only if the unit map $\Sph\to \End(X)$ splits (cf. \Cref{prop:Azsep}), while $\End(X)$ is always Morita equivalent to the sphere spectrum $\Sph$, which is of course separable. Yet there are such spectra such that the unit map does \emph{not} split, such as $X = \Sph/\eta$, the cone of $\eta\in\pi_1(\Sph)$.  

One can analyze this example and make it more general - in particular, one can make a similar example in some category of representations of some group over $\mathbb Q$, and so have such examples in characteristic $0$. 

One could instead formulate a notion of ``projective Morita equivalence'', and prove that separability is projective-Morita invariant.  
\pend \end{rmk}
\subsection{Modules over separable algebras}
We now move on to discussing modules over separable algebras. The main observation in this realm is the following: 
\begin{prop}
Let $A$ be an algebra in $\C$, and consider the free-forgetful adjunction $A\otimes -: \C\rightleftarrows \LMod_A(\C): U$. 

$A$ is separable if and only if the co-unit $A\otimes U(-) \to \id_{\LMod_A(\C)}$ admits a natural $\C$-linear section. 
\end{prop}
\begin{proof}
There is a functor $\Fun_\C(\LMod_A(\C),\LMod_A(\C))\to \BiMod_A(\C)$ given informally by evaluation at the object $A\in \LMod_A(\C)$ \cite[Remark 4.6.2.9., Theorem 4.8.4.1.]{HA}\footnote{Under relatively mild hypotheses on $\C$, this can be made into an equivalence by restricting the domain a little: if $\C$ admits geometric realizations compatible with the tensor product, then $\C$-linear endofunctors of $\LMod_A(\C)$ that commute with geometric realizations are exactly given by bimodules \cite[Theorem 4.8.4.1.]{HA}. For this part of the proof, we do not need an equivalence.  }.

This functor sends $A\otimes U(-)$ to $A\otimes A\op$ as an $A$-bimodule, and $\id_{\Mod_A(\C)}$ to $A$ itself, with its canonical $A$-bimodule structure. In particular, the existence of a natural section as indicated implies the existence of a bimodule section. 

Conversely, suppose that $A$ is separable. Up to embedding $\C$ in a symmetric monoidal \category{} admitting geometric realizations compatible with the tensor product, we may assume that $\C$ has these properties. In that case, the above restriction functor induces an equivalence $\Fun^\Delta_\C(\LMod_A(\C),\LMod_A(\C))\simeq \BiMod_A(\C)$, and so we can reverse the argument from above.

Concretely, the section is described as follows : $$M\xleftarrow{\simeq} A\otimes_A M\xrightarrow{s\otimes_A M}(A\otimes A\op)\otimes_A M \simeq A\otimes M$$
\end{proof}
\begin{cor}\label{cor : retractfree}
If $A$ is a separable algebra in $\C$, then any $A$-module $M$ is a retract of the free $A$-module $A\otimes M$. 
\end{cor}

It will be convenient to have a generalization of this observation in the following direction: if $\M$ is equipped with a coherent tensoring by $\C$, i.e. $\M$ is a $\C$-module, the notion of $A$-module \emph{in} $\M$ makes sense. We have:
\begin{cor}\label{cor : retractfreeinmodule}
Let $\M$ be a $\C$-module, $A$ an algebra in $\C$, and consider the free-forgetful adjunction $A\otimes- : \M\leftrightarrows \LMod_A(\M): U$. If $A$ is separable, then the co-unit $A\otimes U(-)\to \id_{\LMod_A(\M)}$ admits a natural $\C$-linear section. 

In particular, any $A$-module in $\M$, $M$, is a retract of the free $A$-module $A\otimes M$. 
\end{cor}
\begin{proof}
    As in the previous proof - by embedding $\C,\M$ in categories that have geometric realizations compatible with the tensor product (resp. the tensoring of $\C$), we may assume that they have these properties. 

    In this case, $\LMod_A(\M)\simeq \LMod_A(\C)\otimes_\C \M$, and the free-forgetful adjunction for $\M$ is identified with $-\otimes_\C \M$ applied to the free forgetful adjunction for $\C$. The result follows. 

We note that the section has the same concrete description as in the case of $\M=\C$. 
\end{proof}
\begin{rmk}
    We will have several results that hold for an arbitrary $\C$-module $\M$. While this always implies the result for the special case $\M= \C$, we will typically state this special case explicitly, to help with intuition. 
\pend \end{rmk}
Thus, separability allows us to deduce things about $A$-modules based on underlying properties. For instance:
\begin{cor}\label{cor:sectionunderlying}
Let $\M$ be a $\C$-module and let $A$ be a separable algebra in $\C$. Consider a map $M\to N$ of $A$-modules in $\M$. If it has a retraction in $\M$ (resp. a section), then it does so in $\Mod_A(\M)$ as well.
\end{cor}
\begin{proof}
    The map $M\to N$ is a retract of the map $A\otimes M\to A\otimes N$, and the property of having a section (resp. a retraction) is closed under retracts.
\end{proof}
Similarly, we have:
\begin{cor}\label{cor:nullunderlying}
Let $\C$ be a pointed symmetric monoidal category in which $\otimes$ preserves the zero object, and $A\in \Alg(\C)$ a separable algebra. Let $f: M\to N$ be a morphism in $\LMod_A(\C)$, whose underlying map in $\C$ is nullhomotopic, i.e. factors through $0$. 

In this case, $f$ is nullhomotopic in $\LMod_A(\C)$. 

The same holds for morphisms in $\LMod_A(\M)$ whose underlying morphism in $\M$ is nullhomotopic, for any $\C$-module $\M$. 
\end{cor}
\begin{proof}
The proof is the same: retracts of nullhomotopic maps are nullhomotopic. 
\end{proof}
\begin{rmk}
Note that this fact is famously not true in general \categories{} if we do not assume separability. For example, let $A= \End_\mathbb Z(\mathbb Z/p)$, and view $\mathbb Z/p$ as an $A$-module. Then $p : \mathbb Z/p\to \mathbb Z/p$ is not zero as an $A$-module map, but its underlying map is $0$. 
\pend \end{rmk}

\Cref{cor : retractfree}, as well its extension to \Cref{cor : retractfreeinmodule} will be crucial in the next section, where we analyze the relation of separable algebras to homotopy categories, but we can already make good use of it to analyze relative tensor products and internal homs. 

We recall that, for a right (resp. left) $A$-module $M$ (resp. $N$) in $\C$, we can form a simplicial object $\mathrm{Bar}(M,A,N)_\bullet: \Delta\op\to \C$, compatibly with symmetric monoidal functors $\C\to \D$, and that its colimit, if it exists, is the relative tensor product $M\otimes_A N$, cf. \cite[Section 4.4.2.]{HA}. Given a symmetric monoidal functor $f:\C\to \D$, we have a canonical equivalence $f\circ \mathrm{Bar}(M,A,N)_\bullet \simeq \mathrm{Bar}(f(M),f(A),f(N))_\bullet$. 
\begin{defn}
    Let $f:\C\to\D$ be a symmetric monoidal functor, $A\in\Alg(\C)$ an algebra in $\C$, $M$ (resp. $N$) a right (resp. left) $A$-module. 

    We say that $f$ \emph{preserves} the relative tensor product $M\otimes_A N$ if it exists in $\C$, and $f$ preserves the colimit $\colim_{\Delta\op}\mathrm{Bar}(M,A,N)_\bullet$. 
\pend \end{defn}
We can then state:
\begin{prop}\label{prop : tensorsep}
Assume $\C$ admits geometric realizations which are compatible with the tensor product. Let $A\in\Alg(\C)$ be a separable algebra, and $M,N$ be a right $A$-module and a left $A$-module respectively. 

In this case, the relative tensor product $M\otimes_A N$ is a (natural, $\C$-linearly on both sides) retract of $M\otimes N$. 

In particular, if we now remove the assumption that $\C$ admits geometric realizations and replace it with $\C$ being idempotent complete, then $\C$ still admits relative tensor products of $A$-modules; and they are preserved by any symmetric monoidal functor $\C\to \E$. 
\end{prop}
\begin{proof}
The second part can be deduced from the first as follows: freely add geometric realizations to $\C$ to obtain a fully faithful symmetric monoidal functor $\C\to \D$ where $\D$ satisfies the hypotheses of the first part. The image of $A$ in $\D$ is still separable, and so the tensor product $M\otimes_A N$, computed in $\D$, lives in $\C$, because $M\otimes N$ does and $\C$ is idempotent complete (here, we use the first part). Therefore, this colimit of the bar construction is a colimit in $\C$ as well. 

For the first part, we note that $N$ is a (natural, $\C$-linear) retract of $A\otimes N$, so that $M\otimes_A N$ is a (natural, $\C$-linearly on both sides) retract of $M\otimes_A (A\otimes N)\simeq M\otimes N$, as was claimed. 

From the proof, it is clear that these relative tensor products are preserved by any symmetric monoidal functor, because retractions are; this proves the final part. 
\end{proof}
\begin{rmk}
One could instead phrase this, and the next proof, in terms of the canonical resolution of the left $A$-module $N$, namely a simplicial object which looks like  $[n]\mapsto A^{\otimes n+1}\otimes N$. In those terms, the statement would be that the corresponding colimit diagram $N\simeq \colim_{\Delta\op}A^{\otimes n+1}\otimes N$ is an absolute colimit diagram, as it is a retract of the corresponding diagram for $A\otimes N$, which is split augmented and hence an absolute colimit diagram. See the proof of \cite[Lemma 4.7]{NikoLuca} for an argument in this direction.  
\pend \end{rmk}
We now move on to hom objects. Given two left $A$-modules $M,N$, the hom-object from $M$ to $N$, $\hom_A(M,N)$, is the object of $\C$ equipped with a map of left $A$-modules $ev:M\otimes\hom_A(M,N) \to N$ which, if it exists, satisfies the following universal property: restriction along $ev$ induces an equivalence $\map_\C(c,\hom_A(M,N))\simeq \map_{\LMod_A(\C)}(c\otimes M,N)$. 

If $f:\C\to \D$ is a symmetric monoidal functor, and if $\hom_A(M,N)$ exists, then we obtain a map of left $f(A)$-modules $f(M)\otimes f(\hom_A(M,N)) \to f(N)$. 
\begin{defn}
    Let $f:\C\to \D$ be a symmetric monoidal functor, $A\in\Alg(\C)$ an algebra in $\C$ and $M,N$ left $A$-modules. We say that $f$ preserves the internal hom $\hom_A(M,N)$ if it exists, and the induced map $f(M)\otimes f(\hom_A(M,N))\to f(N)$ exhibits $f(\hom_A(M,N))$ as a hom object from $f(M)$ to $f(N)$. 
\pend \end{defn}
We begin with a well-known lemma: 
\begin{lm}\label{lm:hominmod}
    Let $\C$ be symmetric monoidal, and assume it admits totalizations of cosimplicial objects as well as internal homs. 

    In this case, for any algebra $A\in\Alg(\C)$, the right-$\C$-module $\LMod_A(\C)$ admits hom-objects in $\C$. 
\end{lm}
\begin{proof}
    We note that because $\C$ admits internal homs, its tensor product is compatible with any colimit that exists in $\C$.

    Now, given $Y\in\LMod_A(\C)$, we note the following two things: first, if $\C$ admits $I\op$-shaped limits, then the property that $\hom_A(X,Y)$ exist is closed under $I$-shaped colimits in $X$, and second, for any $X\in \C$, $\hom_A(A\otimes X,Y)$ exists. 

    For the first one, we note that indeed, the condition that $\hom_A(X,Y)$ exists is by definition the condition that $c\mapsto\map_A(X\otimes c,Y)$ be a representable functor $\C\op\to \Ss$. Representable functors are closed under $I\op$-shaped limits by assumption, and $-\otimes c$ preserves any colimits that exist in $\C$, so the claim follows at once.

    For the second one, we note that $\map_A(A\otimes X\otimes c, Y)\simeq \map(X\otimes c,Y)\simeq \map(c,\hom(X,Y))$, so $\hom_A(A\otimes X,Y)$ exists and is equivalent to $\hom(X,Y)$. 

    With these two things in hand, we can conclude: any $A$-module is the colimit of a $\Delta\op$-shaped diagram, all of whose terms are of the form $A\otimes X$ for some $X$ \cite[Proposition 4.7.3.14]{HA}. 
\end{proof}
\begin{prop}\label{prop : homsep}
Assume $\C$ admits totalizations of cosimplicial objects, and internal homs. Let $A\in\Alg(\C)$ be a separable algebra, and $M,N\in \LMod_A(\C)$. 
In this case, $\hom_A(M,N)\in \C$ exists, and is a retract of $\hom(M,N)$. Furthermore, any symmetric monoidal functor $\C\to \E$ which is also \emph{closed}, or more generally, which preserves $\hom(M,N)$, preserves $\hom_A(M,N)$. 

If we remove the assumption that $\C$ admits totalizations, while keeping the existence of $\hom(M,N)$ and we assume that $\C$ is idempotent complete, then we get the same conclusion about $\hom_A(M,N)$. 
\end{prop}
\begin{proof}
We begin under the assumption that $\C$ admits totalizations and internal homs. 

In this case, by \Cref{lm:hominmod}, $\hom_A(M,N)$ exists and is a retract of of $\hom_A(A\otimes M,N)\simeq \hom(M,N)$. It is clear that this is preserved by any symmetric monoidal functor which preserves $\hom(M,N)$. 

Now, we go back to a general idempotent-complete $\C$. There is a symmetric monoidal embedding $\C\to \E$ where $\E$ admits totalizations, and which preserves all homs that exist in $\C$ : in fact, the Yoneda embedding into the Day convolution monoidal structure on presheaves has this property. In particular, $\hom_A(M,N)$ in $\E$ is a retract of $\hom(M,N)$ in $\C$, and thus is in $\C$ by idempotent-completeness. The conclusion about preservation follows similarly. 
\end{proof}
An internal hom of specific interest is the center of $A$:
\begin{cor}\label{cor:centerabsolute}
Let $\C$ be an idempotent complete symmetric monoidal \category, and $A\in\Alg(\C)$ a separable algebra. In this case, $Z(A)=\hom_{A\otimes A\op}(A,A)$ exists and is a retract of $A$.  

Furthermore, it is preserved by any symmetric monoidal functor $\C\to \E$. 
\end{cor} 
\begin{nota}\label{nota:center}
    We introduce here the notation $Z(A)=\hom_{A\otimes A\op}(A,A)$ - this is the $\mathbb E_1$-center of $A$, which is an $\mathbb E_2$-algebra \cite[Section 5.3.]{HA}. 
\pend \end{nota}
\begin{rmk}\label{rmk:centerretractlinear}
    Note that the retraction $A\to Z(A)$ is given by precomposition by $s: A\to A\otimes A\op$: $$A\simeq \hom_{A\otimes A\op}(A\otimes A\op, A)\to \hom_{A\otimes A\op}(A,A)=Z(A)$$
    In particular, it has a canonical left $Z(A)$-linear structure. 
\pend \end{rmk}
We conclude this section with the following classical fact:

\begin{prop}\label{prop : towersep}
Assume $\C$ is idempotent-complete, and let $R\in\CAlg(\C)$ be separable. In this case, relative tensor products over $R$ exist and so $\Mod_R(\C)$ is symmetric monoidal. 

Let $A\in \Alg(\Mod_R(\C))$. If $A$ is separable over $R$, then it is separable.
\end{prop}

\begin{proof}
Suppose $A$ is separable over $R$. We then have an $A\otimes_R A\op$-linear, and hence $A\otimes A\op$-linear section $A\to A\otimes_R A\op$. But now, because $R$ is separable, the latter is $A\otimes A\op$-linearly a retract of $A\otimes A\op$. Composing the two retraction gives the claim. 
\end{proof}
We prove the converse in the case of an additive \category{} in \Cref{prop:conversetowersep}. 
\section{Separable algebras and homotopy categories}\label{section:homotopycat}
In this section, we explain how separable algebras in $\C$ are controlled by the homotopy category $\ho(\C)$. This suggests that a big chunk of the study of separable algebras can be performed in the homotopy category, and thus explains morally why separable algebras work so well in tt-categories. This also allows to \emph{lift} many results about separable algebras in the tt-setting to the stable $\infty$-setting.
\subsection{Modules, separability and algebras}
The key property that will drive our analysis is \Cref{cor : retractfree}, which, recall, states that over a separable algebra, any module is a retract of a free module of the form $A\otimes M$. We will use it together with the following general fact: 
\begin{lm}\label{lm : locff}
Let $A\in\Alg(\C)$ be any algebra, and let $X\in \LMod_A$ be a retract of a module of the form $A\otimes M$.

For any $N\in \LMod_A$, the functor $\ho(\LMod_A(\C)) \to \LMod_{hA}(\ho(\C))$ induces an isomorphim $$\pi_0\map_{\LMod_A(\C)}(X,N)\xrightarrow{\cong} \hom_{\LMod_{hA}(\ho(\C))}(hX,hN)$$

More generally, if $\M$ is a $\C$-module, and $X\in\LMod_A(\M)$ is a retract of some module of the form $A\otimes M$, then for any $N\in\LMod_A(\M)$, the functor $\ho(\LMod_A(\M))\to \LMod_{hA}(\ho(\M))$ induces an isomorphism 
$$\pi_0\map_{\LMod_A(\M)}(X,N)\xrightarrow{\cong} \hom_{\LMod_{hA}(\ho(\M))}(hX,hN)$$
\end{lm}
\begin{proof}
The collection of $X$'s for which this map is an isomorphism is clearly closed under retract, so we may assume that $X$ is free on some $M$. But then $hX$ is free on the same $hM$, with the same unit map, from which the claim follows. 
\end{proof}

\begin{cor}\label{cor : sepff}
Let $A\in\Alg(\C)$ be a separable algebra. The functor $\ho(\LMod_A(\C))\to \LMod_{hA}(\ho(\C))$ is fully faithful.

More generally, if $\M$ is a $\C$-module, then $\ho(\LMod_A(\M))\to \LMod_{hA}(\ho(\M))$ is fully faithful.
\end{cor}
\begin{proof}
    This follows from the previous lemma together with \Cref{cor : retractfree} (resp. \Cref{cor : retractfreeinmodule}). 
\end{proof}
\begin{conv}
In the rest of this paper, $\C$ will be assumed to be additive. We will however repeat it in the statements of results for self-containedness. 
\end{conv}
It is not clear to the author whether this condition is necessary, and where, but we use it in some key instances, so it is certainly necessary \emph{for our proofs}, if not the results. Note that additivity is a place where both \categories{} and their homotopy categories can live, so it is a suitable inbetween between \categories{} and $1$-categories. We use this assumption together with the following lemma:
\begin{lm}\label{lm : idemadd}
Let $\C$ be an additive category and $e: X\to X$ an idempotent in $\ho(\C)$. There exists a coherent idempotent $\idem \to\C$ which lifts $e$ (cf. \cite[Section 4.4.5.]{HTT}). 

In particular, if $\C$ is idempotent-complete, then so is $\ho(\C)$. 
\end{lm}

\begin{proof}
This is \cite[Lemma 1.2.4.6., Remark 1.2.4.9.]{HA} - note that as stated, the assumption is that $\C$ is stable, but the proof works just as well if $\C$ is additive. Alternatively, one can deduce the additive case from the stable case by embedding any additive category in a stable one. 
\end{proof}
With all of this, we can show: 
\begin{thm}\label{thm : homod=mod}
Suppose $\C$ is additively symmetric monoidal, and let $A\in \Alg(\C)$ be a separable algebra. The forgetful functor $\ho(\LMod_A(\C))\to \LMod_{hA}(\ho(\C))$ is an equivalence.

More generally, if $\M$ is a $\C$-module, then the forgetful functor $\ho(\LMod_A(\M))\to \LMod_{hA}(\ho(\M))$ is an equivalence. 
\end{thm}
In the stable case, and for $\C=\M$, this is also a consequence of the main theorem of \cite{dellambrogiosanders} (together with \cite{balmerseparability}).
\begin{proof}
We have already shown it is fully faithful, so we are left with proving that it is essentially surjective. Note that we can assume without loss of generality that $\M$ is idempotent-complete: indeed, assume for a second the claim holds for idempotent complete categories, and let $\M\to \M'$ be the idempotent-completion of $\M$, which is in particular a fully faithful functor.

Let $M\in \LMod_{hA}(\ho(\M))\subset\LMod_{hA}(\ho(\M'))$. By the idempotent complete case, this can be lifted to an $A$-module in $\M'$. But the underlying object of $M$ is in $\ho(\M)$, and therefore the underlying object of this lift is in $\M$, which proves the claim. 

So we now assume $\M$ is idempotent complete. It follows that $\LMod_A(\M)$ is also idempotent-complete, cf. \cite[Corollary 4.2.3.3.]{HA} and \cite[Remark 4.4.5.13.]{HTT}. It is also additive and therefore by \Cref{lm : idemadd}, $\ho(\LMod_A(\M))$ is also idempotent complete. So to prove that a fully faithful functor $\ho(\LMod_A(\M))\to \D$ is essentially surjective, it suffices to show that any object in $\D$ is a retract of some object in the image; but here any object of $\LMod_{hA}(\ho(\M))$ is a retract of some $hA\otimes N$, by separability, and $hA\otimes N$ is the image of $A\otimes N$, so we are done.  
\end{proof}
This is the first instance of how separable algebras behave nicely with respect to homotopy categories.

\begin{rmk}
By \Cref{prop : tensorsep} and \Cref{prop : homsep} applied to the symmetric monoidal functor $\C\to \ho(\C)$, the functor $\LMod_A(\C)\to \LMod_{hA}(\ho(\C))$ is compatible with relative tensor products over $A$ and internal homs over $A$ that is, the tensor product $M\otimes_A N$ is the coequalizer in $\ho(\C)$ over the two maps $M\otimes A\otimes N\rightrightarrows M\otimes N$, and the internal hom $\hom_A(M,N)$ is the equalizer in $\ho(\C)$ of $\hom(M,N)\rightrightarrows \hom(A\otimes M,N)$.  

This is therefore compatible with Balmer's construction in the triangulated setting \cite[Section 1]{balmerdegree}. 

Because the functor $\ho(\LMod_A(\C))\to\LMod_{hA}(\ho(\C))$ is also an equivalence, this gives an explanation, at least in the case of tensor triangulated categories which are homotopy categories of symmetric monoidal stable \categories, of the fact that module categories over separable algebras are still triangulated, and why their tensor product behaves nicely. 

We note that this is also discussed in the recent work of Naumann and Pol, see \cite[Lemma 4.7, Remark 4.9]{NikoLuca}. 
\pend \end{rmk}
\begin{rmk}
The previous remark, as well as the equivalence $\ho(\LMod_A(\C))\simeq \LMod_{hA}(\ho(\C))$ applied to $A\otimes B\op$, for separable algebras $A$ and $B$, shows that for separable algebras, there is a reasonable notion of Morita equivalence at the level of the homotopy category, that can be phrased in terms of bimodules in the ``naive'' way. See also \Cref{thm : morsep} for an application of this idea to \emph{morphisms} between separable algebras.  
\pend \end{rmk}

We next show that the picture is even more rigid: separability is detected at the level of the homotopy category, more precisely: 
\begin{prop}\label{prop:homsepimpliessep}
Suppose $\C$ is additively symmetric monoidal and let $A\in\Alg(\C)$ be a homotopy separable algebra. In this case, $A$ is separable. 
\end{prop}
For this, we use the following categorical facts: 

\begin{lm}[{\cite[Example 1.4.7.10. (Tag 00JC)]{kerodon}}]\label{lm : endo}
Let $\Delta^1/\partial \Delta^1\to B\mathbb N$ be the canonical map. It is a categorical equivalence. 
\end{lm}
\begin{lm}{\cite[4.4.5.15.]{HTT}}
The canonical map $\mathbb N_\geq\to \idem$ is cofinal. 
\end{lm}
Moreover, note that this canonical map is given by the following composite : $\mathbb N_\geq \to B\mathbb N\to \idem$ so that we have the following commutative diagram for any category $\D$:
\[\begin{tikzcd}
	{\Fun(\Delta^1/\partial\Delta^1,\D)} & {\Fun(B\mathbb N,\D)} & {\Fun(\mathbb N_\geq,\D)} \\
	& {\Fun(\idem,\D)}
	\arrow["\simeq"', from=1-2, to=1-1]
	\arrow[from=1-2, to=1-3]
	\arrow[from=2-2, to=1-2]
	\arrow[from=2-2, to=1-3]
\end{tikzcd}\]

With this in hand, we can prove the claim:
\begin{proof}[Proof of \Cref{prop:homsepimpliessep}]
Without loss of generality, we assume $\C$ is idempotent complete. 

Consider the idempotent $s: hA\otimes hA\op\to hA\to hA\otimes hA\op$ in the category of homotopy $A$-bimodules.
Note that its source and target are free $A\otimes A\op$-bimodules, and the functor $\ho(\LMod_{A\otimes A\op}(\C))\to \LMod_{h(A\otimes A\op)}(\ho(\C))$ is fully faithful on the full subcategory spanned by $A\otimes A\op$, which proves that this idempotent lifts to a coherent idempotent in $\LMod_{A\otimes A\op}(\C)$, by \Cref{lm : idemadd}. That is, we have a functor $\tilde s: \Idem\to \LMod_{A\otimes A\op}(\C)$ that classifies $s$. 

In the diagram $$\xymatrix{\Fun(\Delta^1/\partial\Delta^1,\LMod_{A\otimes A\op}(\C)) & \Fun(B\mathbb N,\LMod_{A\otimes A\op}(\C))\ar[l]_\simeq \ar[r] & \Fun(\mathbb N_\geq,\LMod_{A\otimes A\op}(\C)) \\
&  \Fun(\idem,\LMod_{A\otimes A\op}(\C)) \ar[u] \ar[ru]}$$ 
if we follow $\tilde s$ up and then left, we simply get $A\otimes A\op\xrightarrow{s}A\otimes A\op$. Now, in $\Fun(\Delta^1/\partial\Delta^1,\C)$, we have an arrow that corresponds to the following commutative square in $\LMod_{A\otimes A\op}(\C)$: 
\[\begin{tikzcd}
	{A\otimes A\op} & {A\otimes A\op} \\
	A & A
	\arrow["s", from=1-1, to=1-2]
	\arrow["{\id_A}"', from=2-1, to=2-2]
	\arrow["\mu"', from=1-1, to=2-1]
	\arrow["\mu", from=1-2, to=2-2]
\end{tikzcd}\]
Note that there exists such a commutative square in $\LMod_{A\otimes A\op}(\C)$, because there is one in $\LMod_{h(A\otimes A\op)}(\ho(\C))$, and the source is a free module (cf. \Cref{lm : locff}).

In particular, if we now go from $\Fun(\Delta^1/\partial\Delta^1, \LMod_{A\otimes A\op}(\C))$ to $\Fun(\mathbb N_{\geq}, \LMod_{A\otimes A\op}(\C))$, we get a map from $A\otimes A\op\xrightarrow{s}A\otimes A\op\xrightarrow{s}\dots$ to $A\xrightarrow{\id_A}A\xrightarrow{\id_A}\dots$. By commutativity of the diagram, the source is simply the restriction of $\tilde s$ along the map $\mathbb N_{\geq}\to \idem$. 

In particular, the source has a colimit given by the splitting of that idempotent, and we get a map from this colimit to $A$ in $\LMod_{A\otimes A\op}(\C)$. But splitting of idempotents are \emph{absolute} colimits \cite[Corollary 4.4.5.12.]{HTT}, so they are preserved by the forgetful functor $\LMod_{A\otimes A\op}(\C)\to \LMod_{h(A\otimes A\op)}(\ho(\C))$. If we redo this story in the latter category, $A$ is the splitting of the idempotent in question, and so the canonical map from this colimit to $A$ is an equivalence. 

It is therefore an equivalence in $\ho(\C)$, and therefore in $\C$, and therefore in $\LMod_{A\otimes A\op}(\C)$. This proves that $A$ is the splitting of some idempotent on $A\otimes A\op$ in $\LMod_{A\otimes A\op}(\C)$, along $A\otimes A\op\xrightarrow{\mu} A$, which is exactly saying that $A$ is separable, and so we are done. 
\end{proof}
We can now prove the converse of \Cref{prop : towersep}, namely:
\begin{prop}\label{prop:conversetowersep}
Assume $\C$ is additively symmetric monoidal and idempotent-complete, and let $R\in\CAlg(\C)$ be separable. In this case, relative tensor products over $R$ exist and so $\Mod_R(\C)$ is symmetric monoidal. 

Let $A\in \Alg(\Mod_R(\C))$. In this case, $A$ is separable if and only if it is separable over $R$.
\end{prop}
\begin{warn}
Classically, if $A$ is separable, then it is so over $R$, with no separability assumption on $R$. This is wrong in homotopical algebra and if one tries to run the classical proof, one will encounter the issue that $A\otimes A\op\to A\otimes_R A\op$ is not an epimorphism.

A counterexample is given by the $\mathbb Q$-algebras $R = \mathbb Q[x]$ and $A=$ any nonzero separable commutative $\mathbb Q$-algebra, all of this in the \category{} of $\mathbb Q$-module spectra. 
\pend \end{warn}
\begin{proof}
We have proved in \Cref{prop : towersep} that if $A$ was separable over $R$, it was separable. Now, assume $A$ is separable.

We observe that by \Cref{prop:homsepimpliessep}, it suffices to show that $A$ is separable in $\ho(\Mod_R(\C))$. But because $R$ is separable, $\ho(\Mod_R(\C))\simeq \Mod_{hR}(\ho(\C))$, symmetric monoidally as the relative tensor products are preserved. In other words, we may assume that $\C$ is a $1$-category. 

But now $A\otimes A\op\to A\otimes_R A\op$ is a split epimorphism because $R$ is separable, and in a $1$-category, split morphisms are epimorphisms. It follows that an $A\otimes A\op$-linear map between $A\otimes_R A\op$-modules is automatically $A\otimes_R A\op$-linear. 

The following composite $A\xrightarrow{s} A\otimes A\op\to A\otimes_R A\op$, where $s$ is a witness that $A$ is separable, is therefore an $A\otimes_R A\op$-linear section of the multiplication map, which proves the claim. 
\end{proof}
We now exploit what we did so far to analyze morphisms between separable algebras. Our main result in the noncommutative world is:
\begin{thm}\label{thm : morsep}
Assume $\C$ is additively symmetric monoidal, and let $A,R\in\Alg(\C)$. If $A$ is separable, the canonical map $$\pi_0\map_{\Alg(\C)}(A,R)\to \hom_{\Alg(\ho(\C))}(hA,hR)$$ is an isomorphism.
\end{thm}

\begin{warn}\label{warn:nondiscrete}
    In general, $\map_{\Alg(\C)}(A,R)$ is not discrete, even if $R$ is also separable, see \Cref{prop:2loopmoduli} and \Cref{ex:modulinotcontr}.
    \pend \end{warn}
The situation is better in the commutative world, as we will see in \Cref{prop : e1comm} and \Cref{thm : commlift}. 
\begin{proof}
Up to embedding $\C$ symmetric monoidally and additively in a presentably additively symmetric monoidal \category, we may assume $\C$ is presentably additively symmetric monoidal. 

There is then a fully faithful functor $\Alg(\C)\to (\Mod_\C)_{\C/}, A\mapsto (\LMod_A(\C),A)$ \cite[Theorem 4.8.5.11]{HA} (see also \cite[Remark 4.8.3.25]{HA}), so that $\map_{\Alg(\C)}(A,R)$ can be described as the fiber of $$\map_{\Mod_\C}(\LMod_A(\C),\LMod_R(\C))\to \map_{\Mod_\C}(\C,\LMod_R(\C))$$ at the map $\C\to \LMod_R(\C)$ classifying the $R$-module $R$. 

By \cite[Theorem 4.8.4.1]{HA}\footnote{See also \cite[Theorem 4.3.2.7]{HA}}, this map can be rewritten as the forgetful map $$_R\BiMod_A(\C)^\simeq \to \LMod_R(\C)^\simeq$$
For simplicity of notation, we simply write $\LMod_R(\C) = \LMod_R$, and then by \cite[Theorem 4.3.2.7]{HA}, this can again be rewritten as $$\RMod_A(\LMod_R)^\simeq\to\LMod_R^\simeq$$

In more concrete terms: an algebra map $A\to R$ is the same data as a right $A$-module structure on the left $R$-module $R$. It is easy to check that the same holds for $1$-categories, with no presentability assumption\footnote{In fact, we can deduce that it also holds in general with no presentability assumption from the presentable case.}. 

Consider now the diagram:\[\begin{tikzcd}
	{\RMod_A(\LMod_R)} & {\RMod_{hA}(\ho(\LMod_R))} & {\RMod_{hA}(\LMod_{hR})} \\
	{\LMod_R} & {\ho(\LMod_R)} & {\LMod_{hR}}
	\arrow[from=1-1, to=2-1]
	\arrow[from=1-3, to=2-3]
	\arrow[from=1-1, to=1-2]
	\arrow[from=1-2, to=1-3]
	\arrow[from=2-1, to=2-2]
	\arrow[from=2-2, to=2-3]
	\arrow[from=1-2, to=2-2]
\end{tikzcd}\]

When restricted to the full subcategory of $\LMod_R$ (resp. $\ho(\LMod_R)$, resp. $\LMod_{hR}(\ho(\C))$) spanned by $R, R\otimes A^{\otimes n}$, the horizontal maps are fully faithful on homotopy categories: for the left square, this follows from \Cref{cor : sepff} applied to the $\C$-module $\M= \LMod_R$, and for the right square, this follows from \Cref{lm : locff}. 

Passing to groupoid cores and restricting to these components, we see that the horizontal maps are therefore $1$-equivalences, and so the induced maps on fibers are $0$-equivalences - the fiber of the leftmost map is, by the previous argument, $\map_{\Alg(\C)}(A,R)$, while the fiber of the rightmost map is $\hom_{\Alg(\ho(\C))}(hA,hR)$, so that this proves the claim.
\end{proof}

\subsection{From homotopy algebras to algebras}\label{subsection:obstruction}
The final goal of this section is to prove that not only is separability detected in the homotopy category, but that separability of a homotopy algeba is strong enough to guarantee that it can be lifted to an $\mathbb E_1$-algebra in $\C$. In other words, we now aim to prove: 
\begin{thm}\label{thm : e1lift}
Let $\C$ be an additive symmetric monoidal \category, and $A\in \Alg(\ho(\C))$ a homotopy separable homotopy algebra. 

There exists an algebra $\widetilde A\in\Alg(\C)$, necessarily separable, which lifts $A$. In fact, the moduli space of such lifts is simply-connected.
\end{thm}
\begin{warn}
The moduli space of lifts is not contractible in general, cf. \Cref{ex:modulinotcontr}. We will later see however that it \emph{is} contractible in the case of a homotopy commutative separable algebra, see \Cref{prop : e1comm}. 
\pend \end{warn}
Let us first specify explicitly what we mean by ``the moduli space of lifts''. 
\begin{defn}\label{defn:modulilifts}
Let $A\in\Alg(\ho(\C))$ be a homotopy algebra. We define the moduli space of lifts of $A$ to $\Alg(\C)$ to be the space $\Alg(\C)^\simeq\times_{\Alg(\ho(\C))^\simeq}\{A\}$, i.e. the fiber of $\Alg(\C)\to \Alg(\ho(\C))$ at $A$. 
\pend \end{defn}
\begin{obs}\label{obs:onlynonempty}
    By \Cref{thm : morsep} together with the fact that $\Alg(\C)\to \Alg(\ho(\C))$ is conservative, we find that $\Alg(\C)^\simeq\to\Alg(\ho(\C))^\simeq$ is injective on $\pi_0$, and an isomorphism on $\pi_1$ at any point of $\Alg(\C)^\simeq$. Furthermore, $\pi_2(\Alg(\ho(\C))^\simeq) = 0$ at every point, so that to prove that the moduli space of lifts is simply connected, it really suffices to prove that it is non-empty. 
\pend\end{obs}
\begin{obs}\label{obs : loopmoduli}
$A$ lies in its connected component in $\Alg(\ho(\C))$, which is equivalent to $B\Aut_{\Alg(\ho(\C))}(A)$. If $\widetilde A$ is a lift of $A$, then the connected component of this lift in the moduli space is a connected component of the fiber of the map $B\Aut_{\Alg(\C)}(\widetilde A)\to B\Aut_{\Alg(\ho(\C)}(A)$. 

In particular, if we already know that the moduli space is connected, then the loop space of this moduli space at $\widetilde A$ is the fiber of $\Aut_{\Alg(\C)}(\widetilde A)\to \Aut_{\Alg(\ho(\C))}(A)$ at $\id_A$. 
\pend \end{obs}
Our proof relies on deformation theory and \Cref{thm : morsep}, as well as the following lemma (which one could make more precise, but the following version is enough for our purposes): 
\begin{lm}\label{lm:Postnikovcats}
    Let $\C$ be an additively symmetric monoidal \category. There exists an additively symmetric monoidal \category{} $\D$ and a commutative diagram of additively symmetric monoidal \categories: 
    \[\begin{tikzcd}
	\ho_{\leq n+1}(\C) & {\ho_{\leq n}(\C)} \\
	{\ho_{\leq n}(\C)} & \D
	\arrow[from=2-1, to=2-2]
	\arrow[from=1-2, to=2-2]
	\arrow[from=1-1, to=2-1]
	\arrow[from=1-1, to=1-2]
\end{tikzcd}\]
such that the induced map $\ho_{\leq n+1}(\C)\to \ho_{\leq n}(\C)\times_\D \ho_{\leq n}(\C)$ is fully faithful.
\end{lm}
Here, $\ho_{\leq n}(\C)$ denotes the homotopy $n$-category of $\C$.

Taking this lemma for granted, the proof of \Cref{thm : e1lift} is not hard.
\begin{proof}[Proof of \Cref{thm : e1lift}]
Fix a homotopy separable $A\in\Alg(\ho(\C))$.  By \Cref{obs:onlynonempty}, it suffices to prove that $A$ admits some lift $\tilde A$ to $\Alg(\C)$. 

As $\C\simeq \lim_n \ho_{\leq n}(\C)$ and $\Alg(-)$ preserves limits (see \Cref{lm:limpres}), it suffices to prove that any given separable $A_n\in \Alg(\ho_{\leq n}(\C))$ admits a lift to $\Alg(\ho_{\leq n+1}(\C))$. For this, we apply \Cref{lm:Postnikovcats}: fix a square \[\begin{tikzcd}
	{\ho_{\leq n+1}(\C)} & {\ho_{\leq n}(\C)} \\
	{\ho_{\leq n}(\C)} & \D
	\arrow["{d_1}", from=1-2, to=2-2]
	\arrow[from=1-1, to=2-1]
	\arrow[from=1-1, to=1-2]
	\arrow["{d_0}"', from=2-1, to=2-2]
\end{tikzcd}\] as in the conclusion of that lemma. 

Viewing $A_n$ as a homotopy algebra in $\C$, we find that $d_0(A_n)\simeq d_1(A_n)$ as homotopy algebras in $\D$. But both are separable, as they are the image of the separable $A_n$ under symmetric monoidal functors $d_0,d_1$. 

In particular, by \Cref{thm : morsep}, this equivalence can be lifted to an equivalence of algebras, and thus we get a lift in the pullback. But the underlying object of this lift is (the underlying object of) $A_n$ in $\ho_{\leq n+1}(\C)$, so that it provides an algebra object $A_{n+1}$ in $\ho_{\leq n+1}(\C)$, by fully faithfulness, and this clearly lifts $A_n$.  
\end{proof}
Finally, let us prove \Cref{lm:Postnikovcats}:
\begin{proof}[Proof of \Cref{lm:Postnikovcats}]
    One could in principle use the methods of \cite{harpaznuitenprasma}, but one would still have to make a number of additional verifications. Instead, let us use the synthetic objects of \cite{hopkinsluriebrauer}.
    
We first deal with the case where $\C$ is stable: let $\E$ be a stably symmetric monoidal \category. We let $\Syn_\E$ denote $\Fun^\times(\E\op,\Sp_{\geq 0})$, this is an additive presentably symmetric monoidal \category, receiving a fully faithful, additive symmetric monoidal functor $\E\to \Syn_\E, c\mapsto \Map(-,c)_{\geq 0}$. 

   By \cite[Proposition 7.3.6.]{hopkinsluriebrauer}, we obtain a pullback square of additively symmetric monoidal \categories, :\[\begin{tikzcd}
	{\Mod_{\one^{\leq n}}(\Syn_\E)} & {\Mod_{\one^{\leq n-1}}(\Syn_\E)} \\
	{\Mod_{\one^{\leq n-1}}(\Syn_\E)} & \D
	\arrow[from=1-1, to=2-1]
	\arrow[from=1-1, to=1-2]
	\arrow[from=2-1, to=2-2]
	\arrow[from=1-2, to=2-2]
	\arrow[from=1-2, to=2-2]
\end{tikzcd}\]
and the composite $\E\to \Syn_\E\to \Mod_{\one^{\leq k}}(\Syn_\E)$ factors through $\ho_{\leq k+1}(\E)$ in a fully faithful way - this is essentially saying that $\tau_{\leq k}\Map(-,\one_\E) \otimes \Map(-,e)\simeq \tau_{\leq k}\Map(-,e)$ which follows from \cite[Lemma 7.1.1. and Corollary 7.3.7.(c)]{hopkinsluriebrauer} (alternatively, the proof of \cite[Lemma 7.1.1.]{hopkinsluriebrauer} works just as well for this statement).

In other words, we have a commuting diagram: 
\[\begin{tikzcd}
	{\ho_{\leq n+1}(\E)} & {\ho_{\leq n}(\E)} \\
	{\ho_{\leq n}(\E)} & {\Mod_{\one^{\leq n}}(\Syn_\E)} & {\Mod_{\one^{\leq n-1}}(\Syn_\E)} \\
	& {\Mod_{\one^{\leq n-1}}(\Syn_\E)} & \D
	\arrow[from=2-2, to=3-2]
	\arrow[from=2-2, to=2-3]
	\arrow[from=3-2, to=3-3]
	\arrow[from=2-3, to=3-3]
	\arrow[from=2-3, to=3-3]
	\arrow[hook, from=2-1, to=3-2]
	\arrow[hook, from=1-1, to=2-2]
	\arrow[hook, from=1-2, to=2-3]
	\arrow[from=1-1, to=2-1]
	\arrow[from=1-1, to=1-2]
\end{tikzcd}\]
where the diagonal arrows are fully faithful. It follows that the map $\ho_{\leq n+1}(\E)\to \ho_{\leq n}(\E)\times_\D \ho_{\leq n}(\E)$ is fully faithful. 

Now, let $\C$ be a general additive \category. Using the Yoneda embedding $\C\to \Fun^\times(\C\op,\Sp_{\geq 0})\to \Fun^\times(\C\op,\Sp)$ \footnote{We implicitly use here that any additive \category{} has a fully faithful Yoneda embedding into its presheaves of connective spectra. This follows from the fact that $\Sp_{\geq 0}\simeq \mathbf{Grp}_{\mathbb E_\infty}$ by the recognition principle.} and considering a small stable subcategory of $\Fun^\times(\C\op,\Sp)$ containing the image of the Yoneda embedding, we find a fully faithful additive symmetric monoidal embedding $\C\to \E$ with $\E$ stably symmetric monoidal. 

The following diagram allows us to conclude: 
\[\begin{tikzcd}
	{\ho_{\leq n+1}(\C)} & {\ho_{\leq n}(\C)} \\
	{\ho_{\leq n}(\C)} & {\ho_{\leq n+1}(\E)} & {\ho_{\leq n}(\E)} \\
	& {\ho_{\leq n}(\E)} & \D
	\arrow[from=2-2, to=3-2]
	\arrow[from=2-2, to=2-3]
	\arrow[from=3-2, to=3-3]
	\arrow[from=2-3, to=3-3]
	\arrow[hook, from=2-1, to=3-2]
	\arrow[hook, from=1-1, to=2-2]
	\arrow[hook, from=1-2, to=2-3]
	\arrow[from=1-1, to=2-1]
	\arrow[from=1-1, to=1-2]
\end{tikzcd}\]
\end{proof}
\begin{rmk}
    In a previous version of this paper, we used obstruction theory to prove \Cref{thm : e1lift}. Unravelling the proof in question, one would arrive at an essentially equivalent proof as the one proposed here. It simply seems to the author that this version is much simpler to parse, and to understand what is going on. 
\pend\end{rmk}
We now analyze the moduli space of lifts a bit further to show that it is not typically contractible. \Cref{thm : e1lift} shows that the moduli space is simply-connected, and we now explain how to describe its $2$-fold loopspace.
\begin{prop}\label{prop:2loopmoduli}
Let $A\in \Alg(\C)$, and $hA$ the corresponding algebra in $\ho(\C)$. Let $\mathcal M$ be the moduli space of lifts of $hA$ to $\Alg(\C)$, and $L_A$ the $\mathbb E_1$-cotangent complex of $A$. 

The double-loop space of $\mathcal M$ at $A$ is equivalent to the following spaces:
\begin{enumerate}
    \item $\Omega(\map_{\Alg(\C)}(A,A), \id_A)$; 
    \item the fiber over $\id_A$ of $\map_{A\otimes A\op}(A,A)\to \map_A(A,A)$, or equivalently if $\C$ is additive, its fiber over $0$; 
    \item $\map_{A\otimes A\op}(\Sigma L_A,A)$ if $\C$ is stable.
\end{enumerate}
\end{prop}
\begin{proof}
Without loss of generality, we assume $\C$ is presentably symmetric monoidal. 

Recall that $\mathcal M = \Alg(\C)\times_{\Alg(\ho(\C))}\{hA\}$ by definition. An equivalent description is the fiber sequence $\mathcal M\to \Alg(\C)^\simeq \to \Alg(\ho(\C))^\simeq$ at the point $hA\in \Alg(\ho(\C))^\simeq$. 

Looping once at $A$, we find the fiber sequence $\Omega \mathcal M\to \Omega(\Alg(\C)^\simeq,A)\to \Omega(\Alg(\ho(\C))^\simeq,hA)$, and thus $$\Omega \mathcal M\to \Aut_{\Alg(\C)}(A)\to \Aut_{\Alg(\ho(\C))}(hA)$$
As an endomorphism of $A$ which is the identity in $\ho(\C)$ must be an equivalence, this also yields a fiber sequence $\Omega \mathcal M\to \map_{\Alg(\C)}(A,A)\to \hom_{\Alg(\ho(\C))}(hA,hA)$ at $\id_{hA}$. As the latter is discrete, we deduce point 1. 

Next, we use the fully faithful embedding $\Alg(\C)\to (\Mod_\C(\PrL))_{\C/}, A\mapsto (\Mod_A,A)$, cf. \cite[Theorem 4.8.5.11]{HA} to describe $\map_{\Alg(\C)}(A,A)$ as the fiber of $\map_{\Mod_\C}(\Mod_A,\Mod_A)\to \map_{\Mod_\C}(\C,\Mod_A)$ over the canonical functor $\C\to \Mod_A$ classified by $A\in\Mod_A$. Using \cite[Theorem 4.8.4.1]{HA}, we rewrite this map as the functor $\BiMod_A^\simeq \to \Mod_A^\simeq$ that forgets the right $A$-module structure. Taking loops at the identity of $A\in\Alg(\C)$ yields the fiber sequence $$\Omega(\map_{\Alg(\C)}(A,A),\id_A)\to \Omega(\BiMod_A^\simeq, A)\to \Omega(\Mod_A^\simeq, A)$$ We rewrite the latter two terms as $\Aut_{A\otimes A\op}(A)\to \Aut_A(A)$ and use again the fact that any $A\otimes A\op$-linear endomorphism of $A$ which is an underlying equivalence is an equivalence to rewrite the fiber of this map as the fiber of $\map_{A\otimes A\op}(A,A)\to\map_A(A,A)$ over the identity. This proves the first half of 2., and using the additivity of $\C$ and the fact that this is a map of grouplike $\mathbb E_\infty$-monoids which has both $\id_A$ and $0$ in its image, we deduce that the fiber is the same over $0$, which is the second half of 2. 

Finally, 3. follows from the second half of 2.: we rewrite $\map_A(A,A)$ as  $\map_{A\otimes A\op}(A\otimes A\op,A)$ and then the restriction map $$\map_{A\otimes A\op}(A,A)\to \map_A(A,A)\simeq \map_{A\otimes A\op}(A\otimes A\op,A)$$ is identified with precomposition by the multiplication map $\mu: A\otimes A\op\to A$ (indeed, the multiplication is just the co-unit of the adjunction that forgets the right $A$-module structure), so that by 2., our double-loop space is identified with $\map_{A\otimes A\op}(\mathrm{cofib}(\mu),A)$. To conclude, we use that $\mathrm{cofib}(\mu)\simeq \Sigma L_A$  \cite[Theorem 7.3.5.1.]{HA}.
\end{proof}
\begin{rmk}\label{rmk : modulinotcontr}
$\map_{A\otimes A\op}(X,A)$ can also be described as $\map(\one, \hom_{A\otimes A\op}(X,A))$. So unless $\map(\one, Z(A)) \to \map(\one, A)$ is an inclusion of components, $\Omega^2(\mathcal M,A)$ is not contractible, and therefore $\mathcal M$ isn't either.  Here, $Z(A)$ is the $\mathbb E_1$-center of $A$, from \Cref{nota:center}. 
\pend \end{rmk}
\begin{ex}\label{ex:modulinotcontr}
Consider the commutative differential graded $\mathbb Q$-algebra $R= \mathbb Q[t]$, where $t$ is in degree $2$ as a commutative ring spectrum and let $\C= \Mod_R$. For any $n\geq 1$, the matrix ring $M_n(R) = \End_R(R^n)$ is separable, and its center is $R$ itself. Therefore $\map_{A\otimes A\op}(A,A)\to \map_A(A,A)$ is the map $\Omega^\infty(\mathbb Q[t]\to M_n(\mathbb Q[t]))$ and this is clearly not an inclusion of components as long as $n\geq 2$ (it is not surjective in any $\pi_{2k}, k\geq 1$). By \Cref{rmk : modulinotcontr}, the moduli space of lifts of $M_n(\mathbb Q[t])$ is not contractible. 
\pend \end{ex}

\begin{rmk}
In \cite{kleintilson}, the authors also study a certain moduli space of algebra structures. Note, however, that this is a different moduli space in that it is the moduli space of algebra structures on an \emph{object} of $\C$, while ours is the moduli space of algebra structures extending a homotopy algebra. The answers we get for the double loop space are thus different, even if there is some similarity in that both involve some version of Hochschild cohomology.  
\pend \end{rmk}

We apply these results in the case of ring spectra. For a commutative ring spectrum $R$, let $\Projsep(R)$ denote the full subgroupoid of $\Alg(\Mod_R(\Sp))$ spanned by separable algebras whose underlying $R$-module is finitely generated projective. This is clearly functorial along basechange. 

The following is our version of \cite[Theorem 6.1]{BRS} (cf. also \cite[Proposition 3.12, Theorem 3.15]{GL}):
\begin{prop}
Let $R$ be a commutative ring spectrum. In the span $$\Projsep(R)\leftarrow \Projsep(R_{\geq 0})\to \Projsep(\pi_0(R))$$
the left leg is an equivalence, and the right leg is essentially surjective, with simply-connected fibers.

In particular, any separable algebra over $\pi_0(R)$ can be (weakly uniquely) realized as $\pi_0$ of a separable algebra over $R$. 
\end{prop}
\begin{proof}
For any commutative ring spectrum $R$, $\Projsep(R)$ can equivalently be described as the space of separable algebras in $\Proj(R)$, the additive category of projective $R$-modules. In particular, this only depends on this additive category, and it is a classical fact that $\Proj(R_{\geq 0})\to \Proj(R)$ is an equivalence. This proves the statement about the left leg. 

For the right leg, we observe that $\pi_0: \Proj(R_{\geq 0})\to \Proj(\pi_0(R))$ witnesses the latter as the homotopy category of the former, and thus, passing to algebras, $\Alg(\Proj(R_{\geq 0}))\to \Alg(\Proj(\pi_0(R))$ is equivalently $\Alg(\Proj(R_{\geq 0}))\to \Alg(\ho(\Proj(R)) $. The statement thus follows from \Cref{thm : e1lift}. 
\end{proof}
\begin{rmk}
Note that if $R$ is a discrete commutative ring, any central separable algebra over $R$ is necessarily finitely generated projective \cite[Theorem 2.1]{AuslanderGoldman}. In particular, this allows us to lift all central separable algebras. 
\pend \end{rmk}
\begin{rmk}
In \cite[Theorem 6.1]{BRS}, the lift along the right leg is said to be ``unique''. \Cref{rmk : modulinotcontr} and \Cref{ex:modulinotcontr} show that this unicity is to be taken with a grain of salt. 
\pend \end{rmk}

The results of this section suggest the slogan that ``Separable algebras and their modules are controlled by the homotopy category''. From the perspective of homotopy theory, this justifies to some extent the study of separable algebras and their modules in tensor triangulated categories, but also suggests that many results in the unstructured setting can be lifted for free to a more structured or coherent setting. Because of the non-unicity pointed out in \Cref{rmk : modulinotcontr}, we see that not \emph{everything} can be lifted for free. 

In the next section we will see that the situation in the commutative world is much better. The moduli spaces become contractible (even the $\mathbb E_1$ ones!), and the mapping spaces become discrete. 
\section{Commutative separable algebras}\label{section:comm}
In this section, we study commutative separable algebras. Unsurprisingly, this situation is much better behaved than in the associative case. In the commutative case, we will see that the obstructions to contractibility from \Cref{section:homotopycat} vanish - in fact, even in the homotopy commutative case. The key difference with the general case is that now, the multiplication map $\mu:A\otimes A\to A$ is an (homotopy) algebra map. 

We begin this section with a study of certain moduli spaces, and of mapping spaces from commutative separable algebras, and we then apply Lurie's deformation theory from \cite[Section 7.4]{HA} to compare étale algebras and separable commutative algebras. 

\begin{defn}
    A commutative separable algebra is a commutative algebra whose underlying algebra is separable.
\pend \end{defn}
We begin with a general proposition:
\begin{prop}\label{prop : retractloc}
Suppose $\C$ is additively symmetric monoidal and idempotent-complete.

Let $A,B\in\CAlg(\C)$ be commutative algebras in $\C$, and $f: A\to B$ a morphism of commutative algebras. If $f$ admits an $A$-module splitting, then $B$ is the localization of $A$ at an idempotent $e$. 

In particular $\Mod_A$ splits, symmetric monoidally, as a product $\Mod_A\simeq \Mod_B\times\Mod_B^\bot$. 

\end{prop}
This follows from a more general, monoidal version, which we needed in some earlier versions of the proofs of this section. We end up only needing the symmetric monoidal version, but we digress for a moment for the convenience of the reader, and to record this simple fact, which is
simply a variant of the discussion of idempotent algebras in \cite[Section 4.8.2.]{HA} in the setting of \emph{non-symmetric} monoidal categories. Parts of it work exactly the same as in the monoidal setting, but some of it does not, so we simply record it here.

First, recall the definition:
\begin{defn}\label{defn:idem}
    Let $\M$ be an $\mathbb E_k$-monoidal \category{} with unit $\one$ and $k\geq 1$. An $\mathbb E_0$-object $e$ therein, that is, an object in $\M_{\one/}$, is called idempotent, if both $e\otimes (\one \to e)$ and $(\one\to e)\otimes e$ are equivalences. 

    For $0\leq d\leq k$, an $\mathbb E_d$-algebra $A$ in $\M$ is called idempotent if its underlying $\mathbb E_0$-algebra is. 
\pend \end{defn}
\begin{rmk}
    In the case of an $\mathbb E_k$-monoidal \category, $k\geq 2$, these two maps are homotopic and so it suffices to require one of them to be an equivalence. 
\pend \end{rmk}
The main result is:
\begin{prop}\label{prop:idemalg}
Let $\M$ be an $\mathbb E_k$-monoidal \category{} with $1\leq k\leq \infty$, and let $0\leq d\leq k$. The forgetful functor from $\mathbb E_d$-algebras to $\mathbb E_0$-algebras restricts to an equivalence between the respective full subcategories of idempotent algebras: $$\Alg_{\mathbb E_d}(\M)^{\idem}\xrightarrow{\simeq}\Alg_{\mathbb E_0}(\M)^{\idem}$$
\end{prop}
We begin with an easy lemma:
\begin{lm}\label{lm:subcatidempotent}
    Let $\M$ be an $\mathbb E_k$-monoidal \category{}, and $e\in \M_{\one/}$ an idempotent $\mathbb E_0$-object. The full subcategory $\M_e$ of $\M$  spanned by those $m$'s for which both $m\otimes (\one\to e)$ and $(\one\to e)\otimes m$ are equivalences determines a full sub-$\mathbb E_k$-operad of $\M$, which is itself an $\mathbb E_k$-monoidal \category.
\end{lm}
\begin{warn}
For $k\geq 2$, the inclusion $\M_e\to \M$ actually admits an $\mathbb E_k$-monoidal left adjoint, given by tensoring with $e$. For $k=1$, this left adjoint is only oplax monoidal, as the canonical map $e\otimes x\otimes y\otimes e\to e\otimes x\otimes e\otimes y\otimes e$ need not be an equivalence. This is the key difference between $k=1$ and higher $k$'s.
\pend \end{warn}
\begin{proof}[Proof sketch]
    Firstly, $\M_e$ is stable under the formation non-empty tensor products, so we only need to prove that it admits a unit, which we claim is $e$. 

    In other words, we need to show that for $m\in \M_e$, the restriction map $\map(e,m)\to \map(\one,m)$ is an equivalence. The inverse is given by $\map(\one,m)\to \map(e,e\otimes m)\to \map(e,m)$ - it is a diagram chase to check that this is indeed an inverse. 
\end{proof}
\begin{proof}[Proof of \Cref{prop:idemalg}]
    Forgetting down to its underlying $\mathbb E_d$-monoidal category, we may assume without loss of generality that $d=k$. 

    We first prove essential surjectivity: fix an idempotent $\mathbb E_0$-algebra $e$. By \Cref{lm:subcatidempotent}, there is a lax $\mathbb E_k$-monoidal inclusion $\M_e\to \M$, which therefore sends $\mathbb E_k$-algebras to $\mathbb E_k$-algebras, and $e$ is the unit in the source, so it has an $\mathbb E_k$-algebra structure in the target as well, which proves essential surjectivity. 

    For fully faithfulness, fix two idempotent $\mathbb E_k$-algebras $e,e'$. We aim to prove that in both \categories, the mapping space from $e$ to $e'$ is empty or contractible, in the same case. Clearly if the mapping space in $\mathbb E_0$-algebras is empty, the same holds in $\mathbb E_k$-algebras, so we may assume it's non-empty, and in this case, we need to prove that both are contractible. 

    So suppose there exists a factorization $\one\to e\to e'$. We claim that in this case, $e'$ is in $\M_e$. By fully faithfulness of the inclusion $\M_e^\otimes\to \M^\otimes$, this will then prove the claim, as $e$ is the unit in $\M_e$. 

But now, note that the map $e'\to e\otimes e'$ can be composed with $e\otimes e'\to e'\otimes e'\simeq e'$,  so that $e'$ is a retract of $e\otimes e'$, which is an object $m$ for which the map $e\to e\otimes m$ is an equivalence. It follows that $e'$ is also such an object. Similarly for the map $e'\to e'\otimes e$. 
\end{proof}
We can now state the desired result: 

\begin{lm}\label{lm : unitsum}
Let $\cat D$ be a semiadditively $\mathbb E_k$-monoidal category with unit $\mathbf 1$, where $1\leq k\leq \infty$. Suppose $\mathbf 1$ splits as $a\oplus b$. 

Then there are essentially unique $\mathbb E_k$-algebra structures on $a,b$ in $\cat D$ for which the unit maps $\mathbf 1\to a, b$ are the projections coming from this decomposition. In particular, $\mathbf 1\to a\times b$ is an equivalence of algebras. 

More precisely, $a,b$ are idempotent algebras in $\D$ in the sense of \Cref{defn:idem}, and therefore have unique algebra structures extending their unit map by \Cref{prop:idemalg}. 
\end{lm}
\begin{proof}
We show that the projections witness $a,b$ as idempotent $\mathbb E_0$-algebras. 

For this, observe that $a\otimes a \oplus a \otimes b \oplus b\otimes a \oplus b\otimes b \simeq (a\oplus b)\otimes (a\oplus b)\simeq \mathbf 1\otimes\mathbf 1\simeq \mathbf 1 \simeq a\oplus b$. 

Second, observe that the morphism $a\otimes b \to \mathbf 1\otimes \mathbf 1\simeq \mathbf 1$ factors as $a\otimes b\to a\otimes \mathbf 1 \simeq a\to \mathbf 1$, but also as $a\otimes b\to \mathbf 1\otimes b \simeq b\to \mathbf 1$. In particular, $a\otimes b\to \mathbf 1$ factors through $0$, but it has a retraction, so $a\otimes b$ must be $0$. 

Similarly, $b\otimes a \simeq 0$. It is then just a matter of diagram chasing to see that $a,b$ are idempotents. (Note that this diagram chases can be made in $\mathrm{Ho}(\cat D)$, as $\cat D\to \mathrm{Ho}(\cat D)$ is conservative, monoidal, and biproduct preserving). 
\end{proof}
\begin{proof}[Proof of \Cref{prop : retractloc}]
We start by assuming $\C$ has geometric realizations that commute with the tensor product in each variable. Thus we can make $\Mod_A(\C)$ into a symmetric monoidal \category{} with the relative tensor product. 

We can now apply \Cref{lm : unitsum} to the category $\Mod_A(\C)$: $A$ is the unit, and it splits as $B\oplus C$ for some $C$, as $\C$ is additive and idempotent-complete, where the projection $A\to B$ is chosen to be $f$. \Cref{lm : unitsum} in the case $k=\infty$ tells us exactly that $C$ admits a unique commutative algebra structure in $\Mod_A(\C)$ extending its unit $A\to C$, and then $A\simeq B\times C$ as algebras, which is exactly saying that $B$ is the localization of $A$ at an idempotent. 

As $\C$ is additive compatibly with the tensor product, it follows that $\Mod_A(\C)\to \Mod_B(\C)\times\Mod_C(\C)$ is an equivalence, and under this identification we clearly have $\{0\}\times \Mod_C(\C) = \Mod_B(\C)^\bot$, where for a subcategory $\E$, $\E^\bot :=\{f\in\Mod_A(\C)\mid \forall e\in\E, \map(e,f) \simeq \pt\simeq \map(f,e)\}$. 

To deduce the statement for general $\C$, we note that if $\C\to \D$ is an additive, symmetric monoidal embedding where $\D$ has geometric realizations compatible with the tensor product, then because $\C$ was assumed idempotent complete, the decomposition $M\simeq B\otimes_A M\oplus C\otimes_A M$ for any $M\in\Mod_A(\C)\subset \Mod_A(\D)$ shows that $B\otimes_A M$ (resp. $C\otimes_A M$) is in fact in $\Mod_B(\C)\subset \Mod_B(\D)$ (resp. $\Mod_C(\C)\subset \Mod_\C(\D)$), which concludes the proof. 
\end{proof}
The following is the key consequence we were aiming for - at the triangulated level, it is already present in the work of Balmer \cite[Theorem 2.1]{balmerdegree} and is integral to his theory of degrees of separable algebras; see also \cite[Corollary 5.2]{NikoLuca} for a treatment at the level of stable \categories. 
\begin{cor}\label{cor : commsepsplit}
   Suppose $\C$ is additively symmetric monoidal. If $A$ is a separable commutative algebra in $\C$, its multiplication map $\mu: A\otimes A\to A$ is a localization at an idempotent, and this is the case more generally for the the iterated multiplication maps $\mu_n : A^{\otimes n}\to A$; in particular for every $n$, there is a unique commutative algebra $C_n$ under $A^{\otimes n}$ with an equivalence $A^{\otimes n}\xrightarrow{(\mu_n,p)}A\times C_n$. 
\end{cor}
\begin{rmk}\label{rmk:diagprop}
    Specializing to $n=2$ we see that in this case, $\mu^*: \Mod_A\to \Mod_{A\otimes A}$ is fully faithful, with essential image the ``diagonal'' bimodules. Separability implies in particular that ``being a diagonal bimodule'' really is a property of the bimodule $M$, and not the extra structure of an equivalence $M\simeq \mu^*N$.
\pend \end{rmk}
\begin{cor}\label{cor : commvanishingcotangent}
    Suppose $\C$ is additively symmetric monoidal. Let $A\in\CAlg(\C)$ be a commutative separable algebra, and let $L_A$ denote its $\mathbb E_1$-cotangent complex. In this case, the mapping space $\map_{A\otimes A\op}(L_A,A)$ is trivial, in fact even the corresponding hom-object in $\C$ is $0$. 
\end{cor}
\begin{proof}
We assume without loss of generality that $\C$ admits geometric realizations compatible with the tensor product, in particular it admits basechange along algebra maps. 

    As $A$ is commutative, we may consider $\mu: A\otimes A\to A$ as a commutative algebra map, and under the identification $A\op\simeq A$, this corresponds to the $A$-bimodule multiplication map $A\otimes A\op\to A$. Let $L_A'$ denote the $A\otimes A$-module fiber of $\mu$. Under the same identification, this corresponds to $L_A$. 

    In particular, as an $A\otimes A$-module, $A$ can be described as $\mu^*A$, so that $\hom_{A\otimes A}(L_A', \mu^*A) \simeq \hom_A(\mu_! L_A,A) $. 

    By \Cref{prop : retractloc}, $\mu: A\otimes A\to A$ is a localization at an idempotent, so that $\mu_!$ of the fiber is trivial, and the claim follows. 
\end{proof}
\begin{rmk}
    This corollary is the crucial difference between the commutative and the associative case. In \Cref{rmk : modulinotcontr}, we saw that it was precisely the nontriviality of $\map_{A\otimes A\op}(L_A,A)$ that makes the moduli space of $\mathbb E_1$-algebra structures non-trivial.
\pend \end{rmk}
As a corollary, we find that in the presence of homotopy commutativity, the obstruction theory for $\mathbb E_1$-structures simplifies greatly. We have: 
\begin{prop}\label{prop : e1comm}
 Suppose $\C$ is additively symmetric monoidal. Let $A\in \CAlg(\ho(\C))$ be a homotopy commutative, homotopy separable homotopy algebra. The moduli space of lifts of $A$ to an $\mathbb E_1$-algebra in $\C$ is contractible. 
\end{prop}
\begin{proof}
We assume without loss of generality that $\C$ admits internal hom's.  

\Cref{thm : e1lift} proves that this moduli space is simply-connected, so it suffices to prove that its double loopspace at any given point is contractible. So we fix an $\mathbb E_1$-algebra $\tilde A$ extending $A\in\Alg(\ho(\C))$.

By point 3.\ in \Cref{prop:2loopmoduli}, it suffices to prove that $\map_{\tilde A\otimes \tilde A\op}(\Sigma L_{\tilde A},\tilde A)$ is contractible, or better, it suffices to prove that the hom object $\hom_{\tilde A\otimes \tilde A\op}(\Sigma L_{\tilde A},\tilde A)$ is zero. The algebra $\tilde A\otimes \tilde A\op$ is separable, so by \Cref{prop : homsep}, this hom object can be computed in $\ho(\C)$. 

It now follows from \Cref{cor : commvanishingcotangent} that it is $0$ - note that as $\tilde A\otimes\tilde A\op\to \tilde A$ is split, its fiber can be computed in $\C$ or in $\ho(\C)$ equivalently.   
\end{proof}
We now explore other consequences of this orthogonality, namely the uniqueness of the separability idempotent of a commutative separable algebra, and we deduce from it nice descent properties of separable algebras. This uniqueness is well-known classically, and Naumann and Pol have also isolated it, as well as the resulting descent properties, in their recent work, cf. \cite[Lemma 6.2, Proposition 6.3]{NikoLuca}. 
\begin{cor}\label{cor:uniqueidem}
Assume $\C$ is additively symmetric monoidal. Let $A\in\CAlg(\C)$ be a commutative separable algebra. The space of separability idempotents for $A$ is contractible. More precisely, we define this space is the fiber of $\map_{A\otimes A\op}(A,A\otimes A\op) \to \map_{A\otimes A\op}(A,A)$ over $\id_A$. 
\end{cor}
\begin{proof}
This is a map of grouplike $\mathbb E_\infty$-monoids, and $\id_A$ is in the image by separability, so the fiber is the same as the fiber over $0$, so it is equivalent to $\map_{A\otimes A\op}(A, L_A)$.

This is contractible for the same reasons as before, cf. \Cref{prop : retractloc} and \Cref{cor : commsepsplit}. 
\end{proof}
We observe that this uniqueness allows us to recover results of Sanders \cite[Corollary 2.12]{sandersetale}, and specifically, to make it homotopy coherent. To state it, we give the following definition:
\begin{defn}
    A separable algebra $A$ in $\C$ is said to be strongly separable if it admits a separability idempotent $s: \one\to A\otimes A$ such that $\tau\circ s\simeq s$ in $\C$, where $\tau: A\otimes A\to A\otimes A$ is the swap map. 

    It is said to be coherently strongly separable if $s$ can be made $C_2$-equivariant with respect to the swap $C_2$-action on $A\otimes A$, and the trivial action on $\one$.
\pend \end{defn}
\begin{cor}
Assume $\C$ is additively symmetric monoidal. Let $A$ be a commutative separable algebra. In this case, $A$ is coherently strongly separable.
\end{cor}
\begin{proof}
Without loss of generality, we assume $\C$ admits geometric realizations compatible with the tensor product. 

The multiplication map $\mu: A\otimes A\to A$ is $C_2$-equivariant, hence so is the adjunction $\mu_!\dashv \mu^*$ between $\Mod_{A\otimes A}(\C)$ and $\Mod_A(\C)$, so that map $\map_{A\otimes A}(A,A\otimes A)\to \map_{A\otimes A}(A,A)$ can be made $C_2$-equivariant - it is given by a co-unit. Similarly, the identity $\id_A\in\map_{A\otimes A}(A,A)$ is a $C_2$-fixed point. 

It follows that the fiber of the first map over $\{\id_A\}$ can be given a compatible $C_2$-action. But it is contractible, by the previous corollary, and hence admits a $C_2$-fixed point. It follows that $\map_{A\otimes A}(A,A\otimes A)$ does too, and hence, by forgetting to $\map(A,A\otimes A)$, we see that it also has a $C_2$-fixed point whose underlying point is a separability idempotent. We can precompose the $C_2$-fixed point $A\to A\otimes A$ by $\one\to A$ to get a separability idempotent as in the above definition.
\end{proof}
We also note that in many cases, strongly separable implies coherently strongly separable:
\begin{prop}\label{prop:cohstrongsep}
    Let $A$ be a separable algebra in $\C$ whose underlying object in $\C$ is dualizable\footnote{This specific proposition does not require $\C$ to be additive, as is clear from the proof.}. If $A$ is strongly separable, then it is coherently strongly separable. 
\end{prop}
\begin{proof}
    By \cite[Proposition 2.30]{sandersetale}, if $A$ is dualizable and strongly separable, then a separability idempotent is given by the dual of the trace pairing $A\otimes A\to A\xrightarrow{t}\one$. 

    Dualizing preserves $C_2$-equivariance, so it suffices to observe that the trace pairing $A\otimes A\to \one$ is $C_2$-equivariant. This is the case by \Cref{prop:C2trace}. 
\end{proof}
\begin{ques}
    In general, is a strongly separable algebra necessarily coherently strongly separable ? 
\end{ques}
Another corollary of uniqueness of separability idempotents is the fact that for commutative algebras, separability can be checked locally. We first make the following definition:
\begin{defn}
We let $\CSep(\C)\subset \CAlg(\C)^\simeq$ denote the subspace spanned by separable algebras. 
\pend \end{defn}
The statement of locality can then be phrased as follows:
\begin{cor}\label{cor:descentcommsep}
The functor $\C\mapsto \CSep(\C)$, defined on additively symmetric monoidal \categories, is limit-preserving. 
\end{cor}
\begin{rmk}
    The corresponding statement for separable $\mathbb E_1$-algebras is wrong. We will give a counterexample involving Azumaya algebras and based on \cite{GL} in \Cref{ex:GL}.
\pend \end{rmk}
\begin{proof}
The proof is similar to the corresponding claim for dualizable objects, cf. \cite[Proposition 4.6.11.]{HA}. 

Consider the space of ``commutative algebras equipped with a separability idempotent'', namely the space of tuples $(A, s: A\to A\otimes A\op, h)$ where $s$ is a map of bimodules $A\to A\otimes A\op$, and $h$ a homotopy witnessing that $\mu\circ s\simeq \id_A$ in $A$-bimodules. 

The functor that assigns this space to $\C$ is clearly limit preserving in $\C$ (it can be written as a limit of spaces that are limit-preserving functors of $\C$), and the projection down to $\CAlg(\C)^\simeq$, which is natural in $\C$, establishes, by \Cref{cor:uniqueidem}, an equivalence with $\CSep(\C)$. The claim thus follows. 
\end{proof}
\begin{rmk}
If one thinks of ``descent''-type statements as statements about recovering a (symmetric monoidal) category as a limit of other (symmetric monoidal) categories, this result can be interpreted as saying that commutative separable algebras satisfy descent. 

For instance, if $\one\to A$ is a universal descent morphism in the sense of \cite[Definition D.3.1.1]{SAG} (e.g. an étale cover in $\CAlg(\Sp)$), one sees that an algebra $R\in\CAlg(\C)$ is separable if and only if $A\otimes R\in \CAlg(\Mod_A)$ is separable: one can check separability (of \emph{commutative} algebras) after passing to a (universal descent) cover. 
\pend \end{rmk}

In the above proof, we have used implicitly the following lemma, which we record explicitly:
\begin{lm}\label{lm:limpres}
    The functor $\CAlg: \CAlg(\Cat)\to \Cat$ preserves limits. This is more generally true for the functor $\Alg_\Oo$, for any \operad{} $\Oo$.
    
    Furthermore, given a limit diagram $\C_\bullet: I^\triangleleft \to \CAlg(\Cat)$, and an algebra object $A\in \Alg(\C_\infty)$\footnote{We use ``$\infty$'' to denote the cone point in $I^\triangleleft$.}, the canonical map $\LMod_A(\C_\infty)\to \lim_I \LMod_{A_i}(\C_i)$ is an equivalence, where $A_i$ is the image of $A$ under the induced functor $\Alg(\C_\infty)\to\Alg(\C_i)$. 
\end{lm}
\begin{proof}
    The first part follows from the existence of envelopes, see \cite[Proposition 2.4.9]{HA}. In more detail, for any \operad{} $\Oo$, there is a symmetric monoidal \category{} $\mathrm{Env}(\Oo)$ with an $\Oo$-algebra $U_\Oo\in\Alg_\Oo(\mathrm{Env}(\Oo))$ such that evaluation at $U_\Oo$ induces an equivalence $$\Fun^\otimes(\mathrm{Env}(\Oo),\C)\to \Alg_\Oo(\C),$$ natural in $\C$. Since the source of this equivalence clearly preserves limits in $\C$, the claim follows. 

    The second part is a corollary of the first: let $\mathcal{LM}$ denote the \operad{} that classifies left modules \cite[Section 4.2.1]{HA}, and $\mathrm{Ass}$ the associative operard, with its canonical inclusion $\mathrm{Ass}\to \mathcal{LM}$ which induces the canonical forgetful functor $\Alg_{\mathcal LM}\to \Alg$. 
    
    We can then write $\LMod_A(\C) = \Alg_{\mathcal{LM}}(\C)\times_{\Alg(\C)}\{A\}$. By the first part of the statement, it follows that $(\C,A)\mapsto \LMod_A(\C)$ is a pullback of limit-preserving functors of the pair $(\C,A)$, and is thus itself a limit-preserving functor. 
\end{proof}
\begin{cor}
    The functor $\CAlg^{sep}(-)$, that assigns to an additively symmetric monoidal \category{} $\C$ the full subcategory of $\CAlg(\C)$ spanned by separable algebras, is limit-preserving.
\end{cor}
\begin{proof}
    Generally, if $f,g: S\to \Cat_\infty$ are functors, $g$ preserves limits and $i:f\to g$ is a pointwise fully faithful natural transformation, then $f$ preserves limits if and only if $f^\simeq: E\to \Ss$ does. ``Only if'' is clear as $(-)^\simeq : \Cat_\infty\to \Ss$ preserves limits. 

    To prove ``if'', we note that limits of fully faithful functors are fully faithful. It follows that for any diagram $X:I\to E$, in the following commutative square 
    \[\begin{tikzcd}
	{f(\lim_I X)} & {\lim_If(X)} \\
	{g(\lim_IX)} & {\lim_Ig(X)}
	\arrow[from=2-1, to=2-2]
	\arrow["i", from=1-1, to=2-1]
	\arrow["i", from=1-2, to=2-2]
	\arrow[from=1-1, to=1-2]
\end{tikzcd}\]
the vertical arrows and the bottom horizontal arrow are all fully faithful. Therefore, so is the top horizontal arrow. Thus, to prove that it is an equivalence, we simply need to check that it is essentially surjective, but this follows from $f^\simeq$ preserving limits. 

We apply this to $f= \CAlg^{sep}, g= \CAlg(-)$: we observed above that $g$ and $f^\simeq = \CSep(-)$ preserved limits. 
\end{proof}

We now move on to the first main theorem of this section, which concerns highly coherent commutative structures on separable algebras. 
\begin{thm}\label{thm : commlift}
Let $\C$ be an additively symmetric monoidal \category, and $A\in\Alg(\C)$ a separable algebra. 

If $A$ is homotopy commutative, then it has an essentially unique $\mathbb E_\infty$-structure extending its given $\mathbb E_1$-structure. 

More generally, for any $n\geq 1$, $A$ has an essentially unique $\mathbb E_n$-structure extending its given $\mathbb E_1$-structure. 
\end{thm}
\begin{rmk}
    We note that this is an obvious commutative analogue of \Cref{thm : e1lift}, but that, as with \Cref{prop : e1comm}, the situation is better in the commutative world.
\pend \end{rmk}
We also note the important corollary that all the previous work in the section, about commutative separable algebras therefore also applies in the case of homotopy commutative separable algebras. 
Combined with \Cref{prop : e1comm}, this yields:
\begin{cor}
    Let $\C$ be an additively symmetric monoidal \category, and $A\in\Alg(\ho(\C))$ a homotopy separable  homotopy commutative homotopy algebra. 

It has an essentially unique $\mathbb E_\infty$-structure extending its given homotopy algebra structure.
\end{cor}
\begin{rmk}
    This theorem is consistent with the experience that all commutative separable algebras in tensor triangulated categories coming from stably symmetric monoidal \categories{} admit highly coherent structures.  
\pend \end{rmk}
\begin{rmk}
  This corollary should be reminiscent of the Goerss--Hopkins--Miller theorem \cite{Goerss-Hopkins}. However, Morava $E$-theory is \emph{not} separable. In \Cref{subsection:Morava}, we introduce the notion of an \emph{ind-separable} algebra to make up for this defect, and observe that Morava $E$-theories are examples of such things. We deduce extensions of the Goerss--Hopkins--Miller theorem to other \operads{} than $\mathbb E_1$ and $\mathbb E_\infty$ (cf. \Cref{thm : opdlift} below and \Cref{cor:GHM}) - these are well-known to experts but do not seem to be recorded in the literature. 
\pend \end{rmk}
In fact, we deduce \Cref{thm : commlift} from a more general statement. To state it, we introduce a certain class of \operads{} which contains the $\mathbb E_n, 1\leq n \leq \infty$. 
\begin{nota}
    Let $\Oo$ be an \operad, with a single color $x$, i.e. $\Oo^\otimes_{\langle 1\rangle}$ has a single object up to equivalence. By definition of an \operad, it follows that $\Oo^\otimes_{\langle n\rangle}$ has a unique object up to equivalence too, denoted $x\oplus ... \oplus x$ (see \cite[Remark 2.1.1.15]{HA} for the notation). In this case, we let $\Oo(n)$ denote the space of $n$-ary operations. In more detail, letting $\mu_n: \langle n\rangle \to \langle 1\rangle$ denote the unique active morphism in $\Fin_*$, we put: $$\Oo(n):=\map_{\Oo^\otimes}(x\oplus ... \oplus x, x)\times_{\hom_{\Fin_*}(\langle n\rangle,\langle 1\rangle)}\{\mu_n\}$$. 
\pend \end{nota}
\begin{defn}
    Let $\Oo$ be an \operad. We say it is \emph{weakly reduced} if:
    \begin{itemize}
        \item It has a single color, i.e. its underlying \category has a unique object $x$ up to equivalence. 
        \item Both $\Oo(0)$ and $\Oo(1)$ are connected. 
    \end{itemize}
\pend \end{defn}
\begin{ex}\label{ex:weaklyred}
    The $\mathbb E_n$-operads, $1\leq n\leq \infty$ are weakly reduced \operads.
\pend \end{ex}
Our more general statement can thus be stated as:
\begin{thm}\label{thm : opdlift}
Let $\C$ be an additively symmetric monoidal \category, and $A\in\Alg(\C)$ a separable algebra. 

For any weakly reduced \operad{} $\Oo$, the space of $\Oo\otimes\mathbb E_1$-structures on $A$ extending the given $\mathbb E_1$-structures is contractible. Here, $\otimes$ denotes the Boardman-Vogt tensor product of \operads{} following \cite[section 2.2.5]{HA}. 
\end{thm}

Let us briefly describe the strategy of proof of \Cref{thm : opdlift}, so that we can also explain the hypotheses on $\Oo$. We will expand on \Cref{thm : morsep}, by proving that under the homotopy commutativity assumption, $\map_{\Alg(\C)}(A,R)$ has lots of discrete components, in fact, enough to guarantee that each $\map_{\Alg(\C)}(A^{\otimes n},A)$ is discrete and equivalent to $\hom_{\Alg(\ho(\C))}(hA^{\otimes n},hA)$.

From this, it follows at once that $\Oo$-algebra structures on $A$ in $\Alg_{\mathbb E_1}(\C)$, i.e. $\Oo\otimes\mathbb E_1$-algebra structures on $A$ extending the given algebra structure, are equivalent to $\Oo$-algebra structures on $hA$ in $\Alg(\ho(\C))$. As $hA$ is commutative and $\Alg(\ho(\C))$ is a $1$-category, the assumptions on $\Oo$ will then guarantee that there is a unique such structure. 

We thus begin with:
\begin{prop}\label{prop:discretecentralhocomm}
    Assume $\C$ is additively symmetric monoidal.
    Let $A\in\Alg(\C)$ be separable and homotopy commutative, and let $R\in\Alg(\C)$ arbitrary. Let $f: A\to R$ be a map in $\Alg(\C)$, and suppose that it is homotopy-central, i.e. the following two maps are equivalent in $\C$: $A\otimes R\xrightarrow{f\otimes\id} R\otimes R\to R$ and $A\otimes R\simeq R\otimes A\xrightarrow{\id\otimes f} R\otimes R\to R$. 

    In this situation, $\Omega(\map_{\Alg(\C)}(A,R),f)$ is contractible, i.e. the component of $f$ in $\map_{\Alg(\C)}(A,R)$ is contractible. 
    \end{prop}
\begin{proof}
 Recall from \cite[Theorems 4.8.4.1 and 4.8.5.11]{HA} that $\map_{\Alg(\C)}(A,R)$ is equivalent to the fiber over $R$ of the forgetful map $_R\BiMod_A^\simeq\to \LMod_R^\simeq$. 

It follows that $\Omega(\map_{\Alg(\C)}(A,R),f)$ is equivalent to the fiber of $\Aut_{R\otimes A\op}(R,R)\to \Aut_{R}(R,R)$ over $\id_R$, where $R$ has the $R\otimes A\op$-module structure induced by $f$. As the forgetful functor $_R\BiMod_A\to \LMod_R$ is conservative, this is equivalently the fiber of the corresponding mapping spaces, again at $\id_R$. Because $\id_R$ is in the image and this map is a map of grouplike $\mathbb E_\infty$-spaces, the fiber over $\id_R$ is equivalent to the fiber over $0$. 

In other words, it suffices to prove that for every $n$, $\pi_n(\map_{R\otimes A\op}(R,R),0)\to\pi_n(\map_R(R,R),0)$ is an isomorphism, or equivalently, that $\pi_0(\map_{R\otimes A\op}(R,\Omega^n R))\to \pi_0(\map_R(R,\Omega^n R))$ is an isomorphism. 

By adjunction, this forgetful map is equivalent to the map given by precomposition with $R\otimes A\to R$: $\pi_0(\map_{R\otimes A\op}(R,\Omega^n R))\to \pi_0(\map_{R\otimes A\op}(R\otimes A,\Omega^n R))$, and because $A$ is separable, $R\otimes A\to R$ is split, so that this map is always injective. It thus suffices to prove that it is surjective. 

Note that this is a map between hom sets in $\ho(_R\BiMod_A(\C))$, and by \Cref{cor : sepff}, it is thus equivalent to a map between hom sets in $\RMod_{hA}(\ho(\LMod_R))$. Furthermore, the source in both cases is free as an $R$-module, so by \Cref{lm : locff} it is equivalent to a map between hom sets in $_{hR}\BiMod_{hA}(\ho(\C))$. In other words, we are trying to prove that every map $R\to \Omega^n R$ of left $hR$-modules is right $hA$-linear. We note that $R,\Omega^n R$ are $R$-bimodules, and the right $A$-module structure is induced from the right $R$-module structure by restricting along $f: A\to R$. In other words, for both $R$ and $\Omega^n R$, the right $A$-module structure is given by $M\otimes A\to M\otimes R\to M$. 

It thus suffices to prove that the right $A$-action on $\Omega^n R$ agrees with the following map: $\Omega^n R\otimes A\simeq A\otimes \Omega^n R\to R\otimes \Omega^n R\to \Omega^n R$. Indeed, this is the case for $R$ by assumption, and it will thus follow immediately that any left $hR$-linear map $R\to \Omega^n R$ is also right $hA$-linear. 

Now, by assumption, we already know that this is the case for the $A$-action on $R$, so it suffices to show that the action map of $A$ on $\Omega^n R$ is given (up to homotopy) by $ \Omega^n R\otimes A\simeq \Omega^n(R\otimes A)\to \Omega^n R$, where the second map is $\Omega^n\rho$, $\rho$ being the right action of $A$ on $R$. But this is clear, as the left $A$-action on $\Omega^n R$ is obtained via restiction of scalars from the left $R$-action, which \emph{is} given this way. 
\end{proof}
\begin{rmk}
    Note that the end of this proof really identifies maps in $\ho(\C)$, and there is no coherence claim. This is what the results from \Cref{section:homotopycat} buy us. 
\pend \end{rmk}

In the proof, we really use the homotopy centrality of $f: A\to R$. The result is not true in general if we drop this hypothesis, as the following example shows (in fact, in this example, $R$ is also separable):
\begin{ex}\label{ex:spacenotequiv}
    We start by a computation in a homotopy category, namely, consider $\D$ the symmetric monoidal $1$-category of $\mathbb Z/2d$-graded $\mathbb Q$-vector spaces, where $d$ is some odd integer different from $1$. Let $H$ be some group with a nontrivial automorphism $\alpha$ of order $d$, and consider the corresponding semi-direct product $G:= H\rtimes \mathbb Z/d$, with projection map $p: G\to \mathbb Z/d$ and section $i:\mathbb Z/d\to G$. 
    
    We let $A = \mathbb Q[u]$ where $|u|= 2$, and $u^d= 1$, and $R= \bigoplus_{g\in G}\mathbb Q[2 p(g)]$, where the algebra structure is the natural one, namely, given by $\mathbb Q[2p(g)]\otimes \mathbb Q[2p(g')]\to \mathbb Q[2p(gg')]$ (note that $A$ is given by the same construction, replacing $G$ by $\mathbb Z/d$). Both $A$ and $R$ are separable in $\D$. Let $f: A\to R$ be given by the section $i$, i.e. $u$ maps to the generator of the $i(\sigma)$ summand in $R$ - this is easily checked to be an algebra map. 
    
    Given $h\in H$ such that $\alpha(h)\neq h$, let $g_0= (h,1)$, and consider the corresponding element $r_0$, corresponding to $1\in \mathbb Q[2p(g_0)] = \mathbb Q[2]$ which corresponds to a left $R$-linear map $R\to R[-2]$. We claim that this left $R$-linear map is not right $A$-linear.  Indeed, it is given by $r\mapsto rr_0$, and right $A$-linearity would be the claim that $rar_0 = rr_0a$ which, when $r=1$, is the claim that $ar_0 = r_0a$, which can be checked to be wrong, essentially because $i$ does not land in the center of $G$. More precisely, when $a=u$, $r_0a, ar_0$ live in different summands of $\bigoplus_{g\in G}\mathbb Q[2p(g)]$, one of them in the summand corresponding to $(h,\sigma)$, and the other in the summand corresponding to $(\alpha(h),\sigma)$.

    We claim that this is now enough to give a counterexample to the previous proposition when $f$ is not homotopy central. Indeed, consider $S= \mathbb Q[t^{\pm 1}]$ as a commutative algebra in $\Mod_\mathbb Q$, where $t$ has degree $2d$. The homotopy category of $\Mod_S$ is symmetric monoidally equivalent to $\D$, and the suspension on $\Mod_S$ corresponds to shifting in $\D$. Further, because $A,R$ are separable in $\D$, they can be lifted to algebras in $\Mod_S$ in a weakly unique way by \Cref{thm : e1lift}, similarly for the map $f:\tilde A\to \tilde R$\footnote{In this case, the algebras in question are just Thom spectra so one can actually give a relatively easy construction both of the algebras and the map, as well as a proof that they are separable. This is expanded upon in \Cref{ex:exsecondtime}.}, and because $A$ is commutative in $\D$, \Cref{thm : commlift} implies that $\tilde A$ is commutative in a unique way. We are now left with proving that $\Omega(\map_{\Alg_S}(\tilde A,\tilde R),f)$ is not discrete. By the analysis in the previous proof, it suffices to prove that $\hom_{R\otimes A}(R,\Omega^2 R)\to \hom_R(R,R)$ is not surjective. We have just done this! 
\pend \end{ex}
As an immediate corollary, we find:
\begin{cor}\label{cor:hocommtargetdiscrete}
    Assume $\C$ is additively symmetric monoidal, and let $A\in\Alg(\C)$ be separable and homotopy commutative, and $R\in\Alg(\C)$ be homotopy commutative. 

    In this case, $\map_{\Alg(\C)}(A,R)$ is discrete and equivalent (via the canonical map) to $\hom_{\Alg(\ho(\C))}(hA,hR)$. 
\end{cor}
\begin{proof}
    This follows from \Cref{prop:discretecentralhocomm} as any map $f:A\to R$ is homotopy central, by homotopy commutativity of $R$. 
\end{proof}
\begin{cor}\label{cor:discreteendopd}
 Assume $\C$ is additively symmetric monoidal, and let $A\in\Alg(\C)$ be separable and homotopy commutative.
 
 In this case, for any $n\geq 0$, $\map_{\Alg(\C)}(A^{\otimes n},A)$ is discrete and equivalent to $\hom_{\Alg(\ho(\C)}(hA^{\otimes n},hA)$. 
\end{cor}
We can now prove \Cref{thm : opdlift}: 
\begin{proof}[Proof of \Cref{thm : opdlift}]
By \Cref{cor:discreteendopd}, the (symmetric monoidal) forgetful functor $\Alg(\C)\to \Alg(\ho(\C))$ restricts to a (symmetric monoidal) equivalence between the full subcategories spanned by $A^{\otimes n}, n\geq 0$ and $hA^{\otimes n},n\geq 0$ respectively. 

It therefore induces an equivalence between the space of $\Oo$-algebra structures on $A$, and the space of $\Oo$-algebra structres on $hA$, for any single-colored operad $\Oo$. Now, $hA$ is commutative, so that we $\Oo$-algebra structures on $hA$ in $\Alg(\ho(\C))$ are the same thing as $\Oo$-algebra structures in $\CAlg(\ho(\C))$, which is cocartesian symmetric monoidal. 

Therefore, by \cite[Proposition 2.4.3.9]{HA}, the assumption that $\Oo$ is weakly reduced guarantees that the space of such structures is contractible\footnote{In this case, $\CAlg(\ho(\C))$ is a $1$-category, so we do not really need anything as sophisticated as \cite[Proposition 2.4.3.9]{HA}}. 
\end{proof}
\begin{proof}[Proof of \Cref{thm : commlift}]
    The operad $\mathbb E_\infty$ is clearly weakly reduced, and by \cite[Corollary 5.1.1.5, Theorem 5.1.2.2]{HA}, $\mathbb E_\infty\otimes\mathbb E_1\simeq \mathbb E_\infty$.

    Similarly, by \Cref{ex:weaklyred}, the operads $\mathbb E_n$ are weakly reduced for $1\leq n\leq \infty$, and again by \cite[Theorem 5.1.2.2]{HA}, $\mathbb E_{n+1}\simeq \mathbb E_n\otimes\mathbb E_1$. 
\end{proof}
In fact, we could have guessed ahead of time that $A$ could be made at least $\mathbb E_2$ in a canonical way, via the following elementary observation:
\begin{obs}\label{lm: existse2}
Assume $\C$ is additively symmetric monoidal. Let $A\in\Alg(\C)$ be a separable algebra, which is homotopy commutative. The forgetful map $Z(A) = \hom_{A\otimes A\op}(A,A)\to A$ (cf. \Cref{nota:center}) is an equivalence of algebras, and in particular $A$ admits an $\mathbb E_2$-algebra structure, as $Z(A)$ always does.
\pend \end{obs}
Indeed, by \Cref{prop : homsep}, this internal hom is preserved by passage to the homotopy category, so it suffices to prove the claim there. But now, $hA$ is literally a commutative algebra in a $1$-category, so the claim is obvious.
\begin{rmk}\label{rmk:equivhocomm}
    We can rephrase \Cref{thm : opdlift} as saying that for any weakly reduced operad $\mathcal O$, the forgetful map $\Alg_{\mathcal O\otimes\mathbb E_1}(\C)^\simeq\to  \Alg(\C)^\simeq$ has trivial fibers over separable, homotopy commutative algebras, and in particular is an equivalence when restricted to the appropriate components. 
\pend \end{rmk}
A special case of the previous remark is that $\CAlg(\C)^\simeq\to \Alg(\C)^\simeq$ is an equivalence when restricted to the components of algebras that are homotopy commutative and separable in the target. We note a corollary of this: 
\begin{cor}\label{cor:descenthocomm}
    Let $h\CSep(\C)$ denote the full subspace of $\Alg(\C)^\simeq$ spanned by the separable, homotopy commutative algebras. The functor $\C\mapsto h\CSep(\C) $, defined on additively symmetric monoidal \categories, is limit-preserving.
\end{cor}
\begin{proof}
    This follows from \Cref{cor:descentcommsep} and \Cref{rmk:equivhocomm} in the case $\mathcal O = \mathrm{Comm}$: the natural map $\CSep(\C)\to h\CSep(\C)$ is an equivalence. 
\end{proof}
We also note the following corollary of \Cref{cor:hocommtargetdiscrete}:
\begin{cor}\label{thm : commmorsep}
Assume $\C$ is additively symmetric monoidal.
    Let $A,R\in\CAlg(\C)$. If $A$ is separable, then the forgetful maps 
    $$\map_{\CAlg(\C)}(A,R)\to \map_{\Alg(\C)}(A,R)\to \hom_{\Alg(\ho(\C))}(hA,hR)$$ are equivalences.

    More generally, if $\Oo$ is any \operad{} and $R\in\Alg_{\Oo\otimes\mathbb E_1}(\C)$ is an algebra whose underlying $\mathbb E_1$-algebra is homotopy commutative, then, viewing $A$ as an $\Oo\otimes\mathbb E_1$-algebra using the unique map of \operads{} $\Oo\otimes\mathbb E_1\to \mathbb E_\infty$, we find that the canonical map $$\map_{\Alg_{\Oo\otimes\mathbb E_1}(\C)}(A,R)\to \hom_{\Alg_{\Oo\otimes\mathbb E_1}(\ho(\C))}(hA,hR)$$ 
    is an equivalence. 
\end{cor}
\begin{rmk}
    The condition on $R$ in the second part of the statement is automatic if $\Oo$ is weakly reduced and has at least one operation in arity $2$, e.g. for $\Oo= \mathbb E_n, 1\leq n\leq \infty$.
\pend \end{rmk}
To see that this really is a corollary, we first record the following classical lemma: 

\begin{lm}\label{lm:ffalgebra}
    Let $f:\C\to \D$ be a symmetric monoidal functor, and $A,B\subset \C$ two full subcategories. Assume $A$ is closed under tensor products in $\C$, and furthermore assume that for every $a\in A, b\in B$, the canonical map $\map_\C(a,b)\to \map_\D(f(a),f(b))$ is an equivalence. 

    In this case, for any \operad{} $\Oo$ and any $R\in\Alg_\Oo(A)$, viewed as an $\Oo$-algebra in $\C$, and any $S\in\Alg_\Oo(\C)$ whose underlying objects are in $B$, the canonical map $\map_{\Alg_\Oo(\C)}(R,S)\to \map_{\Alg_\Oo(\D)}(f(R),f(S))$ is an equivalence.
\end{lm}
This is in turn a special case of:
\begin{lm}\label{lm:mapfunclocal}
    Let $f:C\to D$ be a functor between two \categories{}, and let $A,B\subset C$ be full subcategories such that for each $a\in A,b\in B$, the canonical map $\map_C(a,b)\to \map_D(f(a),f(b))$ is an equivalence. 

    Let $X,Y:I\to C$ be two functors such that for each $i\in I, X_i\in A,Y_i\in B$. In this case, the canonical map $$\map_{\Fun(I,C)}(X,Y)\to \map_{\Fun(I,D)}(f\circ X,f\circ Y)$$ is an equivalence. 
\end{lm}
\begin{proof}
  This follows directly from the description of mapping spaces in $\Fun(I,C)$ as ends, cf. \cite[Proposition 5.1]{gepnerhaugsengnikolaus}, but for the sake of completeness, we give here a more elementary proof. 

Let $C_{A,B}\subset C^{\Delta^1}$ be the full subcategory spanned by arrows $a\to b$ where $a\in A,b\in B$. We note that our assumption guarantees that the following is a pullback square: 
\[\begin{tikzcd}
	{C_{A,B}} & {D^{\Delta^1}} \\
	{A\times B} & {D\times D}
	\arrow[from=2-1, to=2-2]
	\arrow[from=1-1, to=1-2]
	\arrow[from=1-1, to=2-1]
	\arrow[from=1-2, to=2-2]
\end{tikzcd}\]
In particular, it remains so after taking $\Fun(I,-)$. We now note that $\Fun(I,C_{A,B})\simeq \Fun(I,C)_{\Fun(I,A),\Fun(I,B)}$ and $\Fun(I,D^{\Delta^1})\simeq \Fun(I,D)^{\Delta^1}$ compatibly. 

Now for $X,Y$ as in the statement, $\map(X,Y)$ is the fiber of $\Fun(I,C)^{\Delta^1}\to \Fun(I,C)\times\Fun(I,C)$ over $(X,Y)$, so that by fullness of $A,B\subset C$, it is also the fiber of $\Fun(I,C)_{\Fun(I,A),\Fun(I,B)}\to \Fun(I,A)\times \Fun(I,B)$ over $(X,Y)$ and thus, because the above square is a pullback, the fiber of $\Fun(I,D)^{\Delta^1}\to \Fun(I,D)\times\Fun(I,D)$ over $(f\circ X,f\circ Y)$, i.e. $\map(f\circ X,f\circ Y)$, as claimed. 
  \end{proof}
\begin{proof}[Proof of \Cref{lm:ffalgebra}]
    $\Alg_\Oo(\C)$ (resp. $\Alg_\Oo(\D)$) is a full subcategory of $\Fun(\Fin_*,\C^\otimes)\times_{\Fun(\Fin_*,\Fin_*)}\{\id\}$ (resp. $\Fun(\Fin_*,\D^\otimes)\times_{\Fun(\Fin_*,\Fin_*)}\{\id\}$). 
    
    We can thus apply \Cref{lm:mapfunclocal} here, by taking $A^\otimes$ to be the full suboperad of $\C^\otimes$ spanned by objects of $A$, and taking $B^\otimes$ to be the full suboperad of $\C^\otimes$ spanned by objects of $B$. 
    
    We simply need to check the assumptions on $f^\otimes: \C^\otimes\to \D^\otimes$, i.e. we need to prove that for any tuples $A_1,...,A_n\in A, B_1,...,B_m\in B$ with corresponding objects $\underline A\in\C^\otimes_{\langle n \rangle}, \underline B\in\C^\otimes_{\langle m\rangle}$, the canonical map $$\map_{\C^\otimes}(\underline A,\underline B)\to \map_{\D^\otimes}(f^\otimes\underline A,f^\otimes\underline B)$$ is an equivalence. 

    This map is a map of spaces over $\hom_{\Fin_*}(\langle n\rangle, \langle m\rangle)$ and so we can take fibers over a given morphism $\alpha:\langle n\rangle\to \langle m \rangle$, and because $f^\otimes$ is symmetric monoidal, it is compatible with the equivalences $\map_{\C^\otimes}(\underline A,\underline B)\simeq \prod_{i\in\langle m\rangle^o}(\bigotimes_{j\in\alpha^{-1}(i)}A_j, B_i)$; and so the claim follows from the assumption on $A,B$, and the fact that $A$ is closed under tensor products. 
\end{proof}
\begin{proof}[Proof of \Cref{thm : commmorsep}]
 This follows again from \Cref{cor:hocommtargetdiscrete}, using the (definitional) equivalence $\Alg_{\Oo\otimes\mathbb E_1}(\C)\simeq \Alg_{\Oo}(\Alg_{\mathbb E_1}(\C))$, and \Cref{lm:ffalgebra} - we apply the latter to the symmetric monoidal functor $\Alg(\C)\to \Alg(\ho(\C))$, with the full subcategories $A,B$ spanned on the one hand by the commutative separable algebras, and on the other hand by the homotopy commutative algebras. 
\end{proof}
We also record the following special case explicitly: 
\begin{cor}\label{cor : cotrunc}
Assume $\C$ is additively symmetric monoidal.
Let $A\in \CAlg(\C)$ be a commutative separable algebra, and $R\in\CAlg(\C)$ an arbitrary commutative algebra. In this case, $\map_{\CAlg(\C)}(A,R)$ is $0$-truncated, i.e. discrete. 
\end{cor}
It is of course a special case of the above, but to make this consequence more concrete, we give an alternative, more elementary proof that could be useful in different contexts. 
\begin{proof}[Alternative proof of \Cref{cor : cotrunc}]
    It suffices to argue that the diagonal map $\map_{\CAlg(\C)}(A,R)\to \map_{\CAlg(\C)}(A,R)\times\map_{\CAlg(\C)}(A,R)$ is an inclusion of components, i.e. a monomorphism. 
    
    Since $\CAlg(\C)$ admits coproducts given by tensor products, this amounts to the claim that the multiplication map $A\otimes A\to A$ is an epimorphism in $\CAlg(\C)$. But this follows immediately from it being a localization at an idempotent, cf. \Cref{cor : commsepsplit}. 
\end{proof}
\subsection{Deformation theory and étale algebras}\label{section:etale}
In the specific case where $\C = \Mod_R$, for some connective ring spectrum, we can try, as in the étale case, to relate connective commutative separable algebras to their $\pi_0$, rather than to their corresponding homotopy algebra $\ho(\C)$. In that regard, the usual techniques of deformation theory work just as well as in the étale case, cf. \cite[Section 7.5]{HA}. We explain how this works in our situation. In fact, a big chunk of the deformation theory works for general $0$-cotruncated commutative algebras.  For a little while, we will therefore be in the setting of spectra and no longer a general $\C$ (although many of these results could be phrased more genreally in the presence of a $\mathrm{t}$-structure). 
\begin{prop}
    Let $S$ be a connective commutative ring spectrum, $R$ a connective commutative $S$-algebra, and $A$ a $0$-cotruncated connective commutative $S$-algebra. 

    In this case, the canonical maps $$\map_{\CAlg_S}(A,R)\to \map_{\CAlg_{\pi_0(S)}}(A\otimes_S\pi_0(S),\pi_0(R))\to \map_{\CAlg_{\pi_0(S)}^\heart}(\pi_0(A),\pi_0(R))$$ are equivalences. 
\end{prop}
\begin{proof}
    Note that, for both maps, by adjunction it suffices to prove that the map $\map_{\CAlg_S}(A,R)\to \map_{\CAlg_S}(A,\pi_0(R))$ is an equivalence. 

    As $R\simeq \lim_n R_{\leq n}$, it suffices to prove that each $R_{\leq n+1}\to R_{\leq n}$ induces an equivalence $\map_{\CAlg_S}(A,R_{\leq n+1})\to \map_{\CAlg_S}(A,R_{\leq n})$. 

    For this, we note that $R_{\leq n+1}\to R_{\leq n}$ is a square zero extension \cite[Corollary 7.4.1.28]{HA}, so it suffices to prove this claim for arbitrary square zero extensions of commutative $S$-algebras by connective modules. 

    So let $\widetilde R\to R$ denote such a square zero extension, classified by a pullback square 
    \[\begin{tikzcd}
	{\widetilde R} & R \\
	R & {R\oplus \Sigma M}
	\arrow["p", from=1-1, to=1-2]
	\arrow["{d_0}", from=1-2, to=2-2]
	\arrow["{d_\eta}"', from=2-1, to=2-2]
	\arrow["p"', from=1-1, to=2-1]
\end{tikzcd}\]
where $M$ is connective. As this is a pullback square, it remains so after applying $\map_{\CAlg_S}(A,-)$, and so, to prove that the left vertical map becomes an equivalence, it suffices to prove that this is so for the right vertical map. But the right vertical map has a left inverse, namely, the projection $R\oplus \Sigma M\to R$, so it suffices to prove that this one gets sent to an equivalence. 

However, this projection map $R\oplus \Sigma M\to R$ is a trivial square zero extension, so it suffices to prove the claim for these ones, i.e., extensions where the corresponding pullback square has $d_\eta \simeq d_0$. This is where we use $0$-cotruncatedness: applying $\map_{\CAlg_S}(A,-)$ yields a pullback square of \emph{sets} of the form: 
\[\begin{tikzcd}
	{\map_{\CAlg_S}(A,\widetilde R)} & {\map_{\CAlg_S}(A,R)} \\
	{\map_{\CAlg_S}(A,R)} & {\map_{\CAlg_S}(A,R\oplus \Sigma M)}
	\arrow["p", from=1-1, to=1-2]
	\arrow["{d_0}", from=1-2, to=2-2]
	\arrow["{d_0}"', from=2-1, to=2-2]
	\arrow["p"', from=1-1, to=2-1]
\end{tikzcd}\]
Where the two $d_0$'s really are the same map, and further are (split) injections. The pullback of sets along injections is given by the intersection of the images. But if the maps are equal, then their images are equal too, so that the intersection is the whole thing. This implies that the left vertical map is an equivalence, as was to be proved. 
\end{proof}
\newcommand{\Q}{\mathcal Q}
\begin{prop}\label{prop:definv0cotrunc}

    Let $\C\mapsto \Q(\C)$ be a limit-preserving subfunctor of the functor $\C\mapsto \CAlg(\C)$ defined on the category of  symmetric monoidal $\infty$-categories.

    Let $\Q^\cn$ denote the restriction of $\Q$ to the category $\CAlg^\cn$ of commutative connective ring spectra along $R\mapsto \Mod_R^\cn$. Suppose that the image of $\Q^\cn(R)\to \CAlg(\Mod_R^\cn)$ consists of $0$-cotruncated algebras. 
    
    In this case, for any commutative connective ring spectrum $R$, the canonical map $R\to \pi_0(R)$ induces an equivalence $\Q^\cn(R)\to \Q^\cn(\pi_0(R))$.
\end{prop}
\begin{proof}
    The argument is essentially the same as before. We use the fact that $\Mod_\bullet^\cn$ preserves the inverse limits involved in Postnikov towers, namely the limit diagrams of the form $R\simeq \lim_n R_{\leq n}$ \cite[19.2.1.5]{SAG}, and pullback squares defining square zero extensions by connective modules \cite[Theorem 16.2.0.2.]{SAG}. 

As $\Q$ preserves all limits, $\Q^\cn$ preserves these specific limits, so that, to prove that $\Q^\cn(R)\to \Q^\cn(\pi_0(R))$ is an equivalence, it suffices to prove that $\Q^\cn(\widetilde R)\to \Q^\cn(R)$ is an equivalence for all square zero extensions by connective modules and hence, as before, it suffices to prove it for trivial square zero extensions by connective modules. 

For these ones, we note that the same argument as before using pullbacks of sets along equal injections implies, by looking at mapping spaces, that $\Q^\cn(R\oplus \Sigma M)\to \Q^\cn(R)$ is fully faithful. In more detail, we can fit this map in a pullback square: \[\begin{tikzcd}
	{\Q^\cn(R\oplus\Sigma M)} & {\Q^\cn(R)} \\
	{\Q^\cn(R)} & {\Q^\cn(R\oplus\Sigma^2M)}
	\arrow[from=1-1, to=2-1]
	\arrow[from=2-1, to=2-2]
	\arrow[from=1-2, to=2-2]
	\arrow[from=1-1, to=1-2]
\end{tikzcd}\]
where the two maps $\Q^\cn(R)\to \Q^\cn(R\oplus\Sigma^2M)$ are equivalent, and so, looking at mapping spaces, we find pullback squares of \emph{sets} of the form: 
\[\begin{tikzcd}
	X & Y \\
	Y & Z
	\arrow[from=2-1, to=2-2]
	\arrow[from=1-2, to=2-2]
	\arrow[from=1-1, to=2-1]
	\arrow[from=1-1, to=1-2]
\end{tikzcd}\]
where the two maps $Y\to Z$ are equal and (split) injective. It follows that the map $X\to Y$ is an isomorphism as before, and hence, that $\Q^\cn(R\oplus\Sigma M)\to \Q^\cn(R)$ is fully faithful.

Because it also has a right inverse (namely $\Q^\cn(R)\to \Q^\cn(R\oplus \Sigma M)$), it follows that it is also essentially surjective, hence an equivalence.
\end{proof}
\begin{cor}\label{cor:seppi0}
    Let $R$ be a connective commutative ring spectrum. Basechange along $R\to\pi_0(R)$ induces an equivalence $\CSep^\cn(R)\to \CSep^\cn(\pi_0(R))$.
\end{cor}
\begin{proof}
    Combine \Cref{cor : cotrunc} and \Cref{prop:definv0cotrunc}. 
\end{proof}
In the case of étale extensions, however, one can go further: flatness (which is part of the definition of étale) forces étale extensions of $\pi_0(R)$ to also be discrete. We do not know if this is so for arbitrary commutative separable extensions, however, in the noetherian case, Neeman proved the following:
\begin{thm}[{\cite[Lemma 2.1., Remark 2.2.]{Neeman}}]\label{thm:Neeman}
    Let $R$ be a discrete commutative noetherian ring. Any commutative separable algebra in $\Mod_R$ is coconnective\footnote{In \cite{Neeman}, Neeman says ``connective'', but he is working with cohomological conventions.}. In particular, any connective commutative separable algebra is discrete. 

    A discrete separable commutative algebra is also flat.
\end{thm}

\begin{rmk}
    In that last sentence, we are considering separable algebras in $\Mod_R$, which are discrete; and not separable algebras in $\Mod_R^\heart$. The latter can be non-flat: for example, any quotient $R\to R/I$ is separable in $\Mod_R^\heart$, as $R/I\otimes_R R/I\cong R/I$. 
\pend \end{rmk}
\begin{rmk}
    Neeman proves more than \Cref{thm:Neeman}, he completely classifies commutative separable algebras over noetherian schemes.   
\pend \end{rmk}
We can thus deduce the following:
\begin{cor}\label{cor:pi0sep}
    Let $R$ be a connective commutative ring spectrum with noetherian $\pi_0$. The functor $\pi_0: \Mod_R^\cn\to \Mod_R^\heart\simeq \Mod_{\pi_0(R)}^\heart\subset \Mod_{\pi_0(R)}$ preserves separable commutative algebras. 

    Furthermore, any separable commutative algebra over $R$ is flat.

    The functor $\pi_0(-)$, or equivalently $\pi_0(R)\otimes_R -$ induces an equivalence $$\CAlg^{sep}(R)^\cn\to \CAlg^{sep}(\pi_0(R))^{\heart, \flat}$$
\end{cor}
We have used the following notation:
\begin{nota}
    We use the superscript $^\flat$ to indicate flatness - we can use this for $\Mod_R$, where $R$ is a ring spectrum \cite[Definition 7.2.2.10.]{HA}, and by extension for \categories{} that admit natural forgetful functors to it, such as $\CAlg(R)$. 
\pend \end{nota}
To prove this, we use the following standard lemma: 
\begin{lm}\label{lm:flat}
    Let $R$ be a connective ring spectrum and $M$ a right $R$-module. Suppose $M\otimes_R \pi_0(R)$ is a flat (in particular discrete) $\pi_0(R)$-module. In this case, $M$ is a flat $R$-module.
\end{lm}
\begin{proof}
    For any discrete left $R$-module $N$, $M\otimes_R N\simeq M\otimes_R\pi_0(R)\otimes_{\pi_0(R)}N$ is discrete by assumption. This is enough by \cite[Theorem 7.2.2.15.(5)]{HA}. 
\end{proof}
\begin{proof}[Proof of \Cref{cor:pi0sep}]
    Let $A$ be a connective separable algebra over $R$. Basechange along $R\to \pi_0(R)$ is symmetric monoidal, so that $A\otimes_R\pi_0(R)$ is separable and hence, by Neeman's theorem (\Cref{thm:Neeman}), coconnective. As it is also connective, it is therefore discrete.
    
    It follows that it is isomorphic to its $\pi_0$, which is also $\pi_0(A)$, and hence, $\pi_0(A)$ is separable, \emph{as an algebra in} $\Mod_{\pi_0(R)}$\footnote{It is obviously separable in $\Mod_{\pi_0(R)}^\heart$, because $\pi_0:\Mod_R^\cn\to \Mod_{\pi_0(R)}^\heart$ is strong symmetric monoidal. }.

    Furthermore, $A\otimes_R\pi_0(R)$ is flat, again by \Cref{thm:Neeman}, and hence $A$ is flat, by the previous lemma. 

Now, as commutative separable algebras are $0$-truncated, by \Cref{cor : cotrunc}, this implies that the two functors (which we just explained are equivalent) $\pi_0(-)$ and $\pi_0(R)\otimes_R -$ are fully faithful as functors $\CAlg^{sep}(R)^\cn\to \CAlg^{sep}(\pi_0(R))^{\heart, \flat}$. We are left with proving that they are essentially surjective.  

But the inclusion $\Mod_{\pi_0(R)}^{\heart,\flat}\to \Mod_{\pi_0(R)}$ is strong symmetric monoidal, and hence preserves separable algebras. Thus, any object in $\CAlg^{sep}(\pi_0(R))^{\heart, \flat}$ lifts to $\CAlg^{sep}(\pi_0(R))^\cn$, and by \Cref{cor:seppi0}, anything there can be lifted to $\CAlg^{sep}(R)^\cn$.  
\end{proof}
In other words, under this noetherian-ness assumption, connective commutative separable algebras in $\Mod_{\pi_0(R)}$ are exactly the flat ordinary commutative separable algebras. Note that given a flat ordinary commutative separable algebra $A_0$, the corresponding commutative separable algebra over $R$ is flat, and hence has homotopy groups $\pi_*(A)\cong A_0\otimes_{\pi_0(R)}\pi_*(R)$. 

This allows us to compare separability with étale-ness in the sense of Lurie in the noetherian case. Namely, we have:
\begin{prop}\label{prop:etalesep}
    Let $R$ be a commutative ring spectrum and $A$ a commutative $R$-algebra. 
If $A$ is étale in the sense of \cite[Definition 7.5.0.4.]{HA}, then $A$ is separable. 

Conversely, if $R$ is connective and $\pi_0(R)$ is noetherian, then if $A$ is separable, connective and $\pi_0(A)$ is finitely presented over $\pi_0(R)$ then $A$ is étale in the same sense.
\end{prop}
\begin{rmk}
   It is not clear to the author what the optimal statement is. Clearly, one cannot drop all connectivity assumptions: for instance, Galois extensions are separable (see \Cref{ex:galoissep}), but many of them, such as $\KO\to \KU$ are not étale. 

   It is reasonable to expect that one can drop the noetherian assumption, and possibly the connectivity assumption on $A$. 
\pend \end{rmk}
\begin{proof}
    Assume $A$ is étale. By definition, $\pi_0(A)$ is étale over $\pi_0(R)$. In particular, $\pi_0(A)$ is flat over $\pi_0(R)$, and separable in the classical sense. So let $e\in\pi_0(A)\otimes_{\pi_0(R)}\pi_0(A)\cong \pi_0(A\otimes_R A)$ be a separability idempotent. This is in turn an idempotent in $A\otimes_R A$ which gets sent to $1\in \pi_0(A)$ under the multiplication map. 

    In particular, it induces a map $(A\otimes_R A)[e^{-1}]\to A$. Because homotopy groups commute with localizations, and because $A$ is flat, this map can be identified, on homotopy groups, with $(\pi_0(A)\otimes_{\pi_0(R)}\pi_0(A)\otimes_{\pi_0(R)}\pi_*(R))[e^{-1}]\to \pi_0(A)\otimes_{\pi_0(R)}\otimes\pi_*(R)$. 

    Because $(\pi_0(A)\otimes_{\pi_0(R)}\pi_0(A))[e^{-1}]\to \pi_0(A)$ is an isomorphism, this map is also an isomorphism, which proves that $(A\otimes_R A)[e^{-1}]\simeq A$. As $e$ is idempotent, it follows that $A\otimes_R A\to A$ has an $A\otimes_RA$-linear splitting, thus proving that it is separable.

    For the converse, we already know $\pi_0(A)$ is separable, flat and finitely presented, which means that it is étale over $\pi_0(R)$ in the classical sense. Furthermore, we also know that $A$ is flat over $R$, which altogether means that $A$ is étale in the sense of  \cite[Definition 7.5.0.4.]{HA}. 
\end{proof}
\begin{ques}\label{question:noetherian}
    Do these results (\Cref{thm:Neeman}, \Cref{cor:pi0sep} and \Cref{prop:etalesep}) continue to hold without a noetherian-ness assumption ?
\end{ques}
In their recent work \cite{NikoLuca}, Naumann and Pol partially answer this question by removing the noetherian assumption and adding the assumption that the algebra $A$ is perfect as an $R$-module. One easily sees that their proof only uses the assumption that $A$ is \emph{almost perfect} \cite[Definition 7.2.4.10]{HA}. In fact, as it turns out, it follows that in this case, almost perfect implies perfect. We record it here for the convenience of the reader, but the proof is the same as that of \cite[Proposition 10.5]{NikoLuca}:
\begin{prop}\label{prop:etalesepaperf}
    The answer to \Cref{question:noetherian} is yes, when restricted to almost perfect separable algebras. More precisely, fix a connective commutative ring spectrum $R$. Let $A$ be a commutative separable $R$-algebra, whose underlying $R$-module is almost perfect. 

    In this case, $A$ is flat, and hence connective; $\pi_0(A)$ is separable as a $\pi_0(R)$-algebra in $\Mod_{\pi_0(R)}$ (and not only in $\Mod_{\pi_0(R)}^\heart$). In particular, $A$ is étale in the sense of Lurie over $R$.  

In particular, $A$ is \emph{perfect}, and even finitely generated projective as an $R$-module. In other words, almost perfect commutative separable algebras are always perfect/finitely generated projective, and they correspond exactly to finite étale extensions. The functor $\pi_0(-)$, or equivalently $\pi_0(R)\otimes_R -$, induces an equivalence between these and finite étale $\pi_0(R)$-algebras. 
\end{prop}
\begin{proof}
 By \Cref{lm:flat}, to prove that $A$ is flat, it suffices to prove that $A\otimes_R\pi_0(R)$ is flat, and in particular discrete, over $\pi_0(R)$. 
 
 We reduce to the case of a field using \cite[\href{https://stacks.math.columbia.edu/tag/068V}{Tag 068V}]{stacks-project}. In more detail, we note that the word ``pseudo-coherent'' used in \cite{stacks-project} is equivalent to ``almost perfect'' in the case of discrete rings, so that this lemma does apply to our situation. Then, we note that basechange preserves almost perfect modules, so that $A\otimes_R\pi_0(R)$ is almost perfect over $\pi_0(R)$, i.e. pseudo-coherent. Finally, we note that being flat is equivalent to being of $\mathrm{Tor}$-amplitude in $[0,0]$, so that by \cite[\href{https://stacks.math.columbia.edu/tag/068V}{Tag 068V}]{stacks-project} it suffices to prove that $A\otimes_R\pi_0(R)\otimes_{\pi_0(R)}k$ is concentrated in degree $0$ for any field $k$ and morphism $\pi_0(R)\to k$.

 But now $A\otimes_R k$ is the basechange of $A$ along the commutative ring map $R\to \pi_0(R)\to k$, so it is a separable commutative algebra over the field $k$, and by \cite[Proposition 1.6]{Neeman}, these are all discrete, as was to be shown. 

    Now, we have proved that $A$ is flat (and in particular connective), so that $\pi_0(R)\otimes_R A\simeq \pi_0(A)$ - the former is obviously separable, and therefore, so is the latter. 

    To prove that $A$ is étale, as in the proof of \Cref{prop:etalesep}, because we already know that it is flat, it suffices to prove that $\pi_0(A)$ is étale over $\pi_0(R)$. We know that it is separable and flat, so it suffices to prove that it is finitely presented, but it is finitely presented \emph{as a module} over $\pi_0(R)$, which immediately implies that it is finitely presented as an algebra as well, and hence étale.  

Furthermore, $\pi_0(A)$ is a finitely presented flat $\pi_0(R)$-module, hence it is finitely generated projective, and thus, by flatness of $A$, $A$ is a finitely generated projective $R$-module, and in particular perfect. 

    By \Cref{cor:seppi0}, the functor $\CAlg^{sep}(R)^{\mathrm{aperf}} \subset \CAlg^{sep}(R)^\cn \to \CAlg^{sep}(\pi_0(R))^\cn$ is fully faithful, and it lands in the full subcategory $\CAlg^{sep}(\pi_0(R))^{\heart,\flat}$. To conclude the proof, it thus suffices to prove that its essential image is exactly the finite étale extensions. But if $A_0$ is a finite étale extension of $\pi_0(R)$, it is in particular finitely generated projective, and so its unique lift to a flat $\pi_0(R)$-module is also finitely generated projective, and hence almost perfect. 
\end{proof}
\Cref{question:noetherian} remains however open in full generality.

We finally discuss Morita equivalence in the commutative setting. Recall that for ordinary commutative rings, Morita equivalence implies isomorphism. 
\begin{prop}
Suppose $\C$ is presentably symmetric monoidal. 
Let $A,B$ be commutative algebras, and suppose they are Morita equivalent, i.e. there is a $\C$-linear equivalence $\Mod_A(\C)\simeq \Mod_B(\C)$.

If $A$ is separable, then $A\simeq B$ as algebras in $\C$. If $\C$ is additively symmetric monoidal, then they are equivalent as commutative algebras. 
\end{prop}
\begin{proof}
Note that the center $Z(R)$ of a ring is Morita invariant. 

Furthermore, we know that $Z(A)\simeq A$, as $A$ is commutative separable, and as $B$ is commutative we have a retraction $B\to Z(B)\to B$. By Morita invariance, and the previous fact, we thus have a retraction $B\to A\to B$, proving that $B$ is separable, and thus (by commutativity) $B\simeq Z(B)$. 

It follows that $A\simeq B$ as algebras. If we now assume $\C$ is additively symmetric monoidal, then \Cref{thm : commlift} implies that $A\simeq B$ as commutative algebras.
\end{proof}
\begin{ques}
Is this result valid if $A$ is not assumed separable and when $\C= \Sp$ ? 
\end{ques}
\begin{rmk}
If $A$ is not assumed to be separable, and we are allowed to work in a local setting, then the result is \emph{false}. A counterexample is given by taking $\C= \Sp_{T(n)}$, the \category{} of $T(n)$-local spectra\footnote{$K(n)$-local spectra would work just as well.} with $n\geq 1$ and at an implicit prime $p$, and $R$ a $T(n)$-local commutative algebra. Then, by \cite[Theorem B.(1)]{chromaticFT}, for any finite abelian $p$-group $A$, $R[A]$ is Morita equivalent over $\C$ to $R^{BA}$. Yet, they are typically not equivalent when $n\geq 2$.
\pend \end{rmk}
\begin{rmk}
Note that if $Z(A)$ is separable, we can make most of the argument work without assuming that $A$ is commutative: the argument proves that in this case, $B$ is separable. More precisely, if $A\in\Alg(\C), B\in \CAlg(\C)$, if $Z(A)$ is separable, and finally if $A$ is Morita equivalent to $B$, then $B$ is also separable (and it is the center of $A$). It is not clear to the author to what extent ``$Z(A)$ is separable'' is really an extra assumption, cf. \Cref{question : sepcenter} and the discussion in \Cref{section:center}. 
\pend \end{rmk}
\section{A variant: ind-separability}\label{section:indsep}
The goal of this section is to study a variant of the notion of separability, which we call ``ind-separability'', and which is better suited in some ``infinitary'' situations. We will see that they share many of the properties of separable algebras, in particular concerning highly structured multiplicative structures.

We will see that, in the \category{} of $K(n)$-local spectra, Morava $E$-theory is ind-separable - the proof of this will require as its only input the computation of the ring of cooperations of Morava $E$-theory, by Hopkins--Ravenel, Baker, and revisited by Hovey in \cite{hovey}. As a corollary, we will obtain a relatively simple proof of the Goerss--Hopkins--Miller theorem, or in some sense a reorganization of the classical proof; as well as an extension to the folklore claim that $E$-theory admits a unique $\mathbb E_d$-structure for any $1\leq d\leq \infty$\footnote{The $d=1$ case is known as the Hopkins--Miller theorem, and the $d=\infty$ case as the Goerss--Hopkins--Miller theorem.}.

Because of the infinitary nature of the notion of ind-separability, and because separable algebras only have strong enough rigidity properties in the commutative case, the variant we introduce here only really works in the (homotopy) commutative case. For similar reasons, this variant is best suited in the compactly generated case, and we will mostly stick to this assumption. 
\subsection{Ind-separability}
Recall that, if $A$ is a commutative separable algebra, the multiplication map $A\otimes A\to A$ witnesses the target as the localization of the source at an idempotent. The key observation of this section is that many of our results only really need it to be a localization at some set of elements. We thus define:
\begin{defn}
    Let $\C$ be a presentably, stably symmetric monoidal \category{}, and $A\in \CAlg(\C)$ a commutative algebra in $\C$. We say that $A$ is \emph{ind-separable} if there is a set $S\subset \pi_0\map(\one,A\otimes A)$ such that the multiplication map witnesses $A$ as the localization $(A\otimes A)[S^{-1}]$ in the \category{} of $A\otimes A$-modules.
\pend \end{defn}
\begin{rmk}
    Note that a priori, an ind-separable algebras has no particular reason to be a filtered colimit of separable algebras, i.e. an ind-(separable algebra). This will, however, be our main source of examples, cf. the subsequent sections. 
\pend \end{rmk}
This notion is relatively well-suited if we want to study the moduli space of commutative structures extending the underlying $\mathbb E_1$-algebra structure of $A$, at least when $\C$ is compactly generated. 

However, if we also want to get off the ground and go from a homotopy algebra structure to an actual algebra structure, we need to phrase this in ``up-to-homotopy'' terms. Because in a stably symmetric monoidal \category{} with filtered colimits, localizing a commutative algebra at a set of elements is a relatively well understood procedure, namely it is given by a telescope (see e.g. \cite[Appendix C]{bunke2018beilinson}), we can in fact give the following definition in the compactly generated case:
\begin{defn}
    Let $\C$ be a compactly generated presentably, stably symmetric monoidal \category{}, and $A\in \CAlg(\ho(\C))$ a homotopy commutative homotopy algebra in $\C$. We say that $A$ is \emph{homotopy ind-separable} if there exists a set $S\subset \pi_0\map(\one,A\otimes A)$ such that the multiplication map induces, for each compact object $c\in\C^\omega$, an isomorphism $$\pi_*(\map(c,A\otimes A))[S^{-1}] \to \pi_*\map(c,A)$$
\pend \end{defn}
In this definition, the localization is taken outside of $\pi_*$, and by it we simply mean the usual telescope construction in abelian groups. 
\begin{rmk}
Because filtered colimits are exact in $\Ab$, the condition that the above map be an isomorphism \emph{at $c$} is closed under co/fiber sequences in $\C^\omega$, and thus it suffices to check it on generators, e.g. on $R$ if $\C=\Mod_R(\Sp)$ for some commutative ring spectrum $R$. 
\pend \end{rmk}
Again because localizations of commutative algebras in stable \categories{} are computed as telescopes \cite[Appendix C]{bunke2018beilinson}, the following is an immediate consequence of the definition: 
\begin{lm}
     Let $\C$ be a compactly generated presentably, stably symmetric monoidal \category{}, and $A\in \CAlg(\C)$ a commutative algebra in $\C$. If $A$ is ind-separable, then its underlying homotopy algebra is homotopy ind-separable. 
\end{lm}
In some cases of interest though, the compacts of $\C$ are complicated to calculate, and so it can be useful to formulate a criterion at the level of $\C$. It can be hard to phrase in this generality, because the diagram that defines the telescope has no reason to lift to $\C$ in general. However, if $S$ is particularly nice, the diagram \emph{can} be lifted to $\C$ even if $A$ is only a homotopy algebra. 

To explain this in more detail, we begin with a construction. 
\begin{cons}
Let $\C$ be a symmetric monoidal additive \category{} admitting sequential colimits. Let $B\in\CAlg(\ho(\C))$ be a homotopy commutative homotopy algebra, and let $s: \one\to B$ be an ``element'' of $B$. This induces a (homotopy-)$B$-module map $B\to B$ given by multiplication by $s$, and thus, an $\mathbb N$-shaped diagram $B\to B\to B\to\dots$ \emph{in} $\C$. 

We call its colimit the telescope of $B$ at $s$, $\Tel_s(B)$. 

Suppose instead given an $\mathbb N$-indexed family $s=(s_i)$ of elements of $B$. We can then form its telescope as the colimit of the diagram $B\xrightarrow{s_1}B\xrightarrow{s_1s_2}B\xrightarrow{s_1s_2s_3}B\dots$, and  we still denote it by $\Tel_s(B)$. 
\pend \end{cons}
\begin{rmk}
    In this construction, the diagram $B\to B\to B\to \dots$ can really be constructed in $\C$ and not only in $\ho(\C)$, because $\mathbb N$ is free as an $\infty$-category (cf., e.g., \cite[Proof of Proposition 4.4.2.6]{HTT}). 
\pend \end{rmk}
\begin{defn}
    Let $\C$ be a symmetric monoidal additive \category{} admitting sequential colimits. Let $B\in\CAlg(\ho(\C))$ be a homotopy commutative homotopy algebra and let $M$ be a (homotopy-)$B$-module with a (homotopy-)$B$-module map $f: B\to M$. Let $s$ be an $\mathbb N$-indexed family of elements of $B$. 

    We say that $f$ witnesses $M$ as a telescope of $B$ at $s$ if there exist homotopies $f\circ (s_1...s_n)\simeq f$ for all $n$ that induce an equivalence $\Tel_s(B)\simeq M$. 
\pend \end{defn}
\begin{rmk}
    Note that in the latter definition, because $\mathbb N$ is free as an \category{}, the collection of homotopies $f\circ (s_1...s_n)\simeq f$ is sufficient to induce a map from the colimit to $M$.
\pend \end{rmk} 
\begin{rmk}
   If $\C$ is stable and compactly generated, the condition that the induced map $\Tel_s(B)\to M$ be an equivalence does not depend on the chosen homotopies $f\circ (s_1...s_n)\simeq f$. Indeed, it can be checked after applying $\pi_0\map(c,-)$ for all compacts $c$, and the induced map there does not depend on the homotopies. 
\pend\end{rmk}
\begin{defn}
    Let $\C$ be a stably symmetric monoidal \category{} with sequential colimits compatible with the tensor product, and let $A\in\CAlg(\ho(\C))$ be a homotopy commutative homotopy algebra. We say that $A$ is homotopy $\omega$-separable if the multiplication map $A\otimes A\op\to A$ witnesses $A$ as a telescope of $A\otimes A\op$ at some sequence $s=(s_i)$ of elements $s_i:\one\to A\otimes A\op$. 
\pend \end{defn}
The following is again an easy consequence of the definition:
\begin{lm}
     Let $\C$ be a compactly generated presentably, stably symmetric monoidal \category{}, and $A\in \CAlg(\ho(\C))$ a homotopy commutative homotopy algebra in $\C$. If $A$ is homotopy $\omega$-separable, then it is homotopy ind-separable. 
\end{lm}
The way things will go is that we will prove things about (homotopy) ind-separable (homotopy) algebras, and our main example of a homotopy ind-separable homotopy algebra (Morava $E$-theory) will be proved to be so by proving that is a homotopy $\omega$-separable algebra.

\subsection{Obstruction theory}
Having defined (homotopy) ind-separability, our first goal is to argue that most of our results about commutative separable algebras extend to the case of (homotopy) ind-separable algebras, at least when $\C$ is compactly generated, and under suitable assumption on phantom maps. To prove these results, we will use obstruction theory - we could in principle follow a similar approach as in \Cref{section:homotopycat}, but because filtered colimits do not interact so well with the formation of homotopy ($n$-)categories, we would have to stick closer to the proof of \Cref{lm:Postnikovcats} rather than just using it as a lemma. As a result, the proof would be more convoluted and the obstruction theory from \cite{PVK} neatly packages the constructions anyway. 
\begin{warn}
    An earlier version of this document was missing a key assumption which will appear here, about phantom maps. We will discuss this assumption when it is relevant.
\pend\end{warn}

Unlike \Cref{section:homotopycat}, we now start with a compactly generated stably symmetric monoidal $\C$.
\begin{assu}\label{assu:syn}
$\C$ is a compactly generated stably symmetric monoidal \category, in which tensor products commute with colimits in each variable. Further, we assume that the compact objects of $\C$ are closed under non-empty tensor products\footnote{We do \emph{not} assume that the unit is compact.}
\pend\end{assu}

We will rely heavily on \cite{PVK}, and so our first goal is to get ourselves in the setting of the obstruction theory from that paper. 

The following construction is very similar to the construction in \cite[Section 4.4]{hopkinsluriebrauer}, so we only briefly go over the details. 
\begin{cons}\label{cons:synrig}
Let $\C$ be as in \Cref{assu:syn}. 

Let $\Syn_\C \subset \Fun^\times((\C^\omega)\op,\Sp_{\geq 0})$. Similarly to \cite[Section 4.4]{hopkinsluriebrauer}, the assumption that $\C^\omega$ is closed under non-empty tensor products makes $\Syn_\C$ into a non-unital symmetric monoidal \category{} for which the Yoneda embedding $\C^\omega\to \Syn_\C$ is canonically non-unitally symmetric monoidal. 

    Again, similarly to \cite[Section 4.4]{hopkinsluriebrauer}, we obtain an essentially unique non-unitally symmetric monoidal colimit-preserving preserving functor $f:\Syn_\C\to \C$ whose restriction to $\C^\omega$ is (non-unitally symmetric monoidally) equivalent to the inclusion. The right adjoint of $f$, given by the restricted Yoneda embedding $M:c\mapsto \Map(-,c)_{\geq 0}$ thus acquires a canonical (non-unital) lax symmetric monoidal structure.  

    The structure maps $\Map(-,c)_{\geq 0}\otimes \Map(-,d)_{\geq 0}\to \Map(-,c\otimes d)_{\geq 0}$ are equivalences whenever $c,d\in \C^\omega$, so because $\C$ is compactly generated and $M$ preserves filtered colimits, we find that these structure maps are also equivalences for all $c,d\in\C$. 

    In particular, $\Map(-,\one)_{\geq 0}\otimes M(c) \simeq M(c)$ for all $c\in\C^\omega$ and thus, because of the universal property of $\Syn_\C$, $M(\one)\otimes -\simeq \id$ as functors $\Syn_\C\to \Syn_\C$ (see \cite[Lemma 4.4.9]{hopkinsluriebrauer}). 

    It follows that $\Syn_\C$ is in fact a \emph{unital} symmetric monoidal \category{}, and that $M: \C\to \Syn_\C$ is fully faithful and symmetric monoidal. Furthermore, $\Syn_\C$ is clearly Grothendieck prestable, complete and separated. We abuse notation and write $\one$ also for the unit of $\Syn_\C$, i.e. for $M(\one)$. 

    The grading given by $F[1] := F(\Omega -)$, where $\Omega : \C\to \C$ is the loop functor (equivalently the suspension functor on $\C\op$) makes it into a graded Grothendieck prestable category in the sense of \cite{PVK}, and the assembly map induces a shift structure on the unit $\tau: \Sigma\one[-1]\to \one$, again in the sense of \cite{PVK}. This is the only shift algebra we will consider in this section, so ``periodic module'' in the sense of \cite{PVK} should always be understood with respect to this shift algebra. 
\pend \end{cons}
\begin{warn}\label{warn:synsyn}
    When specialized to the case of $\C=\Mod_E(\Sp_{K(n)})$ where $E$ is Morava $E$-theory at height $n$, our definition of $\Syn_\C$ is related to the $\Syn_E$ appearing in \cite{hopkinsluriebrauer}, but they are not the same: $\Syn_E$ is also defined as an \category{} of product-preserving presheaves, but on something smaller than $\C^\omega$. 
\pend\end{warn}
We recall the following definition from \cite{PVK}:
\begin{defn}[{\cite[Definition 2.17]{PVK}}]
    An object $M\in\Syn_\C$ is a \emph{periodic} module over $\one$ if the canonical map induces an isomorphism $\pi_0(M)\otimes_{\pi_0(\one)}\pi_*(\one)\to \pi_*(M)$; equivalently if $\tau: \Sigma M[-1]\to M$ is a $1$-connective cover\footnote{The equivalence between these two conditions is proved as \cite[Proposition 2.16]{PVK}}.
\pend \end{defn}
\begin{lm}\label{lm:repperiodic}
   For any $c\in\C$, $M(c)$ is a periodic module over $\one$.
\end{lm}
\begin{proof}
    We use the second characterization for this: $\tau : \Sigma M(c)[-1]\to M(c)$ identifies with the map $\Sigma (\Map(\Sigma -, c))_{\geq 0}\to \Map(-,c)_{\geq 0}$. 

    Furthermore, $\Map(\Sigma -, c)\simeq \Omega \Map(-,c)$. Now, for any spectrum $X$, the canonical map $\Sigma (\Omega X)_{\geq 0}\to X_{\geq 0}$ is a $1$-connective cover, so we are done. 
\end{proof}
\begin{lm}\label{lm:periodicpi0}
    Let $M,N\in\Syn_\C$ be periodic modules over $\one$. If $f: M\to N$ is a morphism which induces an isomorphism on $\pi_0$, then $f$ is an equivalence. 
\end{lm}
\begin{proof}
    It follows from the first definition of periodic modules that $f$ induces an isomorphism on all homotopy groups. Given the definition of $\Syn_\C$ and of the homotopy groups, it is clear that this implies that it is an equivalence. 
\end{proof}
\begin{lm}\label{lm:periodicrep}
The functor $M:\C\to \Syn_\C$ identifies $\C$ with the full subcategory of $\Syn_\C$ spanned by periodic modules over $\one$.  
\end{lm}
\begin{proof}
By \cite[Proposition 2.22]{PVK}, we have an equivalence $\Syn_\C^{per}\simeq \Sp(\Syn_\C)^{\tau^{-1}}$, where the superscript $\tau^{-1}$ means ``$\tau$-local'', i.e. the $M\in \Sp(\Syn_\C)\simeq \Fun^\times((\C^\omega)\op,\Sp)$ such that the canonical map $\Sigma M(\Sigma -)\to M$ is an equivalence. 

By definition, this canonical map is an equivalence if and only if $M$ sends suspensions in $(\C^\omega)\op$ to loops, i.e. if and only if $M$ is an exact functor \cite[Corollary 1.4.2.14.]{HA}.  But $M:\C\to \Fun((\C^\omega)\op,\Sp)$ identifies $\C$ with the full subcategory of $\Fun((\C^\omega)\op,\Sp)$ spanned by exact functors, i.e. $\Ind(\C^\omega)$. 
\end{proof}
\begin{cor}\label{cor:synlocsep}
    Let $\C$ be as in \Cref{assu:syn}, and let $A\in\CAlg(\ho(\C))$ be a homotopy-ind-separable homotopy commutative homotopy algebra.

    In this case, $\pi_0M(A)\in \Mod_{\pi_0(\one)}(\Syn_\C)$ is an ind-separable commutative algebra.

    If $A\in\Alg(\C)$ is homotopy commutative and ind-separable, the same holds. 
\end{cor}
\begin{proof}
We first observe that the canonical map $\pi_0M(A)\otimes_{\pi_0(\one)}\pi_0M(A)\to \pi_0M(A\otimes A)$ is an equivalence. Granted this observation, the claim simply follows from the definition of (homotopy) ind-separable and the fact that $\pi_0:\Syn_\C\to\Syn_\C$ preserves filtered colimits (and again, the fact that localizations of commutative algebras are given by telescopes). 

To prove the observation, we combine \Cref{lm:repperiodic} and \cite[Proposition 2.16]{PVK} to get that $\pi_0M(-)\simeq M(-)\otimes_{\one}\pi_0(\one)$, from which the claim follows as $M(-)$ is symmetric monoidal, and so is basechange along $\one\to\pi_0(\one)$ in $\Syn_\C$. 
\end{proof}

\begin{lm}\label{lm:extvanish}
    Let $\D$ be a stably symmetric monoidal \category{} admitting filtered colimits that are compatible with the tensor product, and let $A\in\CAlg(\D)$ be an ind-separable commutative algebra.

    Let $L_A$ be the fiber of the multiplication map $A\otimes A\to A$, viewed as an $A$-bimodule. For any $A$-module $M$, viewed as an $A$-bimodule via restriction along the multiplication map, we have that the mapping spectrum $\map_{A\otimes A}(L_A,M)$ vanishes. 
\end{lm}
\begin{proof}
    Basechange along the multiplication map is left adjoint to restriction, so it suffices to prove that the basechange of $L_A$ is zero, and for this it suffices to prove that the co-unit $A\otimes_{A\otimes A}A\to A$ is an equivalence. This follows immediately from ind-separability. 
\end{proof}
\begin{thm}\label{thm:indsepmappingspace}
    Let $\C$ be as in \Cref{assu:syn}, and let $A\in\Alg(\C)$ be homotopy commutative and homotopy ind-separable. For any homotopy commutative algebra $R\in\Alg(\C)$, the mapping space $\map_{\Alg(\C)}(A,R)$ is discrete and equivalent to $\hom_{\Alg(\Syn_\C^\heart)}(\pi_0M(A),\pi_0M(R))$. 
\end{thm}
\begin{proof}
First, note that $M(A),M(R)$ are periodic algebras over $\one$, so that we can use \cite[Proposition 5.7]{PVK} to study the mapping space $\map_{\Alg(\C)}(A,R)\simeq \map_{\Alg(\Syn_\C)}(M(A),M(R))$. 

By \cite[Proposition 5.7]{PVK}, the fiber of $$\map_{\Alg(\Mod_{\one^{\leq n+1}}(\Syn_\C))}(\one^{\leq n+1}\otimes M(A),\one^{\leq n+1}\otimes M(R))\to  \map_{\Alg(\Mod_{\one^{\leq n}}(\Syn_\C))}(\one^{\leq n}\otimes M(A),\one^{\leq n}\otimes M(R))$$ at any point in the target is a space of paths in a certain space. We claim that this space is contractible. Indeed, this space is $$\map_{\BiMod_{\pi_0M(A)}(\Syn_\C)}(L_{\pi_0M(A)/\pi_0(\one)}^{\mathbb E_1}, \Sigma^{n+2}\pi_0M(R)[-(n+1)])$$
Now because $M(R)$ is homotopy commutative, the $M(R)$-bimodule $\pi_0M(R)[-(n+1)]$ is pulled back along the multiplication map $\pi_0M(R)\otimes_{\pi_0(\one)}\pi_0M(R)$, and so the same holds when we see this bimodule as a $\pi_0M(A)$-bimodule. Thus by \Cref{lm:extvanish} applied to $\D= \Mod_{\pi_0(\one)}(\Syn_\C)$ (and \Cref{cor:synlocsep}), this space is contractible, and hence so is the path space between any two points therein. Here we use that by \cite[Theorem 7.3.5.1]{HA}, $L_{\pi_0M(A)/\pi_0(\one)}^{\mathbb E_1}$ is what we have called $L_{\pi_0M(A)}$, computed in $\Mod_{\pi_0(\one)}(\Syn_\C)$.

This proves that the map 
$$\map_{\Alg(\Mod_{\one^{\leq n+1}}(\Syn_\C))}(\one^{\leq n+1}\otimes M(A),\one^{\leq n+1}\otimes M(R))\to  \map_{\Alg(\Mod_{\one^{\leq n}}(\Syn_\C))}(\one^{\leq n}\otimes M(A),\one^{\leq n}\otimes M(R))$$
is an equivalence, and thus by \cite[Remark 5.2]{PVK}, the map $\map_{\Alg(\Syn_\C)}(M(A),M(R))\to \map_{\Alg(\Mod_{\pi_0(\one))}}(\pi_0(\one)\otimes M(A),\pi_0(\one)\otimes M(R))\simeq \map_{\Alg(\Syn_\C^\heart)}(\pi_0M(A),\pi_0M(R))$ is an equivalence. 
\end{proof}
As in the proof of \Cref{thm : commmorsep}, and because $\Syn_\C^\heart$ is a $1$-category, we obtain:
\begin{cor}\label{cor:indsepnoBrown}
    Let $\C$ be as in \Cref{assu:syn} and let $A\in\Alg(\C)$ be an ind-separable homotopy commutative algebra. Let $\Oo$ be an arbitrary one-colored \operad. In this case, the canonical forgetful map $$\Alg_{\Oo\otimes\mathbb E_1}(\C)^\simeq\times_{\Alg(\C)^\simeq}\{A\}\to\Alg_{\Oo\otimes\mathbb E_1}(\Syn_\C^\heart)^\simeq\times_{\Alg(\Syn_\C^\heart)^\simeq}\{\pi_0M(A)\} $$ is an equivalence. In particular, if $\Oo$ is weakly reduced, $\Alg_{\Oo\otimes\mathbb E_1}(\C)^\simeq\times_{\Alg(\C)^\simeq}\{A\}$ is contractible. This is the case e.g. if $\Oo=\mathbb E_d,d\geq 1$.  
\end{cor}
More generally, again arguing as in the proof of \Cref{thm : commmorsep}:
\begin{cor}
 Let $\C$ be as in \Cref{assu:syn} and let $A\in\Alg(\C)$ be an ind-separable homotopy commutative algebra.
    If $\Oo$ is any \operad{} and $R\in\Alg_{\Oo\otimes\mathbb E_1}(\C)$ is an algebra whose underlying $\mathbb E_1$-algebra is homotopy commutative, then, viewing $A$ as an $\Oo\otimes\mathbb E_1$-algebra using the unique map of \operads{} $\Oo\otimes\mathbb E_1\to \mathbb E_\infty$, we find that the canonical map $$\map_{\Alg_{\Oo\otimes\mathbb E_1}(\C)}(A,R)\to \hom_{\Alg_{\Oo\otimes\mathbb E_1}(\Syn_\C^\heart)}(\pi_0M(A),\pi_0M(R))$$ 
    is an equivalence. 
\end{cor}
As showcased in the proof of \Cref{cor:indsepnoBrown}, our obstruction theory based on $\Syn_\C$ really only knows about the $\pi_0$ of $M(A)$, which can be understood as the cohomology theory represented by $A$ \emph{on $\C^\omega$}, rather than the whole of $\C$. The following assumption on $\C$ will thus be needed if we want to lift information about $\Syn_\C^\heart$ to information about $\ho(\C)$: 
\begin{assu}\label{assu:Brown}
    Let $X,Y\in \C$ and let $f: \pi_0\Map(-,X)\cong \pi_0\Map(-,Y)$ be an isomorphism between the cohomology theories they represent on $\C^\omega$. There exists a map $\tilde f : X\to Y$ which lifts $f$ (and is therefore an equivalence).
\pend\end{assu}
\begin{rmk}\label{rmk:BrownHoyois}
    By \cite[Theorem 7]{hoyoisBrown}, this assumption is often satisfied in practice. Specifically, for $\C$ stable, it suffices for $\ho(\C^\omega)$ to be ``countable'', in the sense that it has countably many isomorphism classes, and mapping sets between any two objects are countable. 
    
    Because of long exact sequences induced by fiber sequences, to check this it suffices to check that there is a countable number of generators, and that the mapping sets between (shifts) of these generators are all countable. 
\pend\end{rmk}
\begin{thm}\label{thm:indsepsyn}
    Let $\C$ be as in \Cref{assu:syn}.

    Let $A\in\CAlg(\ho(\C))$ a homotopy ind-separable homotopy commutative homotopy algebra. In this case, the moduli space $\Alg(\C)^\simeq\times_{\Alg(\Syn_\C^\heart)^\simeq}\{\pi_0M(A)\}$ is contractible.

    If $\C$ satisfies \Cref{assu:Brown}, then any lift $\tilde A$ of $\pi_0M(A)$ is equivalent, as an object of $\C$, to $A$. 
        \end{thm}
\begin{proof}
As in \Cref{obs:onlynonempty}, \Cref{cor:indsepnoBrown} shows that the real content here is the non-emptiness of this moduli space. 

We first replace $\Alg(\C)$ by $\Alg(\Syn_\C)^{per}$ (algebras in $\Syn_\C$ whose underlying object is periodic), and we use \cite[Theorem 5.4]{PVK} to prove the existence of a lift - by \Cref{lm:periodicrep}, the map $\Alg(\C)\to \Alg(\Syn_\C)^{per}$ is an equivalence. 

In more detail, \cite[Theorem 5.4]{PVK} tells us that the obstructions to the existence of a lift live in $\mathrm{Ext}^{n+2}(L_{\pi_0M(A)/\pi_0(\one)}^{\mathbb E_1},\pi_0M(A)[-n])$, so it suffices to prove that these groups vanish. By \cite[Theorem 7.3.5.1]{HA}, $L_{\pi_0M(A)/\pi_0(\one)}^{\mathbb E_1}$ is what we have called $L_{\pi_0M(A)}$, computed in $\Mod_{\pi_0(\one)}(\Syn_\C)$.  As the $\pi_0M(A)$-bimodule structure on $\pi_0M(A)[-n]$ is obtained by shifting the bimodule structure on $\pi_0M(A)$, and in particular by restriction along the multiplication map, \Cref{lm:extvanish} implies that these $\mathrm{Ext}$-groups vanish (using that $\pi_0M(A)$ is ind-separable by \Cref{cor:synlocsep}). We thus find a periodic algebra $\tilde A$ in $\Syn_\C$ whose $\pi_0$ is $\pi_0M(A)$ as was claimed. 

Now for the second part, by \Cref{lm:periodicrep}, we may in fact write $M(\tilde A)$ for some algebra $\tilde A\in\C$, with $\pi_0M(\tilde A)\cong \pi_0M(A)$ as algebras. By \Cref{assu:Brown}, it follows that $\tilde A\simeq A$, and so $M(\tilde A)\simeq M(A)$, all of this lifting the isomorphism $\pi_0M(\tilde A)\cong \pi_0M(A)$, and so we do get an algebra structure on $M(A)$ lifting the one on $\pi_0M(A)$. By fully faithfulness of $\C\to \Syn_\C$, this is what we wanted.  
\end{proof}
\begin{rmk}
    Without \Cref{assu:Brown}, this proof constructs a homotopy ind-separable algebra $\tilde A$ such that $\pi_0M(\tilde A)\cong\pi_0M(A)$ as algebras in $\Syn_\C^\heart$, i.e. as multiplicative cohomology theories on $\C^\omega$, but there is no way to guarantee that $A\simeq \tilde A$. 

    Even with \Cref{assu:Brown}, there is no way to guarantee that the multiplication we obtain on $A$ is the one we started with, they only agree up to phantom maps. 
\pend\end{rmk}
One might be tempted to conclude that the same sort of result holds for $\Alg(\C)^\simeq \times_{\Alg(\ho(\C))^\simeq}\{A\}$, because ``$\ho(\C)\to\Syn_\C^\heart$ is fully faithful''\footnote{This is what an earlier version of this document claimed without justification.}. However, $\ho(\C)\to\Syn_\C^\heart $ is generally not fully faithful, for similar reasons that \Cref{assu:Brown} was needed: the hom-set between $\pi_0M(X)$ and $\pi_0M(Y)$ is the set of natural transformation $\pi_0\Map(-,X)\to \pi_0\Map(-,Y)$ as functors on $\C^\omega$, not on $\C$. For example, any phantom map from $X$ to $Y$ is sent to $0$ in $\Syn_\C^\heart$.  Recall:
\begin{defn}
    A phantom map $X\to Y$ in $\C$ is a map such that for any compact $c\in\C^\omega$, the composite $c\to X\to Y$ is nullhomotopic. 
\pend\end{defn}
\begin{rmk}
    Of couse if the nullhomotopy is natural in $c\in\C^\omega$, then $X\to Y$ is nullhomotopic itself, but $\Syn_\C^\heart =\Fun^\times(\ho(\C^\omega)\op,\Ab)$ cannot see this. 
\pend\end{rmk}

Similarly, in the above theorem, if $\C$ satisfies \Cref{assu:Brown}, we obtain a lift on $M(A)$, or equivalently $A$, of the algebra structure on $\pi_0M(A)$, but this tells us that the homotopy algebra structure that we obtained on $A$ need only agree with the original one \emph{up to phantom} maps. And in fact, this is no surprise: our assumption that $A$ be ind-separable only depends on the multiplication $\mu$ up to phantom maps, because it is tested after mapping in from compact objects. The above theorem shows that if two homotopy algebra structures on $A$ agree up to phantom maps, then at most one of them can be lifted to an actual algebra structure. 

In particular, to properly get statements about $\ho(\C)$ rather than $\Syn_\C^\heart$, one needs to make assumptions about phantom maps. 
\begin{obs}
    If there are no phantom maps $A\to R$, then $\hom_{\ho(\C)}(hA,hR)\to \hom_{\Syn_\C^\heart}(\pi_0M(A),\pi_0M(R))$ is injective, and thus so is  $\hom_{\Alg(\ho(\C))}(hA,hR)\to \hom_{\Alg(\Syn_\C^\heart)}(\pi_0M(A),\pi_0M(R))$. 

    For any lift $\tilde A$ of $A$, the map $\map_{\Alg(\C)}(\tilde A,R)\to \hom_{\Alg(\Syn_\C^\heart)}(\pi_0M(A),\pi_0M(R))$ factors through $\hom_{\Alg(\ho(\C))}(A,hR)$, and is an equivalence by \Cref{thm:indsepsyn}. If there are no phantom maps $A\to R$, the above implies that both maps $\map_{\Alg(\C)}(\tilde A,R)\to \hom_{\Alg(\ho(\C))}(A,hR)$ and $\hom_{\Alg(\ho(\C))}(A,hR)\to \map_{\Alg(\Syn_\C^\heart)}(\pi_0M(A),\pi_0M(R))$ are equivalences. 
\end{obs}
The above observation buys us the following version of \Cref{thm:indsepsyn}:
\begin{thm}\label{thm:indsepBrown}
     Let $\C$ be as in \Cref{assu:syn} and \Cref{assu:Brown}. 
    Let $A\in\CAlg(\ho(\C))$ a homotopy ind-separable homotopy commutative homotopy algebra, and assume that $A$ receives no phantom map from any tensor power of $A$. In this case, the moduli space $$\Alg(\C)^\simeq\times_{\Alg(\ho(\C)))^\simeq}\{A\}$$ is contractible, i.e. $A$ admits a unique lift to an algebra in $\C$. 

Let $A\in\Alg(\C)$ be a homotopy commutative, ind-separable algebra, and let $R\in\Alg(\C)$ be a homotopy commutative algebra. If there are no phantom maps from $A$ to $R$, then the mapping space $\map_{\Alg(\C)}(A,R)$ is equivalent to $\hom_{\Alg(\ho(\C))}(hA,hR)$\footnote{One could weaken the assumption to ``Any two homotopy algebra maps that differ by a phantom map are homotopic'', but it does not seem like this is a checkable criterion.}. 
\end{thm}

We could also run an obstruction-theory argument to get to highly structured commutative structures on homotopy ind-separable homotopy commutative homotopy algebras, but as in \Cref{section:comm}, we can also deduce it by more elementary means. 
\begin{cor}\label{cor:indsepBrown2}
    Let $\C$ be as in \Cref{assu:syn} and \Cref{assu:Brown}, and let $A\in\Alg(\C)$ be an ind-separable homotopy commutative algebra. Let $\Oo$ be an arbitrary one-colored \operad.   If $A$ receives no phantom maps from tensor powers of $A$, the canonical forgetful map $$\Alg_{\Oo\otimes\mathbb E_1}(\C)^\simeq\times_{\Alg(\C)^\simeq}\{A\}\to\Alg_{\Oo\otimes\mathbb E_1}(\ho(\C))^\simeq\times_{\Alg(\ho(\C))^\simeq}\{hA\} $$ is an equivalence. 
\end{cor}
More generally, again as a corollary of \Cref{lm:ffalgebra}, we obtain:
\begin{cor}\label{cor:indsepBrown3}
Let $\C$ be as in \Cref{assu:syn} and \Cref{assu:Brown}, and let $A\in\CAlg(\C)$ be ind-separable. If $\Oo$ is any \operad{} and $R\in\Alg_{\Oo\otimes\mathbb E_1}(\C)$ is an algebra whose underlying $\mathbb E_1$-algebra is homotopy commutative, and which receives no phantom maps from $A$, then, viewing $A$ as an $\Oo\otimes\mathbb E_1$-algebra using the unique map of \operads{} $\Oo\otimes\mathbb E_1\to \mathbb E_\infty$, the canonical map $$\map_{\Alg_{\Oo\otimes\mathbb E_1}(\C)}(A,R)\to \hom_{\Alg_{\Oo\otimes\mathbb E_1}(\ho(\C))}(hA,hR)$$ 
    is an equivalence.
\end{cor}
\begin{rmk}
    As in the separable case, a consequence of this corollary is the discreteness of $\map_{\CAlg(\C)}(A,R)$, and, just as in that case, we could give a more elementary proof of this specific fact, cf. \Cref{cor : cotrunc} and its alternative proof. 
\pend\end{rmk}
\begin{cor}
    Let $\C$ be as in \Cref{assu:syn} and \Cref{assu:Brown}, and let $A\in\CAlg(\ho(\C))$ be a homotopy commutative, homotopy ind-separable homotopy algebra in $\C$ which receives no phantom maps from any tensor power of $A$. For any $1\leq d\leq \infty$, the moduli space $\Alg_{\mathbb E_d}(\C)^\simeq \times_{\Alg(\ho(\C))^\simeq}\{A\}$ is contractible. 
\end{cor}
The upshot of this discussion is that, at least in the compactly-generated case, with some assumption on phantom maps, and using slightly less elementary methods, we are able to recover most of the results from the commutative separable case in the commutative (homotopy) ind-separable case.  
\subsection{Examples}
We now discuss examples of ind-separable algebras.
\subsubsection{Ind-(separable algebras)}
The first natural source of examples is filtered colimits of (commutative) separable algebras. Of course, separable algebras are ind-separable (one can pick the set $S$ to consist of the single separability idempotent). 
\begin{lm}\label{lm:ind-sep implies indsep}
Let $\C$ be as in \Cref{assu:syn}, and let $A_\bullet: I\to \CAlg(\C)$ be a filtered diagram of commutative separable algebras. In this case, $\colim_I A_i$ is ind-separable. 

If $I$ is countable, one can choose $S$ in the definition of ind-separable to be countable.
\end{lm}
\begin{proof}
    For every $i\in I$, let $s_i:\one\to A_i\otimes A_i\to A\otimes A$ be the image in $A\otimes A$ of the separability idempotent of $A_i$, and let $S$ be the set of the $s_i$'s. It is easy to verify that this does the job.
\end{proof}
\begin{ex}
    Let $X$ be a profinite set. The algebra $C(X;\mathbb Z)$ of continuous functions on $X$ is ind-separable, as the filtered colimit of $i\mapsto C(X_i;\mathbb Z)$ for any presentation of $X$ as $\lim_i X_i$, where each $X_i$ is finite. However, if $X$ is not finite, it is not separable. 
\pend \end{ex}
\begin{ex}\label{ex:proGalois}
    More generally, Rognes' pro-Galois extension \cite[Definition 8.1.1]{rognes} are ind-separable, by the above lemma together with \Cref{ex:galoissep}. In particular, taking the Goerss--Hopkins--Miller theorem for granted, Devinatz and Hopkins prove in \cite{devinatzhopkins} that Morava $E$-theory is a pro-Galois extension of the $K(n)$-local sphere $\mathbb S_{K(n)}$ in the \category{} of $K(n)$-local spectra. As we wish to give a non-circular proof of the Goerss--Hopkins--Miller theorem, we will give a different proof that Morava $E$-theory is ind-separable below. 
\pend \end{ex}
\subsubsection{Morava E-theory}\label{subsection:Morava}
In this section, we study Morava $E$-theory. \Cref{ex:proGalois} together with its description as a profinite Galois extension \cite{rognes} show that it is ind-separable in the \category{} of $K(n)$-local spectra. However, the proof that it is a pro-Galois extension relies on its highly commutative multiplicative structure, cf. \cite{devinatzhopkins}, i.e. on the Goerss--Hopkins--Miller theorem. 

We offer here a proof of the latter based on our earlier work on ind-separable algebras. The key (and in fact, only) input that we need about Morava $E$-theory is the computation of $\pi_*(L_{K(n)}(E\otimes E))$, as done by Hopkins--Ravenel, Baker, and revisited by Hovey in \cite{hovey} (we refer to \textit{loc. cit.} for a brief history of this computation). 
\begin{rmk}\label{rmk:isitanewproof?}
    Our results on ind-separable algebras rely on the obstruction theory from \cite{PVK}, an obstruction theory which was designed and used to give a proof of the Goerss--Hopkins--Miller theorem, so one might wonder to what extent our proof is actually different. It is not completely clear to the author - it however seems that it is at the very least a re-organization of that proof. Indeed, we first prove a single result about $E$-theory, namely its ind-separability, and then let the obstruction theory machine take its course, with no further input needed, unlike in \cite[Section 7]{PVK}, where calculations about Morava $E$-theory show up alongside the obstruction theory (among other things, $\mathrm{Ext}$-group computations in $E_*E$-comodules).

    Furthermore, as is clear from our proofs, we only really need the obstruction theory to get an $\mathbb E_1$-structure and describe $\mathbb E_1$-maps to other algebras - our proof clarifies the formal aspect of going from there to higher $\mathbb E_d$'s (including $d=\infty$). In particular, we obtain a proof of the folklore fact that Morava $E$-theory admits a unique $\mathbb E_d$-structure also for $1<d<\infty$ that does not require computing the corresponding $\mathbb E_d$-cotangent complexes - while this computation is not complicated (they all vanish, for $d>1$), it does not allow for generalizations to more general operads of the form $\Oo\otimes\mathbb E_1$. 

Finally, while we use the same obstruction theory as in \cite[Section 7]{PVK}, we apply it to a much simpler \category{}: our $\Syn_\C$ has no completion/localization coming into its definition. 

    In other words, it is not clear to what extent our proof is really new, but it is a re-packaging of the classical proof which has several advantages. 
\pend \end{rmk}
Fix a (from now on, implicit) prime $p$ and a height $n$. For a perfect field $k$ of characteristic $p$, and a formal group $\mathbf G$ of height $n$ over $k$, we have a spectrum $E(k,\mathbf G)$, called Morava $E$-theory (or Lubin-Tate theory), usually denoted $E$ or $E_n$. It can for instance be constructed using the Landweber exact functor theorem, and has a homotopy associative, homotopy commutative ring structure. It is also $K(n)$-local, so we can consider it as an object in $\CAlg(\ho(\Sp_{K(n)}))$. We refer to \cite[Part 1]{rezk1998notes} for an introduction to these homotopy ring spectra. 

As we mentioned, the only input we need is a computation of  $\pi_*(L_{K(n)}(E\otimes E))$. In the statement, we write $\hat\otimes$ for the $K(n)$-local tensor product, and $C(X,R)$ for the graded ring of continuous functions from a topological space $X$ to a graded topological ring $R$. For $k$ algebraic over $\mathbb F_p$ (and perfect), Hovey proves:
\begin{thm}[{\cite[Theorem 4.11]{hovey}}]
    There is an isomorphism $$\pi_*(E\hat\otimes E)\cong C(\Gamma,E_*)$$
    for which the multiplication map $\pi_*(E\hat\otimes E)\to E_*$ is identified with evaluation at the neutral element $e\in\Gamma$, $C(\Gamma,E_*)\to E_*$. Here, $\Gamma$ is the (profinite) Morava stabilizer group, equivalently, the group of automorphisms of $E$ in $\Alg(\ho(\Sp_{K(n)}))$. 
\end{thm}
    Let $X$ be a profinite space with a point $x\in X$. Write $X = \lim_i X_i$ where the $X_i$'s are finite sets, with projection maps $p_i:X\to X_i$, and let $\delta_i:X\to X_i\to \{0,1\}$ denote the indicator function of $(p_i)^{-1}(p_i(x))$. 
    \begin{lm}\label{lm:locindicator}
   Composing the $\delta_i$'s with the inclusion $\{0,1\}\to \mathbb Z$, form the subset $S$ of $C(X,\mathbb Z)$ consisting of the $\delta_i$'s. 

    Then evaluation at $x$, as a ring map $e:C(X,\mathbb Z)\to \mathbb Z$, witnesses the target as the localization of the source at $S$. 
    \end{lm}
\begin{proof}
As $\mathbb Z$ is discrete, $C(X,\mathbb Z)$ is the colimit of the $C(X_i,\mathbb Z)$ along restriction maps. Now, the localization of $C(X_i,\mathbb Z)$ at the indicator function of $p_i(x)$ is clearly $\mathbb Z$, and the result follows easily. 
\end{proof}
We also recall the following lemma from \cite{hovey}:
\begin{lm}[{\cite[Proposition 2.5]{hovey}}]\label{lm:ctsfuncsbasechange}
Suppose $G$ is a profinite group and $R$ is a graded commutative ring that is complete in the $\mathfrak a$-adic topology for some homogeneous ideal $\mathfrak a$. Then there is a natural isomorphism $R\hat\otimes C(G,\mathbb Z)\to C(G,R)$, where $\hat\otimes$ is the $\mathfrak a$-adically completed tensor product. 
\end{lm}
\begin{cor}
    The homotopy algebra $E\in\CAlg(\ho(\Sp_{K(n)}))$ is homotopy ind-separable.
\end{cor}
\begin{proof}
We prove that it is in fact homotopy $\omega$-separable. 

Let $\Gamma \cong \lim_k \Gamma/U_k$ be a description of the Morava stabilizer group as a countable inverse limit of its finite quotients (we implicitly use here that $\Gamma$ is first countable, cf. \cite[Theoem 1.4]{hovey}, and let $\delta_k$ denote the indicator function of $U_k$ (this corresponds to $\delta_i$ in \Cref{lm:locindicator} with $x=$ the neutral element of $\Gamma$). 

    Let $S\subset \pi_0(E\hat\otimes E)\cong C(\Gamma,E_0)$ correspond to the set of the $\delta_k$'s. We claim that the multiplication map $E\hat\otimes E\to E$ witnesses the latter as a telescope of $E\hat\otimes E$ at $S$ in $\Sp_{K(n)}$.  
Indeed, this telescope is the $K(n)$-localization of the same telescope \emph{in} $\Sp$, and we can compute that the homotopy groups of the latter are simply $\pi_*(E\hat\otimes E)[S^{-1}]\cong C(\Gamma,E_*)[S^{-1}]$. In particular, they are concentrated in even degrees and  the sequence $(p,u_1,...,u_{n-1})$ is a regular sequence on them. To express this precisely, we can e.g. observe that $E\hat\otimes E$ can be viewed as an $\MU$-module, and so we can make sense of $(p,u_1,...,u_{n-1})$ on it, and they agree with the ones coming from $E_*$. The same can be said for $E\hat\otimes E[S^{-1}]$.

Now, for an $\MU$-module  $M$ on which $u_n$ acts invertibly, the $K(n)$-localization is given by $\lim_k M\otimes_{\MU}\MU/(p^k,...u_{n-1}^k)$, and so, if $M$ is concentrated in even degrees and the sequence $(p,u_1,...,u_{n-1})$ is regular on $M$, then the homotopy groups of $L_{K(n)}M$ are simply the $\mathfrak m = (p,u_1,...,u_{n-1})$-adic completion of the homotopy groups of $M$.  
 
In particular, $\pi_*(L_{K(n)}((E\hat\otimes E)[S^{-1}])$ is the $\mathfrak m$-adic completion of $C(\Gamma,E_*)[S^{-1}]$, i.e., by \Cref{lm:ctsfuncsbasechange} the $\mathfrak m$-adic completion of $E_*\otimes C(\Gamma,\mathbb Z)[S^{-1}]$, and so, by \Cref{lm:locindicator}, just $E_*$. This is only a verification on homotopy groups, but it is not hard to see that it implies the desired statement. 
\end{proof}
\begin{rmk}
    Note that $\Sp_{K(n)}$ \emph{is} compactly generated, and since its compacts are also dualizable, they are closed under non-empty tensor products. However, the unit is not compact. 
\pend \end{rmk}
\begin{lm}\label{lm:K(n)Brown}
    The \category{} of $K(n)$-local spectra satisfies \Cref{assu:syn} and \Cref{assu:Brown}. 
\end{lm}
\begin{proof}
    \Cref{assu:syn} is clear, in fact $\Sp_{K(n)}$ is compactly generated (as a stable \category) by $L_{K(n)}X$ for any finite spectrum $X$ of type $n$. 

    Therefore, by \Cref{rmk:BrownHoyois}, it suffices to prove that for a type $n$ spectrum $X$, $\pi_*\map(L_{K(n)}X,L_{K(n)}X)$ is countable. Because $X$ is a finite spectrum, this reduces to proving that $\pi_*(L_{K(n)}X)$ is countable. For this, we refer to the discussion about finite type in the introduction of \cite{devinatzfinite}. We sketch the argument below for the convenience of the reader. In what follows, we let $E_n$ denote Morava theory at height $n$ over $\mathbb F_{p^n}$.  

    The argument is essentially that there is a strongly convergent spectral sequence of signature $$E_2^{s,t}=H^s(\mathbb G_n, (E_n)_t X)\implies \pi_{t-s}(L_{K(n)}X)$$ by \cite[Proposition 6.7]{DH}, using that $X$ has type $n$. 

    Using again that $X$ has type $n$, we observe that $E_n\otimes X$ is in the thick subcategoy generated by $K(n)$. Now the spectral sequence with $E_2$-page $H^s(\mathbb G_n, K(n)_t)$ consists of countable groups: $\mathbb G_n$ is a profinite group with a countable basis, and each $K(n)_t$ is a discrete countable group. Thus, the same holds for $H^s(\mathbb G_n,(E_n)_tX)$.

    Finally, this spectral sequence has a vanishing line, i.e. for a fixed $r\geq 2$, $E_r^{s,t}=0$ for $s>>0$ by the smashing theorem, so the countability of the $E_2$ terms implies the countability of the groups it converges to (there are no infinite extensions because of the vanishing line).
\end{proof}
\begin{rmk}
    Alternatively, the proof of the analogous result for $\Sp_{T(n)}$ is simpler because for a type $n$ finite spectrum $X$, $L_{T(n)}X$ is a telescope of a $v_n$-self map on $X$. One can then simply observe that $L_{T(n)}(E\otimes E)\simeq L_{K(n)}(E\otimes E)$ because $E$ is an $\MU$-module (note that in $\Sp_{T(n)}$, $L_{K(n)}$ is smashing, so this equivalence between $L_{T(n)}$ and $L_{K(n)}$ for $\MU$-modules follows from the same one for $\MU$, which in turn follows from \cite[Theorem 2.7.(iii)]{ravenel1993life}). 
\pend\end{rmk}
It already follows from \Cref{thm:indsepsyn} and \Cref{cor:indsepnoBrown} that there is a unique commutative algebra in $\Sp_{K(n)}$, $\tilde E$, which represents $E^*(-)$ on compact $K(n)$-local spectra, and is equivalent to $E$ as a ($K(n)$-local) spectrum; and furthermore its endomorphism operad is entirely determined by the corresponding one for $E^*(-)$, which one can compute - for exemple its endomorphism space is discrete and isomorphic to the Morava stabilizer group.

For completeness, to reassure the reader about phantom maps and to relate our work to algebra structures in the homotopy category, we spend some time discussing phantom maps to Morava $E$-theory. 
\begin{lm}\label{lm:landweberphantom}
    Let $E,E'$ be Landweber exact spectra.
    \begin{itemize}
        \item $E\otimes E'$ is Landweber exact;
        \item There are no nonzero phantom maps $E\to E'$.
    \end{itemize}
\end{lm}
\begin{proof}
    The first part is a consequence of \cite[§15]{rezk1998notes}, and the second is \cite[Lecture 17, Corollary 7]{lurienotes}. 
\end{proof}
\begin{warn}
    This lemma is about phantom maps in $\Sp$. There are more phantom maps in $\Sp_{K(n)}$, because the compact objects are of the form $L_{K(n)}X$ for $X$ a finite type $\geq n$ spectrum, so there are fewer compact objects.
\end{warn}
\begin{lm}\label{lm:phantomK(n)}
  Fix a perfect $\mathbb F_p$-algebra $k$, and a formal group $\mathbf G$ of height $n$ over $k$, and let $E=E(k,\mathbf G)$ be the corresponding Morava $E$-theory, and let $X\to E$ be a $K(n)$-locally phantom map. It is also phantom in $\Sp$.
\end{lm}
\begin{proof}
    Let $V$ be a finite spectrum, and $f: V\to X$ a map. We wish to show that $V\to X\to E$ is null, or equivalently that $L_{K(n)}V\to L_{K(n)}X\to E$ is null (as $E$ is $K(n)$-local). Now note that $L_{K(n)}V$, being $\omega_1$-compact, is $K(n)$-locally a sequential colimit of finite type $n$ spectra, say $L_{K(n)}V\simeq L_{K(n)}\colim_\mathbb N V_k$. Now each $V_k\to L_{K(n)}X\to L_{K(n)}E$ is null because of our assumption, so the only obstruction to $L_{K(n)}V\to E$ being null is in $\lim^1_\mathbb N \pi_1\Map(V_k, E)$. It therefore suffices to argue that this $\lim^1$ is $0$, by e.g. showing that it satisfies the Mittag-Leffler condition. 

    But each $V_k$ is a type $n$ complex, so $\Map(V_k,E)$ is in the thick subcategory generated by $K(n)$, and thus its $\pi_m$, for any fixed $m$, is an Artinian $\pi_0(E)$-module\footnote{If $k$ is a finite field, it is in fact a \emph{finite} abelian group; but for $k$ a large perfect field, even $\pi_0(K(n))$ is not finite - however the homotpy groups of $K(n)$ are finite dimensional vector spaces over $\pi_0(K(n))\cong k$.}. This automatically implies the Mittag-Leffler condition. 
\end{proof}
\begin{cor}
    There are no nonzero phantom maps in $K(n)$-local spectra from any tensor powers of Morava $E$-theories to any Morava $E$-theory. 
\end{cor}
\begin{proof}
    Morava $E$-theories are Landweber exact, so there are no nonzero phantom maps in $\Sp$ of the form $\bigotimes_{i=1}^k E(k_i,\mathbf G_i)\to E$ by \Cref{lm:landweberphantom}. By \Cref{lm:phantomK(n)}, this implies that there are no nonzero phantoms in $\Sp_{K(n)}$. 
\end{proof}
We thus obtain the Goerss--Hopkins--Miller theorem, and its variants for other operads, namely:
\begin{cor}\label{cor:GHM}
 Fix a perfect algebraic extension $k$ of $\mathbb F_p$, and a formal group $\mathbf G$ of height $n$ over $k$, and let $E=E(k,\mathbf G)$ be the corresponding Morava $E$-theory, considered as a (homotopy commutative) homotopy algebra. We have:
    \begin{enumerate}
        \item For any weakly reduced \operad{} $\Oo$ (e.g. $\mathbb E_d, 1\leq d\leq \infty$), the moduli space $$\Alg_{\Oo\otimes\mathbb E_1}(\Sp)^\simeq\times_{\Alg(\ho(\Sp))^\simeq}\{E\}$$ is contractible. 
        \item For any \operad{} $\Oo$ and any $R\in\Alg_{\Oo\otimes\mathbb E_1}(\Sp_{K(n)})$ whose underying algebra is homotopy commutative and which receives no phantom map from $E$, viewing $E$ as an $\Oo\otimes\mathbb E_1$-algebra using the unique map of \operads{} $\Oo\otimes\mathbb E_1\to \mathbb E_\infty$, the canonical map $$\map_{\Alg_{\Oo\otimes\mathbb E_1}(Sp_{K(n)})}(E,R)\to \hom_{\Alg_{\Oo\otimes\mathbb E_1}(\ho(\Sp_{K(n)}))}(hE,hR)$$ is an equivalence. This is the case e.g. if $R$ is a Lubin-Tate theory.
    \end{enumerate}

    In particular, if we consider the underlying spectrum of $E$, its space of $\mathbb E_d$-structures, for any $1\leq d\leq \infty$, is equivalent to $B\Aut(\Gamma)$.
\end{cor}
\begin{proof}
We first note that $\Sp_{K(n)}$ satisfies \Cref{assu:syn} and \Cref{assu:Brown} by \Cref{lm:K(n)Brown}. 

Now, by \Cref{thm:indsepBrown}, \Cref{cor:indsepBrown2} and \Cref{cor:indsepBrown3}, the only thing left to comment on is why we could write $\Sp$ in place of $\Sp_{K(n)}$ in item 1. The point is that $E$ is $K(n)$-local, and $\Sp_{K(n)}$ is a symmetric monoidal Bousfield localization of $\Sp$, so that the space of $\Oo$-algebra structures on $E$ in $\Sp$ is equivalent to the one in $\Sp_{K(n)}$. There, $E$ is ind-separable and so the results from the previous subsection apply.
\end{proof}
The right hand side of these equivalences, i.e. homotopy algebra maps from $hE$ to $hR$ can also be computed, at least under favourable circumstances, e.g. if $R$ is also a Morava $E$-theory, or more generally if it is even $2$-periodic, cf. e.g.  \cite{rezk1998notes}. 
\begin{rmk}
    At some point, the Goerss--Hopkins--Miller theorem was the only known way to construct a commutative ring structure on Morava $E$-theory. Lurie proposed an alternative construction in \cite{lurie2018elliptic} where he directly gives a construction of $E$-theory with its commutative ring structure. 
\pend \end{rmk}

In the case of $E(k,\mathbf G)$ for an algebraic extension of $\mathbb F_p$, $k$, we wanted to give a self-contained argument for the ind-separability, for our proof to be at least a somewhat new proof of the Goerss--Hopkins--Miller theorem. We now move on to the case of a general perfect commutative $\mathbb F_p$-algebra - for this, we use an analogue of the computation of $\pi_*(L_{K(n)}(E\otimes E))$ for general perfect $\mathbb F_p$-algebras which follows from Lurie's work \cite{lurie2018elliptic}. We prove: 

\begin{thm}\label{thm:Moravaindsepperf}
    Let $R$ be a perfect (discrete) commutative $\mathbb F_p$-algebras, $\mathbf G$ a formal group of height exactly $n$ over $R$. Assume that $R$ is ind-separable over $\mathbb F_p$, i.e. that the multiplication $R\otimes_{\mathbb F_p}R\to R$ is a localization at a set $S$ of elements. 

    In this situation, $E(R,\mathbf G)$ is ind-separable in $\Sp_{K(n)}$.  
\end{thm}
\begin{rmk}
The \category{} $\Sp_{K(n)}$ is a smashing localization of $\Sp_{T(n)}$, so that this result implies the same one in $\Sp_{T(n)}$. 
\pend\end{rmk}
As mentioned above, the key ingredient is again a computation of $E(R,\mathbf G)\otimes E(R,\mathbf G)$ (to be taken in $\Sp_{K(n)}$) - this computation is a combination of the universal property of Morava $E$-theory \cite[Theorem 5.1.5.]{lurie2018elliptic} together with an explicit analysis of the coproduct in Lurie's $\mathcal{FG}$ (cf. \cite[Remark 5.1.6.]{lurie2018elliptic}), by way of an analysis of the stack of isomorphisms between formal groups, cf. \cite[Theorem 5.23]{goerssMFG} (Goerss credits Lazard for this result). 
\begin{proof}
 Consider $R\otimes_{\mathbb F_p} R$ with the two formal groups induced from $\mathbf G$ along the two inclusions of $R$, say $\mathbf G_1$ and $\mathbf G_2$. 

By \cite[Theorem 5.23]{goerssMFG}, we can find a sequence $R_k$ of finite étale extensions of $R\otimes_{\mathbb F_p}R$ whose colimit $R_\infty$ classifies isomorphisms of formal groups between $\mathbf G_1$ and $\mathbf G_2$ (in particular, it acquires one specific formal group, $\mathbf G_\infty$, which comes with isomorphisms to the basechanges of $\mathbf G_i, i=1,2$)\footnote{Note that in \cite{goerssMFG}, Goerss $\mathrm{Iso}(\mathbf G_1,\mathbf G_2)_k\to \Spec(R\otimes_{\mathbb F_p} R)$ is finite étale, and hence it is also affine - the corresponding ring is our $R_k$. }.

    Furthermore, it is not hard to deduce from \cite[Theorem 5.1.5.]{lurie2018elliptic} (see also \cite[Remark 5.1.6.]{lurie2018elliptic}) that $E(R,\mathbf G)\otimes E(R,\mathbf G) \simeq E(R_\infty,\mathbf G_\infty)$ as commutative algebras (the tensor product is taken in $K(n)$-local spectra), and the multiplication map to $E(R,\mathbf G)$ corresponds to the map $R_\infty\to R$ classifying the identity isomorphism of $\mathbf G$. 

Let $S$ be a set of elements of $R\otimes_{\mathbb F_p}R$ such that the multiplication map $R\otimes_{\mathbb F_p}R\to R$ witnesses the target as the localization of the domain at $S$ (this exists by assumption). We claim that each induced map $R_k[S^{-1}]\to R$ is a localization, in fact, a split localization. Indeed, $R_k$ is a finite étale extension of $R\otimes_{\mathbb F_p}R$, so that $R_k[S^{-1}]$ is a finite étale extension of $R$. But via $R_k\to R_\infty \to R$, this is an étale extension which admits a section, and thus a splitting, and this splitting gives an $R$-algebra isomorphism $R_k[S^{-1}]\cong R\times T$, for some $R$-algebra $T$. Let $e_k$ be the corresponding idempotent in $R_k[S^{-1}]$. It follows that $R_\infty[(S\cup\{e_k, k\in\mathbb N\})^{-1}]\cong R$ via the canonical map $R_\infty\to R$. 

Thus, to conclude, it suffices to prove the following: if $(R,\mathbf G)\to (A,\mathbf H)$ be a morphism of formal groups over perfect $\mathbb F_p$-algebras which witnesses the target as a localization of the source at a set $S$ of elements, then there is a set of elements $\tilde S\subset \pi_0E(R,\mathbf G)$ such that $E(R,\mathbf G)\to E(A,\mathbf H)$ witnesses the target as the localization of the source at this set of elements (in $\Sp_{K(n)}$). 

But this follows from \cite[Theorems 5.1.5. and 5.4.1.]{lurie2018elliptic}: indeed, consider any subset $\tilde S\subset \pi_0E(R,\mathbf G)$ whose image under the surjective morphism $\pi_0E(R,\mathbf G)\to R$ is $S$, and let $T= E(R,\mathbf G)[\tilde S^{-1}]$ (computed in $\Sp_{K(n)}$). By \cite[Remark 4.1.10.]{lurie2018elliptic}, $T$ is complex periodic.  It is then clear that $T$ and $E(A,\mathbf H)$ have the same universal property in the \category{} of complex periodic $K(n)$-local commutative ring spectra. 

\end{proof}
\begin{rmk}
    When $R$ is strictly henselian, one can prove a converse to this theorem, that is, if $E(R,\mathbf G)$ is ind-separable in $\Sp_{K(n)}$, then $R$ is ind-separable over $\mathbb F_p$. It is also reasonable to expect a converse in general, but the author has not found a proof. 

    In particular, as there are strictly henselian rings that are not ind-separable over $\mathbb F_p$, one sees that the result is not true in full generality.
\pend\end{rmk}
 
\begin{rmk}
     We make a final note that the results in this subsection say nothing about the homotopy algebra structures on the spectra $E(R,\mathbf G)$. In particular, while for a given perfect field $k$, and a given height $n$, the various spectra $E(k,\mathbf G), \mathrm{ht}(\mathbf G) = n$ are homotopy equivalent \cite{lueckepeterson}, they are not equivalent as ring spectra if the formal groups are not isomorphic. But this is already the case at the level of homotopy algebras (the formal group only depends on the homotopy algebra structure). 
\pend\end{rmk}
\section{Examples}\label{section:examples}
We take a bit of time away from theory to look at some examples of separable algebras. All the examples we mention here are fairly standard. We begin with Galois extensions, and then move on to certain ``cochain algebras'' which appear among other places in equivariant stable homotopy theory, and can be organized through \categories{} of spans. We later go to the setting of group rings under certain assumptions on the ``cardinality'' of the group - these appear among other places in ambidexterity theory, and can also be organized through \categories{} of spans. We later mention examples related to algebraic geometry, namely we recall that étale maps of schemes induce separable algebras, and that (certain) Azumaya algebras are separable. Finally, we conclude with a non-example, by pointing out that separability is really a ``linear'' story, namely that there are no interesting examples in cartesian cases.

\begin{warn}

In the cases of \categories{} of spans, it is convenient to use $(\infty,2)$-categorical technology to organize the proofs that the relevant algebras are separable, by going through the $(\infty,2)$-category of correspondences. However, some of this technology has not been developed yet, and is only really known in the case of $2$-categoies. The reader can thus view these examples as either sketches (``a complete proof is left to the reader''), conjectures, or as proving less than what we claim, in the following sense : our proofs will still be valid at the homotopy category level, because there we only need the $2$-categorical version of the aforementioned technology. We note that because of the results of the previous sections, for most purposes, this is not a real restriction: as long as one maps those span categories to an \emph{additive} symmetric monoidal \category{}, homotopy separability guarantees full-fledged separability. 

We will indicate with a (*) the statements that are subject to this warning. 
\pend \end{warn}
\subsection{Galois theory}\label{subsection:galoisex}
In this section, we review one of the main examples of separability, namely Galois extensions. Originally introduced in field theory, they were later studied in the more general context of commutative rings \cite{AuslanderGoldman}, and later, by work of Rognes \cite{rognes}, in the setting of commutative ring spectra. His definition extends verbatim to more general stable homotopy theories. We recall the definition for the convenience of the reader: 
\begin{defn}
Let $\C$ be a cocompletely, stably symmetric monoidal \category{}, and let $A\in\CAlg(\C)$. For an $\mathbb E_1$-group $G$, an object $B\in\CAlg(\Mod_A)^{BG}$ is called a $G$-Galois extension of $A$ if: 
\begin{itemize}
    \item The induced map $A\to B^{hG}$ is an equivalence (of commutative algebras);
    \item the natural map $B\otimes_A B\to F(G_+,B)$, adjoint to the action map $A[G]\otimes_A B\otimes_A B\to B\otimes_A B\to B$, is an equivalence (informally, this map is given by $x\otimes y\mapsto (g\mapsto g(x)y)$). 
\end{itemize}
\pend \end{defn}
When $G$ is a discrete group, $F(G_+,B) \simeq \prod_G B$ and the multiplication map $B\otimes_A B\to B$ becomes identified with evaluation at $e\in G, \prod_G B\to B$. In particular, this clearly has a section as $\prod_G B$-modules, and we obtain:
\begin{prop}[{\cite[Lemma 9.1.2.]{rognes}}]\label{ex:galoissep}
Let $G$ be a discrete finite group, and $A\to B$ a $G$-Galois extension in $\CAlg(\C)$. In this case, $B$ is a separable $A$-algebra. 
\end{prop}
\begin{rmk}
    In the case of a Galois extension, a proof of \Cref{thm : commlift} was already sketched by Mathew in \cite[Theorem 6.25]{Akhilgalois}. 
\pend \end{rmk}
\begin{ex}
    Any Galois extension of fields $K\to L$ is $\mathrm{Gal}(L/K)$-Galois. More generally, Galois extensions of commutative rings are Galois, this follows from \cite[Proposition 2.3.4.(c)]{rognes}.
    \pend \end{ex}
\begin{ex}
    Profinite Galois extension in the sense of \cite[Definition 8.1.1]{rognes} are in general only ind-separable, cf. \Cref{section:indsep}. 
\pend \end{ex}
We conclude this subsection with examples of Galois extensions which are \emph{absolutely} separable in the following sense:
\begin{defn}\label{defn:abssep}
Suppose $\C$ admits geometric realizations compatible with the tensor product. 
Let $A\in\CAlg(\C)$, and let $B\in \Alg(\Mod_A)$ be an $A$-algebra. We say $B$ is \emph{absolutely separable} if it is separable over $A$, and furthermore if it has a separability idempotent which factors as $\one\to B\otimes B\op\to B\otimes_A B\op$. 
\pend \end{defn}
\begin{rmk}
    We do not require the lift $\one\to B\otimes B\op$ to be an idempotent. 
\pend \end{rmk}
This condition implies that the results of \Cref{section:descent} apply: we will see specifically that it implies descent in topological Hochschild homology (\Cref{cor:HHdescent}).
\begin{ex}
If $A\to B$ is a $G$-Galois extension between connective ring spectra, then $B\otimes B\to B\otimes_A B$ is surjective on $\pi_0$, and so $A\to B$ is absolutely separable. 
\pend \end{ex}
\begin{ex}\label{ex:Adamssummand}
Let $L^\wedge_p$ be the Adams summand of $KU^\wedge_p$. The canonical map $L^\wedge_p\to KU^\wedge_p$ is an $\mathbb F_p^\times$-Galois extension in the \category{} of $p$-complete spectra \cite[5.5.2]{rognes}, and the idempotent comes from a splitting of $KU_{(p)}$ \cite[5.5.4]{rognes} so it already lives in $KU^\wedge_p\otimes KU^\wedge_p$. 
\pend \end{ex}
\begin{ex}\label{ex:KOKU}

The $C_2$-Galois extension $\KO\to \KU$ \cite[Proposition 5.3.1]{rognes} also witnesses $\KU$ as an absolutely separable extension of $\KO$. Indeed, consider the fiber sequence of $\KO$-modules $\KO\to \KU\to \Sigma^2\KO$. If we smash it with $\KU$ over $\KO$ we get the following commutative diagram \cite[Diagram (5.3.3)]{rognes}: 
$$\xymatrix{\KU\ar[r] \ar[d]^{=} & \KU\otimes_{\KO}\KU \ar[r]\ar[d] & \Sigma^2 \KU \ar[d]^\beta\\
\KU \ar[r] & \prod_{C_2}\KU \ar[r] & \KU}$$

If we smash it with $\KU$ over the sphere, we get the following, still commutative comparison diagram : 
$$\xymatrix{\KU\otimes \KO \ar[r]\ar[d] & \KU\otimes \KU \ar[r]\ar[d] & \KU\otimes\Sigma^2 \KO\ar[d] \\ \KU\ar[r] \ar[d]^{=} & \KU\otimes_{\KO}\KU \ar[r]\ar[d] & \Sigma^2 \KU \ar[d]^\beta\\
\KU \ar[r] & \prod_{C_2}\KU \ar[r] & \KU}$$

We wish to know if the composite $\KU\otimes \KU\to \prod_{C_2}\KU$ hits $(1,0)$ on $\pi_0$. For this, note that $\KU\otimes \KO\to \KU$ has a splitting given by the unit of $\KO$, and so does $\KU\otimes\Sigma^2\KO\to \Sigma^2 \KU$. 

Furthermore, note that $\KU\simeq \KO\otimes \Sigma^{-2}\mathbb CP^2$ and we can write the top rightmost vertical morphism as $\KO\otimes (\KU\otimes \Sigma^{-2}\mathbb CP^2\to \KU\otimes \Sigma^{-2}S^4)$. Furthermore, although $\mathbb CP^2\to S^4$ does not have a splitting, $\KU\otimes \mathbb CP^2\to \KU\otimes S^4$ does (indeed, the cofiber is $\KU\otimes S^4\to \KU\otimes S^3$, and this is a $\KU$-module map which must therefore be $0$ for degree reasons). 

It follows that there is a map $\Sigma^2\KU\to \KU\otimes \KU$ making the following diagram commute: 
$$\xymatrix{& \Sigma^2 \KU \ar[d] \ar[ld] \\
\KU\otimes \KU\ar[r] & \KU\otimes\Sigma^2\KO}$$
Now, in $\pi_0(\KU\otimes \KU)$, consider the image of $\beta^{-1}$ by this map. If you go down-down-right in the big diagram, it's the same as right-down-down, which gives you $1\in \pi_0(\KU)$. So the image in $\prod_{C_2}\KU$ must have been of the form $(n+1,n) = (1,0) + (n,n)$. 

Consider now the image of $n$ under $\KU\to \KU\otimes \KO\to \KU\otimes \KU$ on $\pi_0$. Going down-down gives you exactly $(n,n)$, and so $(1,0)$ is in the image of $\pi_0$, this is what we wanted to prove. 
\pend \end{ex}
\subsection{Spans and equivariant stable homotopy theory}
In this subsection and the next, we will deal with span categories. For an account, see \cite{barwick,mackeyII} \footnote{Where the \category{} of spans is called the ``effective Burnside category''}. For the proofs, it will also be convenient to use the $(\infty,2)$-categories of correspondences that extend them \cite{stefanich},\cite{macpherson}. 

One reason to be interested in span categories is their relation to equivariant stable homotopy theory: the category of genuine $G$-spectra, $\Sp_G$, can be described as the category of spectral Mackey functors, i.e. direct sum preserving functors $\Span(\Fin_G)\op\to \Sp$. 

In \cite{restrictionétale}, Balmer, Dell'Ambrogio and Sanders describe, for a subgroup $H\leq G$, the category $\Sp_H$ as the category of modules over some algebra $A^G_H \in\CAlg(\Sp_G)$ which they prove is separable. In particular, all their work at the level of homotopy categories works at the level of stable $\infty$-categories by \Cref{section:homotopycat}. 

Note that $A_H^G$ is the image under the symmetric monoidal Yoneda embedding $\Span(\Fin_G)\to \Sp_G$ of an algebra \emph{in} $\Span(\Fin_G)$. The object of this subsection is to prove that this algebra is already separable there. Note that $\Span(\Fin_G)$ is not additive, so we cannot apply \cite[Theorm 1.1]{restrictionétale} directly and work in $\ho(\Span(\Fin_G))$, where the result is simpler to prove. 

More generally, we prove
\begin{thm}[*]\label{thm:spancoalg}
Let $C$ be a small category with finite limits, and $X\in C$. We view $X$ as a commutative algebra in $C\op$, and thus, $X^\vee$ as a commutative algebra in $\Span(C)$ under the canonical symmetric monoidal functor $C\op\to \Span(C)$.

If the evaluation map from the cotensoring $X^{S^1}\to X$ is an equivalence, then $X^\vee$ is a separable commutative algebra in $\Span(C)$. 
\end{thm}
\begin{rmk}
This applies in particular if $C$ is a $1$-category such as $\Fin_G$. 
\pend \end{rmk}
In the course of this proof, we use the following:
\begin{conj}\label{conj:rightlaxlinear}
Let $\mathfrak B$ be a symmetric monoidal $(\infty,2)$-category, $A\in\Alg(\iota_1\mathfrak B)$ an algebra in (the underlying $(\infty,1)$-category of) $\mathfrak B$, $M,N$ $A$-modules in $\mathfrak B$, and $f: M\to N$ an $A$-module map. If $f$ admits a right adjoint $f^R$, and the square: 
\[\begin{tikzcd}
	{A\otimes M} & {A\otimes N} \\
	M & N
	\arrow[from=1-1, to=2-1]
	\arrow["f", from=2-1, to=2-2]
	\arrow[from=1-2, to=2-2]
	\arrow["{A\otimes f}", from=1-1, to=1-2]
\end{tikzcd}\]
is horizontally right-adjointable, then $f^R$ is canonically $A$-linear; and more precisely $f$ admits a right adjoint in $\Mod_A(\mathfrak B)$.
\end{conj}
We note that in the case where $\mathfrak B$ is the $(\infty,2)$-category of \categories{}, this conjecture is essentially proved in \cite[Remark 7.3.2.9]{HA}. 
\begin{rmk}
    This conjecture should also have a more general form, similarly to the calculus of mates in \cite{HHLN}. Namely, in the above, if we only assume that $f$ admits a right adjoint $f^R$, then this right adjoint should be canonically \emph{lax} $A$-linear, that is, come with suitably compatible and coherent maps ``$a\otimes f^R(m)\to f^R(a\otimes m)$'', and it should then be a property (namely, adjointability) that these maps are equivalences. Conversely, the left adjoint of a lax $A$-linear morphism should always be oplax $A$-linear, and this should be a perfect correspondence between oplax $A$-linear left adjoints, and lax $A$-linear right adjoints. In the case $\mathfrak B= \Cat$, this can be deduced from \cite{HHLN}, but below we need it for $\mathfrak B$ being an $(\infty,2)$-category of correspondences.  
\pend \end{rmk}
As explained in the introduction to this section, this conjecture is well-known (and classical) in the case of $2$-categories, so the arguments that we give apply unconditionally to the homotopy category $\ho(\Span(C))$, and thus to any additive \category{} $\C$ with a symmetric monoidal map $\Span(C)\to \C$. 

\begin{proof}
We use the $(\infty,2)$-category of correspondences, $\Corr(C)$, see \cite{stefanich}, \cite{macpherson}. In particular, its underlying \category{} is $\Span(C)$. 

We note that the multiplication map of $X$ is given by the span $X\times X \xleftarrow{\Delta} X \xrightarrow{=}X$. We note that, as a morphism in $\Corr(C)$, it admits a left adjoint, cf. \cite{stefanich}. Because the multiplication map is $X\times X$-linear, it is a \emph{property} that this left adjoint is actually $X\times X$-linear, namely that the square from be left adjointable (by the dual of \Cref{conj:rightlaxlinear}). 

Let us assume for now that we have checked this - the composite is then the composite of spans $$X\xleftarrow{=}X \xrightarrow{\Delta} X\times X \xleftarrow{\Delta} X \xrightarrow{=}X$$
which is easily seen to be given by the span $X\xleftarrow{ev} X^{S^1}\xrightarrow{ev}X$. Our assumption guarantees that this is an equivalence, hence an equivalence of $X\times X$-modules, and so up to composing by its inverse, we find that $X$ is separable as an algebra. 

Let us now check the property : we need to check that the oplax-$X$-linear structure maps are strict, we do it for the left-$X$-linear one, and the right-$X$ linear one follows by symmetry. The left $X$-linearity of the multiplication map is given by the following commutative diagram in $\Corr(C)$: 
\[\begin{tikzcd}
	{X\times X\times X} & {X\times X} \\
	{X\times X} & X
	\arrow["\mu", from=2-1, to=2-2]
	\arrow["{\mu\times X}"', from=1-1, to=2-1]
	\arrow["\mu"', from=1-2, to=2-2]
	\arrow["{X\times \mu}", from=1-1, to=1-2]
\end{tikzcd}\]
In this diagram, all maps are in $C\op$, so this is just the image under $C\op\to \Corr(C)$ of the canonical coassociativity diagram for $X$, and this canonical coassociativity diagram is a pullback square: 
\[\begin{tikzcd}
	{X\times X\times X} & {X\times X} \\
	{X\times X} & X
	\arrow["\Delta"', from=2-2, to=2-1]
	\arrow["{\Delta\times X}", from=2-1, to=1-1]
	\arrow["\Delta", from=2-2, to=1-2]
	\arrow["{X\times \Delta}"', from=1-2, to=1-1]
\end{tikzcd}\]
In particular, it is adjointable in $\Corr(C)$ e.g. by \cite{stefanich},\cite{macpherson}, so we are done. 
\end{proof}

\subsection{Spans, ambidexterity and Thom spectra}
In this subsection, we study a situation similar to the one of the previous subsection, except that we start with a monoid $G$ in $C$, and view it as a monoid in $\Span(C)$. 

The result that we prove is:
\begin{thm}[*]\label{thm:spansepgroup}
Let $f: \Span(C)\to \C$ be a symmetric monoidal functor, and suppose it sends the span $\pt \leftarrow G\to \pt$ to an equivalence. Then $f(G)$ is a separable algebra in $\C$. 
\end{thm}
\begin{ex}\label{ex:sepchar0}
Consider the case where $C=\Fin$, the category of finite sets. In this case, $\Span(\Fin)$ is the initial semiadditively symmetric monoidal \category. In particular, for any semiadditively symmetric monoidal $\C$, there is an essentially unique symmetric monoidal, semiadditive functor $\Span(\Fin)\to \C$. It sends a finite set $X$ to $\bigoplus_X \one$. 

In this case, a $G$ as in the theorem is simply a finite group. The theorem is saying that if its order $|G|$ is invertible in $\C$, then $\one[G]$ is separable. This is typical from classical algebra: the group algebra $\mathbb Q[G]$ is always separable, and  more generally, for a field $k$, $k[G]$ is separable over $k$ if and only if $|G|\in k^\times$.
\pend \end{ex}
A generalization of the previous example, and our motivating example for this section, comes from the theory of higher semi-additivity, cf. \cite{yonatan,AmbiChro}. This is a context where one can sum not only over finite sets, but also over finite groupoids, or more generally, $m$-finite spaces, i.e. spaces $X$ with finitely many components, and with, at every point, $\pi_k(X)=0$ for $k>m$ (or possibly only the $m$-finite spaces, all of whose homotopy groups are $p$-groups, for some fixed prime $p$). We refer to the above references for a more detailed account of this theory. We let $\Ss_m^{(p)}$ denote the \category{} of $m$-finite spaces all of whose homotopy groups are $p$-groups.

In that case, when $\C$ is ($p$-typically) $m$-semiadditive \cite[Definition 3.1.1]{AmbiHeight}, there is a unique symmetric monoidal functor $\Span(\Ss_m^{(p)})\to \C$ which preserves $p$-typical $m$-finite colimits, which we denote by $\one[-]$. The span $\pt\leftarrow G\to \pt$ is sent to the \emph{cardinality} $|G|_\C$ of $G$, as a morphism $\one\to \one$. The property that this be an equivalence is related to the so-called \emph{semi-additive height} of $\C$.  For example, ``height $0$'' corresponds to the rational case, where all these cardinalities are invertible. Higher heights are also related to chromatic height - we refer to \cite{AmbiHeight} for more details. 

\begin{proof}
The proof again makes use of the higher categorical structure of $\Corr(C)$. Just as before, we observe that $\mu : G\times G\to G$ has a right adjoint, and by \cref{conj:rightlaxlinear}, it is simply a property for it to be $G\times G\op$-linear, which we can check in the exact same way as in the proof of \Cref{thm:spancoalg}. The key point is that the associativity diagram for $G$ in $C$ (which is also the ``left $G$-linearity'' diagram) is a pullback diagram in $C$: 
\[\begin{tikzcd}
	{G\times G\times G} & {G\times G} \\
	{G\times G} & G
	\arrow["{G\times \mu}"', from=1-1, to=2-1]
	\arrow["{\mu\times G}", from=1-1, to=1-2]
	\arrow["\mu"', from=2-1, to=2-2]
	\arrow["\mu", from=1-2, to=2-2]
\end{tikzcd}\]
and hence, it is adjointable in $\Corr(C)$. This is exactly what we need for the adjoint of $\mu$ to be $G\times G\op$-linear.

Now, this gives us a $G\times G\op$-linear morphism $G\to G\times G\op$ in $\Span(C)$. The composition $G\to G\times G\op\to G$ is given by the span $G\xleftarrow{\mu}G\times G\xrightarrow{\mu}G$, and as a morphism in $\Span(C)$, this is equivalent to $G\xleftarrow{pr_1}G\times G\xrightarrow{pr_1}G$ because of the shear map $G\times G\to G\times G$. We can rewrite the latter span as $(\pt\leftarrow G\to \pt)\times G$. The claim now follows in the same way: up to inverting the span $\pt\leftarrow G\to \pt$, we have a separability idempotent.  
\end{proof}
\begin{ex}
    In \cite[Definition 4.7]{CycloChro}, the authors introduce, for any stable $\infty$-semiadditive presentably symmetric monoidal \category{} $\C$ a height $n$ $p^r$th-cyclotomic extension $\one[\omega^{(n)}_{p^r}]$, which is a higher height analogue of the usual cyclotomic extensions. 

    This cyclotomic extension is defined as the splitting of an idempotent on $\one[B^nC_{p^r}]$ and the definition of ``height $n$'' guarantees that $|B^nC_{p^r}|$ is invertible in $\C$, in other words, that the previous theorem applies. So $\one[B^nC_{p^r}]$ is separable, and hence so is $\one[\omega^{(n)}_{p^r}]$. This shows that, even if it is not always Galois (cf. \cite[Proposition 3.9]{yuan}), it is separable, which is a notion not too far from ``étale'' in the commutative setting. 
\pend \end{ex}
For a group $G$, the group algebra $\one[G]$ can be seen as the colimit of the constant diagram with value $\one$, indexed by $G$. We saw in \Cref{ex:sepchar0} that when $|G|$ is invertible, this algebra is separable - we now describe a slight extension of this result, namely to Thom objects. First, we recall the following construction: 
\begin{cons}
Let $X$ be a space, and $f:X\to\Pic(\C)$ a map, where $\Pic(\C)\subset \C^\simeq$ is the maximal subgroupoid spanned by the invertible objects in $\C$. One may take the colimit of the composite $X\to \Pic(\C)\to \C$, if it exists. 

If $\C$ is, say, cocomplete, this corresponds to the unique colimit-preserving functor $\Ss_{/\Pic(\C)}\to \C$ which restricts to the canonical inclusion along the Yoneda embedding $\Pic(\C)\to \Ss_{/\Pic(\C)}\to \C$.As a consequence, this functor $\Ss_{/\Pic(\C)}\to \C$ is symmetric monoidal, so it sends groups $G$ equipped with a group map $G\to \Pic(\C)$ to an algebra object in $\C$. 
\pend \end{cons}
\begin{prop}[*]
    Assume $\C$ is $m$-semiadditive for some $0\leq m\leq\infty$. The above construction extends uniquely to an $m$-semiadditive, symmetric monoidal functor $\Span((\Ss_m)_{/\Pic(\C)})\to \C$. 

    In particular, if $f:G\to \Pic(\C)$ is a group map from an $m$-finite group $G$, where $|G|_\C$ is invertible in $\C$, then its Thom object $\colim_G f$ is a separable algebra in $\C$. 
\end{prop}
\begin{proof}
    The ``in particular'' part follows from \Cref{thm:spansepgroup}, together with the observation that the span $(\pt,\one)\leftarrow (G,\one)\to (\pt,\one)$ is indeed sent to $|G|_\C$ in $\C$. 

    Now, for the first part, namely the existence of the map, we use \cite[Theorem 5.28]{yonatan} in the special case where $\mathcal C = \Pic(\C)$. We note that the canonical symmetric monoidal structure on $\Span((\Ss_m)_{/X})$, when $X$ is a symmetric monoidal $\infty$-groupoid, is the one induced by the universal property of \cite[Theoem 5.28]{yonatan} because the natural map $X\to\Span((\Ss_m)_{/X})$ is symmetric monoidal for this symmetric monoidal structure. 

    The cited theorem thus implies that a symmetric monoidal map $X\to \C$ (here the inclusion $\Pic(\C)\subset \C$) extends essentially uniquely to an $m$-semiadditive symmetric monoidal functor $\Span((\Ss_m)_{/X})\to \C$. 
\end{proof}
\begin{ex}
    In the case $m=0$, an $m$-finite group is simply an ordinary finite group, and if furthermore every point in $G$ is sent to the unit $\one\in \Pic(\C)$, then the Thom object is simply a twisted group ring $\one_\alpha[G]$. 
\pend \end{ex}
\begin{ex}\label{ex:exsecondtime}
    The algebra from \Cref{ex:spacenotequiv} is an example of this construction. Indeed, let $\D= \Mod_{\mathbb Q[t^{\pm 1}]}$ with $t$ in degree $2d$ for some odd $d\neq 1$. We let $G=H\rtimes \mathbb Z/d$ as in \Cref{ex:spacenotequiv}, and $G\to \Pic(\D)$ is the map $G\to \mathbb Z/d\to \Pic(\D)$, where the latter map picks out $\Sigma^2\mathbb Q[t^{\pm 1}]$.  Let us briefly explain why $\mathbb Z/d\to \Pic(\D)$ can be made into a map of commutative groups. This picard element is clearly classified by a map $\mathbb S\to \Pic(\D)$, and because it is $d$-torsion, by a map $\mathbb S/d\to \Pic(\D)$. The homotopy groups of $\Pic(\D)$ are rational above $\pi_2$, and the homotopy groups of $\mathbb S/d$ are finite, so this map canonically factors through $\tau_{\leq 1}(\mathbb S/d)$, which is $\mathbb Z/d$ because $d$ is odd.

    Now colimits over $G,\mathbb Z/d$ are just coproducts, so it is easy to check that the algebra structure in the homotopy category of $\D$ is the one we described in \Cref{ex:spacenotequiv}. Because $|G|$ and $|\mathbb Z/d|$ are invertible in $\D$, we find that these algebras are indeed separable (note that this does not depend on \Cref{conj:rightlaxlinear} because $\D$ is additive).  
\pend \end{ex}
Along the way, we record the following result we have sketched in the previous example (cf. also \cite[Example 2.30]{lawsonroots} and the surrounding discussion):
\begin{lm}
    Let $\D$ be a symmetric monoidal \category{}, and $L\in\Pic(\D)$ be an invertible element with $L^{\otimes d}\simeq \one_\D$. Assume that $d$ is odd, and invertible in $\pi_*\map(\one_\D,\one_\D), *\geq 1$. The space of maps of commutative groups $\mathbb Z/d\to \Pic(\D)$ classifying $L$ is equivalent to the space of equivalences $L^{\otimes d}\simeq \one$. 
\end{lm}
\subsection{Twists of Morava {$K$}- and {$E$}-theories}\label{subsection:SW}
In this subsection, we recover the main results of \cite{satiwesterland} (namely \cite[Theorems 1.1 and 1.2]{satiwesterland}), and correct along the way \cite[Theorem 1.2]{satiwesterland}, as well as get rid of any need for obstruction theory. For the convenience of the reader, we recall these theorems:
\begin{thm}[{\cite[Theorem 1.1]{satiwesterland}}]\label{thm:SWK}
    Let $n\geq 1$ and $K(n)$ be Morava $K$-theory at height $n$ and an implicit prime $p$. The canonical map is an equivalence:$$\map_*(K(\mathbb Z,n+2),BGL_1(K(n)))\to \hom_{\Alg(K(n)_*)}(K(n)_*K(\mathbb Z,n+1), K(n)_*)$$
\end{thm}
We warn the reader that, as stated, the following theorem contains a mistake, which we correct later:
\begin{thm}[{\cite[Theorem 1.2]{satiwesterland}}]\label{thm:SWE}
   Let $n\geq 1$ and $E_n$ be Morava $K$-theory at height $n$ and an implicit prime $p$. The canonical maps are equivalences:$$\map_{\CAlg(\Ss)}(K(\mathbb Z,n+2),BGL_1(E_n))\to \map_{\Ss_*}(K(\mathbb Z,n+2),BGL_1(E_n))\to \hom_{\Alg((E_n)_*)}((E_n)_*K(\mathbb Z,n+1), (E_n)_*)$$
\end{thm}

Their proof relies on a computation of the Morava $K$-theories of Eilenberg-MacLane spaces due to Ravenel and Wilson \cite{ravenelwilson}, together with Goerss--Hopkins obstruction theory. We aim to explain how one can get rid of the latter, and recover this result based only on Ravenel and Wilson's calculations. 
    \begin{warn}\label{warn:mistakeSW}
    Our proof proceeds by proving that $E_n[K(\mathbb Z/p^k,n)]$ is separable in $K(n)$-local spectra -- in particular, it involves $K(n)$-localization, and thus we note that \cite[Theorem 1.2]{satiwesterland} is wrong as stated. Namely, the correct equivalence is 
      $$\map_{\CAlg(E_n)}(E_n[K(\mathbb Z,n+1)],E_n)\simeq \hom_{\CAlg(\pi_*(E_n))}((E_n)^\vee_*K(\mathbb Z,n+1),\pi_*(E_n))$$
      where $(E_n)^\vee_*(X) := \pi_*(L_{K(n)}(E_n\otimes X))$ is completed Morava $E$-theory. One can prove that without completion, this equivalence does not hold. For example if $n$ is even (so that $n+1$ is odd, and $H_*(K(\mathbb Z,n+1);\mathbb Q) =\mathbb Q[\epsilon], |\epsilon| = 1$), it is not so hard\footnote{The proof uses the rational computation together with the fact that $(E_n)_*$ is torsion free and concentrated in even degrees.} to prove that $\hom_{\CAlg(\pi_*(E_n))}((E_n)_*K(\mathbb Z,n+1), (E_n)_*) \cong \hom_{\CAlg(\pi_*(E_n))}((E_n)_*, (E_n)_*) \cong \pt$. On the other hand, with completed Morava $E$-theory, one does find a set isomorphic to $\mathbb Z_p$, as claimed in \cite{satiwesterland}. 
    \end{warn}
    Thus, the corrected version of \Cref{thm:SWE} is: 
    \begin{thm}[{\cite[Theorem 1.2]{satiwesterland}}]\label{thm:SWEcorrect}
   Let $n\geq 1$ and $E_n$ be Morava $K$-theory at height $n$ and an implicit prime $p$. The canonical maps are equivalences:$$\map_{\CAlg(\Ss)}(K(\mathbb Z,n+2),BGL_1(E_n))\to \map_{\Ss_*}(K(\mathbb Z,n+2),BGL_1(E_n))\to \hom_{\Alg((E_n)_*)}((E_n)^\vee_*K(\mathbb Z,n+1), (E_n)_*)$$
\end{thm}
We prove \Cref{thm:SWK}, indicating along the way the necessary changes for \Cref{thm:SWEcorrect}. 
\begin{proof}
Let $M_K:=\map_*(K(\mathbb Z,n+2),BGL_1(K(n)))$, $M_E^{d+1}:=\map_{\Alg_{\mathbb E_d}(\Ss)}(K(\mathbb Z,n+2),BGL_1(E_n))$ (including for $d=0$ and $\infty$).
    We begin by noting that $$M_K\simeq \map_{\Alg(\Ss)}(K(\mathbb Z,n+1),GL_1(K(n)) \simeq \map_{\Alg(\Sp_{K(n)})}(\Sph_{K(n)}[K(\mathbb Z,n+1)],K(n))$$. In the case of $E$-theory, one can similarly move to $$M_E^d\simeq \map_{\Alg_{\mathbb E_d}(\Mod_{E_n}(\Sp_{K(n)}))}(E_n[K(\mathbb Z,n+1)],E_n)$$. 

    Now, $K(\mathbb Z,n+1)$ is $p$-adically equivalent to $\colim_k K(\mathbb Z/p^k,n)$ so that we can rewrite our space as inverse limits of similar spaces with $K(\mathbb Z/p^k,n)$ in place of $K(\mathbb Z,n+1)$. By \cite[Lemma 5.3.3.]{AmbiHeight}, we may apply \Cref{thm:spansepgroup} to $G=K(\mathbb Z/p^k,n)$ and $\C=\Sp_{K(n)}$ (resp. $\Mod_{E_n}(\Sp_{K(n)})$) and obtain that $\Sph_{K(n)}[K(\mathbb Z/p^k,n)]$ (resp. $E_n[K(\mathbb Z/p^k,n)]$, computed in $\Sp_{K(n)}$) is separable as a $K(n)$-local ($E_n$-)algebra\footnote{For the connection to \Cref{thm:spansepgroup}, see the discussion following \Cref{ex:sepchar0}.}. 

    In the case of $E$-theory, we can directly conclude that $\map_{\Alg_{\mathbb E_d}(E_n)}(E_n[K(\mathbb Z/p^k,n)],E_n)\simeq \hom_{\CAlg(\ho(\Mod_{E_n}))}(E_n[K(\mathbb Z/p^k,n)],E_n)$ by \Cref{thm : commmorsep}. By \cite[Proposition 3.4.3., Proposition 2.4.10.]{hopkinslurieambi}, $\pi_*(E_n[K(\mathbb Z/p^k,n)])$ is a free $(E_n)_*$-module, so that $$\hom_{\CAlg(\ho(\Mod_{E_n}))}(E_n[K(\mathbb Z/p^k,n)],E_n)\cong \hom_{\CAlg((E_n)_*)}(\pi_*(E_n[K(\mathbb Z/p^k,n)]), (E_n)_*)$$
   Thus $$M_E^d\simeq \hom_{\Alg((E_n)_*)}(\pi_*(\colim_k L_{K(n)}(E_n[K(\mathbb Z/p^k,n)])), (E_n)_*)$$
   Finally, $(E_n)_*$ is $\mathfrak m$-adically complete (where $\mathfrak m = (p,v_1,...,v_{n-1})$ is the maximal ideal in the local ring $\pi_0(E_n)$), and $\colim_k L_{K(n)}(E_n[K(\mathbb Z/p^k,n))$ is concentrated in even degrees, so that the $\mathfrak m$-adic completion of its homotopy groups is the same as the homotopy groups of its $K(n)$-localization, and so we get:
   $$M_E^d\simeq \hom_{\Alg((E_n)_*)}(\pi_*( L_{K(n)}(E_n[K(\mathbb Z,n+1)])), (E_n)_*)$$ which was to be proved. 

   In the case of $K$-theory, the approach is the same but more care must be taken at the prime $2$, as Morava $K$-theory is not homotopy commutative\footnote{Since we are working over the sphere, there is only one $\mathbb E_1$-structure on Morava $K$-theory, and it is homotopy commutative at odd primes. }.
   
   In the case of odd primes, Morava $K$-theory is homotopy commutative, and so by combining \Cref{thm : morsep} and \Cref{cor:hocommtargetdiscrete} we obtain $M_K\simeq \hom_{\Alg(\ho(\Sp))}(\mathbb S_{K(n)}[K(\mathbb Z,n+1)],K(n))$. In $\ho(\Sp)$, $K(n)$ is now a commutative algebra, so this is equivalent to $$\hom_{\Alg(\Mod_{K(n)}(\ho(\Sp)))}(K(n)[K(\mathbb Z,n+1)],K(n))$$ and $\Mod_{K(n)}(\ho(\Sp))$ is monoidally equivalent to $\Mod_{K(n)_*}(\mathbf{GrVect}_{\mathbb F_p})$\footnote{Every (homotopy) $K(n)$-module is free up to shifts.} so that the latter is equivalent to $\hom_{\Alg(K(n)_*)}(K(n)_*K(\mathbb Z,n+1),K(n)_*)$, as claimed. 

   The case of the prime $2$ is a bit more subtle, but it can be approached using the work of Würgler \cite{wurgler}. More specifically, \cite[Proposition 2.4, Remark 2.6.(b)]{wurgler} shows that there is a map $Q: K(n)\to \Sigma^{2^n-1}K(n)$ such that the multiplication map $\mu: K(n)\otimes K(n)\to K(n)$ differs from its twist $K(n)\otimes K(n) \overset{\tau}{\simeq} K(n)\otimes K(n)\to K(n)$ by $v_n\cdot \mu\circ (Q\otimes Q)$. 

   Note that $Q$ is of odd degree. Ravenel and Wilson's computation \cite{ravenelwilson} shows, in particular, that $K(n)^*(K(\mathbb Z,n+1))$ is concentrated in even degrees, so that for any map $f: \Sph[K(\mathbb Z,n+1)]\to K(n)$, the composition $\Sph[K(\mathbb Z,n+1)]\to K(n)\to \Sigma^{2^n-1}K(n)$ is $0$ ($2^n-1$ is odd as $n\geq 1$). In particular, for any such map, the composites $\Sph[K(\mathbb Z,n+1)]\otimes K(n)\to K(n)\otimes K(n)\to K(n)$ and $\Sph[K(\mathbb Z,n+1)]\otimes K(n)\simeq K(n)\otimes \Sph[K(\mathbb Z,n+1)]\to K(n)\otimes K(n)\to K(n)$  agree. We may thus apply \Cref{prop:discretecentralhocomm} even though $K(n)$ is not homotopy commutative, and conclude in the same way as before, using \Cref{thm : morsep} to compute $\pi_0$. 
\end{proof}

\subsection{Scheme theory}
In \cite{balmerNT}, Balmer proves the following (compare \Cref{prop:etalesep}): 
\begin{thm}[{\cite[Theorem 3.5]{balmerNT}}]
    Let $f:V\to X$ be a separated étale morphism of quasicompact, quasiseparated schemes. In this case, $f_*\mathcal O_V$ is a separable algebra in $\mathrm{QCoh}(X)$. 
\end{thm}
As already mentioned, Neeman proved in \cite{Neeman} that, at least in the noetherian case, this is not far from exhausting all examples:
\begin{thm}[{\cite[Theorem 7.10]{Neeman}}]
    Let $X$ be a noetherian scheme and $A\in \mathrm{QCoh}(X)$ a commutative  separable algebra. There exists an étale morphism $g:U\to X$ and a specialization-closed subset $V\subset U$ such that $A\simeq g_*L_V\mathcal O_U$ of commutative algebras\footnote{Neeman only proves that this is an equivalence of algebras in the homotopy category, but \Cref{thm : commlift} tells us that this suffices.}.

    Here, $L_V$ is the Bousfield-localization of $\mathrm{QCoh}(U)$ associated to the specialization-closed subset $V$.
\end{thm}
In other words, up to idempotent algebras and under a noetherianity assumption, all commutative separable algebras come from étale maps. 

\subsection{Azumaya algebras}
In \Cref{section:brauer}, we will see that there is a strong connection between Azumaya algebras and separable algebras. We will prove that many Azumaya algebras are separable, specifically (cf. \Cref{prop:Azsep}):
\begin{prop}
Assume $\C$ is presentably symmetric monoidal.
Let $A\in \Alg(\C)$ be an algebra. If $A$ is Azumaya \emph{and} the unit $\eta : \one\to A$ admits a retraction, then $A$ is separable.
\end{prop}

We will recall the definition of Azumaya algebras in higher algebra in \Cref{section:brauer}. Doing so, we will along the way correct a mistake in \cite[Proposition 1.4]{BRS}, which states this result without the assumption that the unit splits - we will provide counterexamples to this statement, cf. \Cref{ex:MoravaK} and \Cref{ex:cofibeta}.

In classical algebra, the assumption on the unit is automatic, and we have:
\begin{thm}[{\cite[Theorem 2.1]{AuslanderGoldman}}]
    Let $R$ be an ordinary commutative ring. An Azumaya algebra in $\Mod_R^\heart$ is separable. 
\end{thm}

\subsection{Cartesian symmetric monoidal categories}
We conclude this Examples section with a situation where there are no interesting examples. The unit of a symmetric monoidal category is of course always separable, and we show:
\begin{prop}
Let $\C$ be cartesian symmetric monoidal. The only separable algebra in $\C$ is the unit, i.e. the terminal object. 
\end{prop}
\begin{proof}
As $\C\to\ho(\C)$ preserves products, and as the unit object has an essentially unique algebra structure, we may assume $\C$ is a $1$-category. Using the (classical) Yoneda embedding, we may even assume $\C= \Set$.\footnote{These reductions are purely \ae sthetic, the proof goes through more or less unchanged in the general case. } 

Let $M$ be a monoid with multiplication map $\mu : M\times M\to M$ and neutral element $\eta: \pt\to M$, which we assume to be separable, with section $s: M\to M\times M $. We write $s$ as $(s_1,s_2)$. Left $M$-linearity of $s$ guarantees that $s_2$ is constant. Indeed, for any $x\in M$, we have $(s_1(x),s_2(x)) = s(x) = s(x\cdot 1) = x\cdot s(1) = x\cdot (s_1(1), s_2(1)) = (x\cdot s_1(1), s_2(1)) $; and similarly right $M$-linearity guarantees that $s_1$ is constant. This proves that $s$ is constant and hence $\mu\circ s$ is, i.e. $\id_M$ is constant, from which it follows that $M =\pt$.
\end{proof}
\section{Auslander-Goldman theory}\label{section:brauer}
In \cite{AuslanderGoldman}, Auslander and Goldman lay the foundations of a systematic study of separable algebras (in classical algebra). One of the key results that they prove is the following: for a (discrete) commutative ring $R$, an $R$-algebra $A$ is separable if and only if its center $C$ is separable over $R$, and $A$ is Azumaya over its center. This allows one to reduce the study of general separable algebras to two special cases : the commutative case, which is closely related to étale algebras, and the central case, which is closely related to the theory of Azumaya algebras and the Brauer group. 

Our goal in this section is to raise two questions such that a positive answer to both, or at least reasonable conditions under which they have positive answers, would allow one to give a similar treatment of separable algebras in homotopical algebra. 

We separate this key result in two parts: first, the center $C$ of $A$ is separable over $R$ (and commutative), and second, $A$ is Azumaya over its center.

In homotopical algebra, the center $Z(A)$ of $A$ is in general an $\mathbb E_2$-algebra, but if it is separable, it is therefore canonically $\mathbb E_\infty$, i.e. commutative, by \Cref{thm : commlift}. This raises the following question (cf.  \cite[Theorem 2.3]{AuslanderGoldman}):
\begin{ques}\label{question : sepcenter}
Let $A\in \Alg(\C)$ be a separable algebra. Is its center $Z(A)$ separable too ? 
\end{ques}

The second key result is that $A$ is Azumaya over its center $Z(A)$. We start by offering a few recollections about Azumaya algebras, along the way correcting an error in \cite{BRS} about the relation between separable algebras and Azumaya algebras.  Once this is done, we can phrase the second main question of this section. 

We then start this section by answering the Azumaya question (\Cref{question : centralazumaya}) in certain cases; and we then attack \Cref{question : sepcenter}, again answering it in certain cases. For both questions, we fall short of answering it in the generality of $\Mod_R(\Sp)$, where $R$ is some commutative ring spectrum - this is essentially because we lack ``residue fields'', as will be clear from our discussion. 

\subsection{Azumaya algebras}
We start by recalling a possible definition of Azumaya algebras: 
\begin{defn}
Let $\C$ be presentably symmetric monoidal. An algebra $A\in\Alg(\C)$ is \emph{Azumaya} if $\LMod_A(\C)$ is invertible in $\Mod_\C(\PrL)$. 
\pend \end{defn}

We also recall several equivalent characterizations. For this, we need the following proposition/definition:
\begin{prop}[{\cite[Proposition 2.1.3., Corollary 2.1.4.]{hopkinsluriebrauer}}]
Let $\C$ be presentably symmetric monoidal. Let $M$ be a dualizable object of $\C$. The following are equivalent: 
\begin{enumerate}
    \item $M$ generates $\C$ under $\C$-colimits, that is, the smallest tensor ideal of $\C$ closed under colimits and containing $M$ is the whole of $\C$; 
    \item The ($\C$-linear) functor $\hom(M,-): \C\to \RMod_{\End(M)}(\C)$ is an equivalence; 
    \item $\End(M)$ is ($\C$-linearly) Morita equivalent to the unit $\one$; 
    \item $M\otimes -$ is conservative.
\end{enumerate}
If $M$ satisfies one (and hence all) of these properties, it is called \emph{full}. Furthermore, any (and hence all) of these properties are stable under passing to the dual $M^\vee$.
\end{prop}
\begin{proof}
We prove (1) $\implies$ (4) $\implies$(2) $\implies$ (3) $\implies$ (1).

Assume (1), and let $f: X\to Y$ be a map such that $M\otimes f$ is an equivalence. The collection of $Z$'s such that $Z\otimes f$ is an equivalence is certainly a tensor ideal of $\C$, closed under colimits, so that by (1), it contains $\one$. In particular, $f$ is an equivalence, thus proving (4). 

Let us now assume (4). The functor $G= \hom(M,-): \C\to \RMod_{\End(M)}$ preserves limits and colimits hence has a left adjoint $F= M\otimes_{\End(M)}-$. The unit map at $\End(M)$, $\End(M)\to \hom(M, M\otimes_{\End(M)}\End(M))$ is easily seen to be an equivalence, and  both the source and the target of the unit $\id\to GF$ are $\C$-linear and colimit-preserving, hence the unit is an equivalence at all $\End(M)$-modules. 

To prove that the counit is an equivalence, by the triangle identities, it thus suffices to show that the right adjoint $\hom(M,-)$ is conservative, and because the forgetful functor $\RMod_{\End(M)}\to \C$ is conservative, it suffices to show that $\hom(M,-) : \C\to \C$ is conservative. By dualizability, this is equivalent to $M^\vee\otimes -$. Now if $M^\vee\otimes f$ is an equivalence, so is $M\otimes M^\vee \otimes M\otimes f$; and thus, so is $M\otimes f$, as $M$ is a retract of $M\otimes M^\vee \otimes M$. By conservativity of $M$, it follows that $f$ is an equivalence, and hence $\hom(M,-)$ is conservative. This proves (2). 

(2) clearly implies (3), by definition of Morita equivalence. 

So let us now assume (3). The existence of a Morita equivalence yields a right $\End(M)$-module $X$ and a left $\End(M)$-module $Y$ such that $X\otimes_{\End(M)}Y\simeq \one$. The smallest $\C$-linear subcategory of $\RMod_{\End(M)}$ closed under colimits and containing $\End(M)$ contains $X$, so that the smallest $\C$-linear subcategory of $\C$ closed under colimits and containing $Y\simeq \End(M)\otimes_{\End(M)}Y$ also contains $\one$. But now $Y$ is a retract (in $\C$) of $\End(M)\otimes Y \simeq  M\otimes M^\vee\otimes Y$ so that (1) follows.  
\end{proof}
We also briefly need: 
\begin{defn}\label{defn:atomic}
Let $\C$ be presentably symmetric monoidal, and let $\M$ be a $\C$-module in $\PrL$. An object $x\in\M$ is called $\C$-\emph{atomic} if the canonical map $c\otimes\hom(x,y)\to \hom(x,c\otimes y)$ is an equivalence for all $c\in\C,y\in\M$, and $\hom(x,-)$ preserves all colimits. Here, $\hom$ denotes the $\C$-valued hom object of $\M$. 
\pend \end{defn}
\begin{rmk}
This definition appears in \cite[Definition 2.2]{BMS} in the case where $\C$ is a \emph{mode}, so that it actually suffices to assume that $\hom(x,-)$ preserves colimits, cf. \cite[Remark 2.4]{BMS}. 
\pend \end{rmk}
The following is immediate from the definitions:
\begin{lm}
   Let $\C$ be presentably symmetric monoidal. 
   \begin{itemize}
       \item $\C$-atomic objects in $\C$ are exactly dualizable objects. 
       \item If $f:\M_0\to \M_1$ is an equivalence of $\C$-modules in $\PrL$, it carries $\C$-atomic objects to $\C$-atomic objects. 
   \end{itemize}
\end{lm}
We can now prove: 
\begin{prop}[{\cite[Corollary 2.2.3.]{hopkinsluriebrauer}}]\label{prop : azeq}
Let $\C$ be presentably symmetric monoidal, and let $A\in\Alg(\C)$ be an algebra. The following are equivalent: 
\begin{enumerate}
    \item $A$ is Azumaya; 
    \item $A$ is dualizable, full, and the canonical map $A\otimes A\op\to \End(A)$ is an equivalence; 
    \item $A$ is dualizable, full, and there is an equivalence of algebras $A\otimes A\op\simeq \End(A)$; 
    \item There is some full dualizable module $M$ and an equivalence of algebras $A\otimes A\op\simeq \End(M)$; \item $A\otimes A\op$ is ($\C$-linearly) Morita equivalent to the unit $\one$; 
    \item There exists an algebra $B$, a full dualizable object $M$, and an equivalence $A\otimes B\simeq \End(M)$
    \item There exists an algebra $B$ and a ($\C$-linear) Morita equivalence between $A\otimes B$ and $\one$
    
\end{enumerate}
\end{prop}
\begin{proof}
We prove (1) $\implies$ (2) $\implies$ (3) $\implies$ (4) $\implies$(5) $\implies$ (7) $\implies$ (1), and we prove (6) $\iff$ (7). 

Note that (2) $\implies$ (3) $\implies$ (4) are just each specializations of the previous one, so these implications are obvious, same for (5)$\implies$ (7). 

For (4) $\implies$ (5) (resp. (6)$\implies$(7)), we simply observe that for a full dualizable object $M$, $\End(M)$ is Morita equivalent to $\one$ by the previous proposition/definition. 

(7)$\implies$ (1) follows from the observation that $\LMod_A \otimes_\C\LMod_B\simeq \LMod_{A\otimes B}$, and hence (7) implies that $\LMod_A\otimes_\C\LMod_B\simeq \C$, which is the definition of Azumaya. 

We are left with (1)$\implies$(2) and (7)$\implies$(6). The proof of (7)$\implies$(6) poceeds by observing that any algebra Morita equivalent to $\one$ is of the form $\End(M)$ for some full dualizable $M$. Indeed, suppose $A$ is such an algebra, and fix a Morita equivalence $F:\LMod_A\simeq \C$. Note that $A\in\LMod_A$ is $\C$-atomic, i.e. $\hom_A(A,-): \LMod_A\to \C$ is $\C$-linear and colimit-preserving - indeed, $\hom_A(A,-)$ is $\C$-linearly equivalent to the forgetful functor. As $F$ is a $\C$-linear equivalence, $F(A) \in \C$ is also atomic by the previous lemma, and it is therefore dualizable, also by the previous lemma. Furthermore, we have $A\simeq \End_A(A)\op\simeq \End(F(A))\op\simeq \End(F(A)^\vee)$. 

It follows that $F(A)^\vee$ is a dualizable object with $\End(F(A)^\vee)$ Morita equivalent to the unit, so by the previous proposition/definition, it is full, which proves the claim.

Finally, we need to prove that (1) implies (2). The observation here is that $\LMod_A$ is always dualizable in $\Mod_\C$, so that invertibility is the property that the evaluation and coevaluation maps, $\LMod_{A\op}\otimes_\C\LMod_A\to \C$ and $\C\to \LMod_A\otimes_\C\LMod_{A\op}$ respectively, be equivalences. 

For the second one, it implies in particular that $A$ is proper, i.e. that $A$ is dualizable as an object of $\C$.

Next, note that the map $ \LMod_{A\otimes A\op}\simeq  \LMod_A\otimes_\C\LMod_{A\op}\to \C$ is given by tensoring over $A\otimes A\op$ with the $A\otimes A\op$-module $A$. It therefore sends $A\otimes A\op$ to $A$, and as it is an equivalence, it induces an equivalence of algebras $A\otimes A\op\simeq \End_{A\otimes A\op}(A\otimes A\op)\xrightarrow{\simeq} \End(A)$. It is easy to check that this is the canonical map. 

To prove (2), we are left with checking that $A$ is full. But it follows from what we just said that $\LMod_{\End(A)}$ was equivalent to $\C$, and $A$ is dualizable, so by the previous proposition/definition, $A$ is full, and so we are done. 
\end{proof}
A further key property of Azumaya algebras is their \emph{centrality}. 
\begin{lm}\label{lm:aziscentral}
    Let $A\in \Alg(\C)$ be an Azumaya algebra. In this case, the center of $A$ is equivalent to the unit $\one$.
\end{lm}
\begin{proof}
    The center of $A$ is equivalent to the endomorphism object of the $\C$-module $\LMod_A(\C)$. Since the latter is invertible, the functor $\C\to \Fun^L_\C(\LMod_A(\C),\LMod_A(\C))$ is a $\C$-linear equivalence, and it sends $\one$ to $\id_{\LMod_A(\C)}$. It follows that $Z(A)\simeq\End(\id_{\LMod_A(\C)})\simeq \End(\one)\simeq \one$, as claimed. 
\end{proof}
In \cite[Proposition 1.4]{BRS}, it is claimed that an Azumaya algebra is necessarily separable, in analogy with \cite[Theorem 2.1.]{AuslanderGoldman}. Unfortunately, there is an error in their argument: in their notation, the module $\tilde F(A) = A\wedge_R A$ is \emph{not} the canonical bimodule $A\wedge_R A\op$, but rather the bimodule obtained by tensoring the canonical bimodule $A$ with the object $A$. There are, in fact, counterexamples to this statement. We give two: a local one, and a global one.

\begin{ex}\label{ex:MoravaK}
There are some associative ring structures on Morava K-theory $K(n)$ in the category $\Sp_{K(n)}$ of $K(n)$-local spectra, which are Azumaya algebras, cf. \cite{hopkinsluriebrauer}. However, none of these are separable: a bimodule splitting as in \Cref{defn:sep} would yield a retraction of $$E_n\simeq \Map_{K(n)\otimes K(n)\op}(K(n),K(n))\to K(n)$$ and there is clearly no such thing (the first equivalence follows from \Cref{lm:aziscentral}). 
\pend \end{ex}
\begin{ex}\label{ex:cofibeta}
The same example as in \Cref{rmk : notMorita} also provides a global example here, that is, without needing to localize. Namely, if $X$ is a type $0$ spectrum, such as the cofiber of $\eta$, $\End(X)$ is Morita equivalent to $\Sph$ and hence Azumaya, but we already argued that it is not separable.  
\pend \end{ex}

We can now state the corrected version of \cite[Proposition 1.4]{BRS}
\begin{prop}\label{prop:Azsep}
Let $A\in \Alg(\C)$ be an algebra. If $A$ is Azumaya, then $A$ is separable if and only if the unit $\eta : \one\to A$ admits a retraction. 
\end{prop}
\begin{proof}
First note that the multiplication map $A\otimes A\op\to A$ factors as $A\otimes A\op\to \End(A)\to A$, where the second map is evaluation at the unit $\eta : \one\to A$, as a map of bimodules.

If $A$ is Azumaya, it follows that this multiplication admits a bimodule section if and only if $ev_\eta : \End(A)\to A$ has an $A$-bimodule section. Now $\End(A)\simeq A\otimes A^\vee$ as $A$-bimodules, where the latter has the structure of $A$-bimodule coming from $A$, so that this map is really $A\otimes (A^\vee\xrightarrow{ev_\eta} \one)$. 

Finally, $A^\vee\to \one$ is dual to $\eta:  \one \to A$. So, if the unit has a retraction, then $ev_\eta: A^\vee \to \one$ has a section, and therefore so does $A\otimes (A^\vee\to \one)$, as a map of bimodules, and by the previous discussion, so does $A\otimes A\op\to A$, so that $A$ is separable. 

Conversely, if $A$ is separable, then the canonical map $Z(A)\to A$ admits a section (cf. \Cref{cor:centerabsolute}). As $A$ is Azumaya, we can combine this with \Cref{lm:aziscentral} to obtain that $\one \simeq Z(A)\to A$ admits a retraction. 
\end{proof}
We use this proposition to prove that, unlike in the commutative case, separability cannot be checked locally:
\begin{ex}\label{ex:GL}
    In \cite[Proposition 7.17]{GL}, Gepner and Lawson construct a twisted form of $M_2(\KU)$, that is, a $\KO$-algebra $Q$, necessarily Azumaya, for which $Q\otimes_{\KO}\KU\simeq M_2(\KU)$ as algebras. This algebra is not $M_2(\KO)$, in fact $\pi_*Q\cong \KU_*\langle C_2\rangle$, a twisted group ring. 

    It follows that $Q$ is not separable: it is Azumaya so by the above proposition, if it were separable, its unit would split. But it has no $\pi_1$, so such a splitting is impossible as $\pi_1(\KO)\neq 0$. Therefore we have a non-separable algebra, $Q$, whose basechange along a Galois-extension is separable - it follows that $\Alg^{sep}(\Mod_{\KU}^{hC_2})\to \Alg^{sep}(\Mod_{\KU})^{hC_2}$ is not an equivalence: the former is equivalent to $\Alg^{sep}(\Mod_{\KO})$, and the latter to the full subgroupoid of $\Alg(\Mod_{\KO})$ consisting of those algebras whose basechange to $\KU$ is separable. 
\pend \end{ex}
\begin{rmk}
Let us mention, without too much detail, the following interpretation of separable Azumaya algebras. Let $\C$ be an additively symmetric monoidal \category{}, and $A$ an Azumaya algebra therein. In this case, $A$ is separable (equivalently, its unit splits) if and only if it remains Azumaya in $\Syn_\C=\Fun^\times(\C\op,\Sp)$ under the symmetric monoidal Yoneda embedding $\C\to \Syn_\C$, if and only if it is ``absolutely Azumaya'', i.e. it remains Azumaya after applying any additive symmetric monoidal functor (the part of the definition of ``Azumaya'' which is not clearly preserved by any functor is the ``fullness'' property, but is here guaranteed by the retraction onto the unit). 

Sven van Nigtevecht has independently observed\footnote{Private communication.} that the obstruction theory from \cite{PVK}\footnote{Which used to be used in \Cref{subsection:obstruction}} can be used in the case where $A$ is an Azumaya algebra which remains Azumaya in $\Syn_\C$, for then the mapping spectrum in $\Mod_{A\otimes A\op}(\Syn_\C)$ from $A$ to itself is simply $\Map_{\Syn_\C}(\one_{\Syn_\C},\one_{\Syn_\C})\simeq \Map(\one_\C,\one_\C)_{\geq 0}$. From our perspective, this is explained by the fact that under this assumption, $A$ is actually separable.  
\pend \end{rmk}

In the setting of classical rings, a stronger result holds: if $A$ is dualizable, separable, and \emph{central}, i.e. its center $Z(A)$ is the unit $\one$, then $A$ is Azumaya. We do not know whether the converse holds in our generality, and we therefore raise it as a question:
\begin{ques}\label{question : centralazumaya}
Let $A\in\Alg(\C)$ be a dualizable separable algebra which is \emph{central}, i.e. the unit map $\one\to Z(A)$ is an equivalence. In particular, $A$ is full, as it retracts onto $Z(A)\simeq \one$. 

Is $A$ necessarily Azumaya ? 
\end{ques}
We provide a positive answer in the following cases:
\begin{thm}\label{thm:Brauer}
If $\C$ is one of the following: 
\begin{itemize}
    \item $\QCoh(X)$ for some (connective) spectral Deligne-Mumford stack $X$ \cite[Definition 1.4.4.2]{SAG}; 
    \item  $\Mod_R(\Sp)$, where $R$ is some commutative ring spectrum for which $R\otimes\mathbb F_p=0$ for all primes $p$; 
    \item $\Mod_R(\Sp)$, where $R$ is a commutative ring spectrum which is even, $2$-periodic and whose $\pi_0$ is regular noetherian, and in which $2$ is invertible.
\end{itemize}
then \Cref{question : centralazumaya} has a positive answer for $\C$: if $A\in\Alg(\C)$ is a dualizable separable algebra which is \emph{central}, i.e. the unit map $\one\to Z(A)$ is an equivalence, then $A$ is Azumaya. 
\end{thm}
\begin{rmk}
    The second situation of \Cref{thm:Brauer} is somewhat orthogonal to the first one: such a commutative ring $R$, unless it is rational, must be non-connective, and of ``chromatic'' flavour. For instance, Morava $E$-theories fall into this category. 

    The third situation allows for certain non-connective commutative $\mathbb F_p$-algebras at odd primes, such as $\mathbb F_p^{tS^1}$, but not, e.g., $\mathbb F_p^{tC_p}$. 
\pend \end{rmk}
\begin{proof}
    Combine \Cref{cor:BrDMStack}, \Cref{cor:Brchro} and \Cref{prop:BrEvenRegNoeth}. 
\end{proof}
The strategy of proof in all cases of \Cref{thm:Brauer}, which is also the one we will use in \Cref{section:center} to adress \Cref{question : sepcenter}, is to try to \emph{descend} the question to simpler and simpler $\C$'s, until we reach a classical algebraic $\C$, where the usual proofs just go through. 

The ``descent'' statement in this case is the following:
\begin{lm}\label{lm:Brbootstrap}
Let $f: \C\to \D$ be a \emph{conservative} symmetric monoidal functor. For a dualizable, full algebra $A\in\Alg(\C)$, if $f(A)$ is Azumaya, then so is $A$. 

More generally, if $f_i: \C\to \D_i$ is a jointly conservative family of symmetric monoidal functors, if each $f_i(A)$ is Azumaya, then so is $A$. 
\end{lm}
\begin{proof}
We deal with the case of a single functor, the other case being similar (or simply a consequence, by taking $f= (f_i)_{i\in I} : \C\to\prod_I \D_i$).

$A$ is already assumed to be dualizable and full, so by point 2. in \Cref{prop : azeq}, it suffices to show that the canonical map $A\otimes A\op\to \End(A)$ is an equivalence. 

The functor $f$ is symmetric monoidal, and $A$ is rigid, so that applying $f$ to this map yields the canonical map $f(A)\otimes f(A)\op\to \End(f(A))$. By conservativity of $f$, if this is an equivalence, then so was the canonical map.
\end{proof}
The key example we try to reduce to is the category of modules over a graded field. 
\begin{nota}
    We consider the category of graded abelian groups as symmetric monoidal using the Koszul convention: the symmetry isomorphism $A\otimes B\cong B\otimes A$ is $a\otimes b\mapsto (-1)^{|a||b|}b\otimes a$ for homogeneous elements $a,b$ of respective degrees $|a|,|b|$.
\pend \end{nota}
\begin{defn}
    A graded field is a commutative algebra $k$ in graded abelian groups such that every homogeneous element $x\in k_*$ is invertible. 

    A graded division algebra is similar, except we do not require commutativity.
\pend \end{defn}
\begin{rmk}
    Graded fields are easy to classify: they are either fields concentrated in degree $0$, or of the form $k[t^{\pm 1}]$ for some $t$ of positive degree - necessarily even if the characteristic of $k$ is not $2$. 

    On the other hand, graded division algebras are more complicated to classify: even if the degree $0$ part is a field (i.e. commutative), the non-commutativity of the multiplication in higher degrees allows for a wealth of examples. 
\pend \end{rmk}

\begin{prop}\label{prop:BrGrVect}
Let $k$ be a graded field, and $\D$ the category of graded $k$-vector spaces. Any central separable algebra in $\D$ is Azumaya. 
\end{prop}
One way to go about this proof is to prove the following lemma, which is classical in the ungraded case and most likely well-known in the graded case too. There is, however, an easier proof given our assumption, so we will simply mention the lemma here and let the reader fill in the details of this proof if they are interested. 
\begin{lm}\label{lm:gradeddivision}
    Let $k$ be a graded field, and $D,D'$ central graded division algebras over $k$. The algebra $D\otimes_k D'$ is graded simple, i.e. it has no nontrivial homogeneous ideal.
\end{lm}
We are in a simpler situation, as we assume separability:
\begin{lm}\label{lm:sepimpliessemisimp}
    Let $\D$ be a symmetric monoidal abelian category which is semi-simple, and let $A\in\Alg(\C)$ be a separable algebra. Any (bilateral) ideal $I$ in $A$ splits: there is an isomophism of algebras $A\cong I\times A/I$. 

    In particular, if $A$ is central, i.e. $\one\cong Z(A)$, and $\End_\D(\one)$ has no nontrivial idempotents, then any (bilateral) ideal is $0$ or $A$. 
\end{lm}
\begin{proof}
    As $\D$ is semi-simple, the inclusion $I\to A$, which is a morphism of $A\otimes A\op$-modules, admits a section in $\D$. Because $A\otimes A\op$ is separable, \Cref{cor:sectionunderlying} implies that it admits a section of $A$-bimodules. The result follows by \Cref{lm : unitsum}. 

    The ``in particular'' follows from the fact that $Z(A\times B)\cong Z(A)\times Z(B)$. 
\end{proof}
\begin{lm}\label{lm:centertensor}
    Let $\C$ be idempotent-complete, and let $A,B\in\Alg(\C)$ be separable. The canonical map $Z(A)\otimes Z(B)\to Z(A\otimes B)$ is an equivalence. 
\end{lm}
\begin{proof}
    Note that this a map of the form $\hom_R(M,N)\otimes \hom_S(P,Q)\to \hom_{R\otimes S}(M\otimes P,N\otimes Q)$, and the latter is natual in $M,N,P,Q$. Furthermore, for $M=N=R, P=Q=S$, it is clearly an equivalence. Hence it is so for any tuple $(M,N,P,Q)$ which is a retract of $(R,R,S,S)$.

    By separability of $A,B$, $(A,A,B,B)$ is a retract of $(A\otimes A\op,A\otimes A\op,B\otimes B\op,B\otimes B\op)$ and so we are done. 
\end{proof}
\begin{rmk}
    In fact an easy modification of this proof shows that it suffices that $A$ is separable if we also assume that $B$ is smooth, or that $A$ is proper, see \cite[Section 4.6.4]{HA} for definitions.
\pend \end{rmk}
\begin{proof}[Proof of \Cref{prop:BrGrVect}]
As $k$ is a graded field, the category $\D$ of graded $k$-vector spaces is semi-simple, and $\End_\D(\one)= \End(k) = k$ has no nontrivial idempotents. 

In particular, by \Cref{lm:sepimpliessemisimp} if $A$ is central and separable, then it is simple: it has no nontrivial ideals.  

We apply this to $A\otimes A\op$ instead: it is still separable (\Cref{lm:general}) and central (\Cref{lm:centertensor}), and therefore by the above it is simple.  

It follows that the canonical map $A\otimes A\op\to \End(A)$, which is an algebra map, has no kernel, i.e. it is injective. Comparing the dimensions of both sides implies that it is an isomorphism\footnote{We are in a graded setting, but over a graded field, so dimensions still make sense.}. By \Cref{prop : azeq}, point 2., we are done. 
\end{proof}
The case of ordinary fields is enough to bootstrap to all connective Deligne-Mumford stacks:
\begin{cor}\label{cor:BrDMStack}
    Let $X$ be a (connective) Deligne-Mumford stack. Any dualizable central separable algebra in $\QCoh(X)$ is Azumaya. 
\end{cor}
\begin{proof}
    By \cite[Proposition 6.2.4.1]{SAG}, $\QCoh(X)$ is a limit, in $\CAlg(\PrL)$, of \categories{} of the form $\Mod_R(\Sp)$, where $R$ is a connective commutative ring spectrum. 

    Since for any diagram $f:I\to \Cat$ and any essentially surjective map from a set $I_0\to I$, the forgetful functor $\lim_I f\to \prod_{I_0}f$ is conservative, we can apply \Cref{lm:Brbootstrap} to reduce to the case of $\Mod_R(\Sp)$, where $R$ is a connective commutative ring spectrum. 

    Since we assumed the algebra was dualizable, we can in fact reduce to $\Perf(R)$. Now, for a connective ring spectrum $R$, the restriction of $\pi_0(R)\otimes_R -$ to bounded below $R$-modules is symmetric monoidal and conservative, so we can reduce to the case where $R$ is discrete, again by \Cref{lm:Brbootstrap}. 

    For a discrete commutative ring $R$, the basechange functors along all ring maps $R\to k$, where $k$ is a field are jointly conservative on perfect $R$-modules, so we can reduce to the case of a field.

     Now note that for any $\C$, $\C\to \ho(\C)$ is conservative and symmetric monoidal. For a field $k$, $\ho(\Mod_k)$ is symmetric monoidally equivalent to the $1$-category of graded $k$-vector spaces, and so \Cref{prop:BrGrVect} allows us to conclude. 
\end{proof}
To deal with the second case of \Cref{thm:Brauer}, we first specialize to $R=$ Morava $E$-theory - the nilpotence theorem \cite{HS} and the chromatic Nullstellensatz \cite{ChroNS} will be our tools to reduce to this key case. 

The situation is simpler at odd primes than at the even prime, so we first deal with the odd primes, even though the proof we will give for the prime $2$ also works for odd primes. 
\begin{prop}
    Let $R=E=E(k,\mathbf G)$ be a Morava $E$-theory\footnote{See \cite[Section 2.4]{ChroNS} for a modern introduction} over some field $k$ of odd characteristic, and at some height $n>0$, and let $A\in\Alg(\Mod_E)$ be a rigid separable $E$-algebra. If $A$ is furthermore central, then $A$ is Azumaya.
\end{prop}
\begin{rmk}
In contrast to \Cref{ex:MoravaK}, these Azumaya algebras are not ``atomic'' in the sense of \cite{hopkinsluriebrauer}, precisely because they are separable and therefore retract onto $E$.
\pend \end{rmk}
\begin{proof}
    Because we are working at an odd prime, there exist ring structures on Morava $K$-theory $K(n)$ that are homotopy commutative \cite[Section 3]{strickland}. In this case, $K(n)_*:\Mod_E\to \Mod_{K(n)_*}(\mathbf{GrVect}_k)$ is a symmetric monoidal functor, and it is conservative when restricted to $K(n)$-local $E$-modules, in particular when restricted to perfect, or equivalently dualizable, $E$-modules. 
    
As $K(n)_*$ is a graded field, \Cref{prop:BrGrVect} applies again, and we are done, again by \Cref{lm:Brbootstrap}. 
\end{proof}
In fact, thanks to work of Mathew \cite{mathewthick}, the same argument works more generally: 
\begin{prop}\label{prop:BrEvenRegNoeth}
    Let $R$ be a commutative ring spectrum which is even, $2$-periodic, with regular noetherian $\pi_0$, and such that $2\in \pi_0(R)^\times$. Let $A\in\Alg(\Mod_R)$ be a dualizable separable $R$-algebra. If $A$ is furthermore central, then $A$ is Azumaya.
\end{prop}
\begin{proof}
    The same proof as above works, where we replace $K(n)$ by the $K(\mathfrak p)$'s, cf. \cite[Definition 2.5]{mathewthick}. Indeed, each $K(\mathfrak p)_*$ is a graded field by \textit{loc. cit.}, they are jointly conservative on perfect $R$-modules \cite[Proposition 2.8]{mathewthick} (in fact on all modules), and finally by \cite[Section 3]{strickland}, if $2\in \pi_0(R)^\times$, they can be chosen to be homotopy commutative.
\end{proof}
We now deal with the even prime. The point is that in this situation, Morava $K$-theory cannot be chosen to be homotopy commutative, so that $K(n)_*$ is only monoidal, but not symmetric monoidal, which means that it is possibly not compatible with the map $A\otimes A\to \End(A) \simeq A^\vee\otimes A$, and in particular we cannot check Azumaya-ness through this functor.

There is a way out, using the notion of Milnor modules \cite[Section 6]{hopkinsluriebrauer}. The main take-away of this notion for us is the following: 
\begin{thm}\label{thm:Milnor}
Let $E=E(k,\mathbf G)$ be a Morava $E$-theory at height $n$ at the prime $p$, possibly even. There is a symmetric monoidal $1$-category $\Mil_E$ of \emph{Milnor-modules} together with a (strong) symmetric monoidal homology theory $h_*:\Mod_E\to \Mil_E$. 

For any choice of a Morava $K$-theory $K\in\Alg(\ho(\Mod_E))$, the monoidal homology theory $K_*: \Mod_E\to \coMod_{K_*^EK}(\Mod_{K_*}((\mathbf{GrVect}_k))$ factors through a \emph{monoidal}\footnote{At the prime $2$, there is no choice of Morava $K$-theory that makes this symmetric monoidal.} equivalence $\Mil_E\simeq \coMod_{K_*^EK}(\Mod_{K_*}((\mathbf{GrVect}_k))$. 
\end{thm}
Note that the notion of separable algebra, and of rigidity can be phrased completely in monoidal terms (for duality, one needs to worry about left vs right duality, but these notions still make sense). The only part of ``rigid central separable algebra'' that requires symmetry is the centrality part. 

In particular, if $A\in\Alg(\Mod_E)$ is rigid and separable, $K_*(A)$ is a rigid separable algebra in $K_*$-modules in $\mathbf{GrVect}_k$. This will turn out to be enough for us. 

We begin with a lemma:
\begin{lm}\label{lm:idempotentcomod}
    Let $\C$ be a symmetric monoidal $1$-category, $H$ a commutative Hopf algebra in $\C$, i.e. a group object in $\CAlg(\C)\op$. 

Let $I$ be a non-unital algebra in $\coMod_H(\C)$ such that the underlying non-unital algebra $I$ in $\C$ admits a unit \cite[Definition 5.4.3.1]{HA}. In this case, $I$ admits a unit in $\coMod_H(\C)$.  
\end{lm}
\begin{rmk}
        We state and prove this lemma for $1$-categories because in the proof, we use a description of comodules as ``algebraic representations'' of an ``algebraic group'' (see below). This description is elementary for $1$-categories, while for $\infty$-categories, it is highly expected to hold completely analogously, but we did not want to get into the intricacies of its proof. 

        We later only use it for $1$-categories, so this is not an issue, but it would be interesting to prove the corresponding description for symmetric monoidal \categories{} (the lemma, for instance, would follow immediately in the same generality). 
    \pend\end{rmk}
As mentioned in the remark, to prove this lemma, it is convenient to use the usual description of $\coMod_H(\C)$ as ``algebraic representations of $\Spec_\C(H)$''. Let us make this a bit more precise. The corepresented functor $M=\Spec_\C(H):\CAlg(\C)\to \Ss$ given by $\map(H,-)$ is canonically a monoid whenever $H$ is a comonoid in $\CAlg(\C)$. The categoy of transformations $BM(R)\to \Mod_R(\C)$, natural in $R\in\CAlg(\C)$ can be viewed as a category of ``algebraic representations of $M$'' - here, $R\mapsto \Mod_R(\C)$ is functorial along base-change\footnote{Because every such natural transformation has a value $c$ at $R=\one$, the value at every other $R$ is of the form $R\otimes c$, and so all the required basechanges exist, along arbitrary maps $R\to S$, therefore, to make this definition, we do not actually need $\C$ to have arbitrary relative tensor products.}.

It is an instructive exercise to prove that this category is symmetric monoidally equivalent to the category of $H$-comodules, compatibly with the forgetful functor to $\C$. We use this fact without further comment. 
\begin{nota}
    Let $M:\CAlg(\C)\to \Mon$ be a functor from commutative algebras in $\C$ to (discrete) monoids. We let $\Rep_M(\C)$ denote the symmetric monoidal category of algebraic representations of $M$, as described above. 
\end{nota}
    \begin{rmk}
    We note that here, $H$ needs to be a Hopf algebra - the lemma is not true for general commutative bialgebras. For example, let $H$ be the bialgebra in abelian groups whose underlying algebra is $\mathbb Z\times\mathbb Z$. The functor on $\CAlg(\Ab)$ it corepresents is simply $\Idem: R\mapsto \Idem(R)$, the functor mapping a ring to its set of idempotents, and we can make it a commutative monoid under multiplication, thus making $H$ into a bialgebra. In this case, one can make $\mathbb Z$ into an algebraic representation of $\Spec_\C(H)$, i.e. an $H$-comodule, via the canonical action of $\Idem(R)$ on $R$ by multiplication. It is easy to check that this makes it into a non-unital algebra whose underlying algebra is unital, but it is not unital.  
    \pend \end{rmk}
    
\begin{proof}
By \cite[Theorem 5.4.3.5]{HA}, the unit of a non-unital algebra, if it exists, is unique. More precisely, the forgetful functor $\Alg(\D)^\simeq\to \Alg^{nu}(\D)^\simeq$ is fully faithful\footnote{Note that this is not true if one removes the symbol $^\simeq$, it is only faithful.}

 In particular, for any group $G$, if $A\in \Alg^{nu}(\Fun(BG,\D))$ is a non-unital algebra such that the underlying $A\in \Alg^{nu}(\D)$ admits a unit, then $A$ admits a unit too.  More precisely, the canonical map $\Alg(\Fun(BG,\D))\to \Alg^{nu}(\Fun(BG,\D))\times_{\Alg^{nu}(\D)}\Alg(\D)$ is an equivalence.

 It follows that the same holds for the category of representations of any functor $G:\CAlg(\C)\to \Grp$, i.e. the canonical map $\Alg(\Rep_G(\C))\to \Alg^{nu}(\Rep_G(\C))\times_{\Alg^{nu}(\C)}\Alg(\C)$ is an equivalence for any such $G$. 
 
 The result now follows from the symmetric monoidal equivalence $\coMod_H(\C)\simeq \Rep_{\Spec_\C(H)}(\C)$, compatible with the forgetful functor as discussed before the proof. 
\end{proof}

\begin{prop}\label{prop:BrEthy}
    Let $R=E=E(k,\mathbf G)$ be a Morava $E$-theory over some field $k$ of positive, possibly even characteristic, and at some height $n>0$, and let $A\in\Alg(\Mod_E)$ be a dualizable separable $E$-algebra. If $A$ is furthermore central, then $A$ is Azumaya.
\end{prop}
\begin{proof}
Fix an atomic $E$-algebra $K$ \cite[Definition 1.0.2]{hopkinsluriebrauer}, i.e. a Morava $K$-theory. 

By \Cref{thm:Milnor}, $K_*$ factors through $h_*:\Mod_E\to \Mil_E$, and $K_*$ is conservative on perfect $E$-modules, hence by \Cref{lm:Brbootstrap}, it suffices to prove the result in $\Mil_E$. 

We prove the following intermediary result: let $A$ be a central separable algebra in $\Mil_E$, then $A$ is simple, i.e. any (bilateral) ideal $I\hookrightarrow A$ is $0$ or $A$.  Notice that the functor $\Mil_E\to \coMod_{K_*^EK}(\Mod_{K_*}((\mathbf{GrVect}_k))\to \Mod_{K_*}(\mathbf{GrVect}_k)$ is (strong) monoidal, and conservative, so it sends ideals to ideals, and dualizable separable algebras to dualizable separable algebras. 

By \Cref{lm:sepimpliessemisimp}, there is a central idempotent $e$ in $A$ such that $I = eA$. Furthermore, we started with an ideal in $\coMod_{K_*^EK}(\Mod_{K_*})$, and \Cref{lm:idempotentcomod} will in fact imply that $e$ is a morphism $K_*\to A$ in comodules, and not only in $\Mod_{K_*}$ (note that $K_*^EK$ is a commutative Hopf algebra by \cite[Lemma 2.6]{piotrtobi}- this is so even at the prime $2$).

The algebra we apply \Cref{lm:idempotentcomod} to is $I$, viewed as a non-unital algebra in comodules. The existence of the central idempotent $e$ in $A$ such that $I = eA$ guarantees that $I$ is unital in $\Mod_{K_*}$, and thus, the lemma guarantees that it is unital in comodules.

This means that its unit is a morphism of $K_*^EK$-comodules $K_*\to I$, i.e., that the idempotent $e$ is a map of $K_*^EK$-comodules $K_*\to A$.

This further implies that $A$ splits as an algebra in $\coMod_{K_*^EK}$ as $I\times A/I$. This being a statement only about the monoidal structure of $\coMod_{K_*^EK}$, it holds also in the monoidal category of Milnor modules, i.e. $\Mil_E$. But there, $A$ is central by assumption, and so $I= 0$ or $A$, as was to be proved. We have thus proved that $A$ was simple. 

We now apply this to $A\otimes A\op$, which is dualizable, central and separable as well, and hence simple. It follows that the canonical map $A\otimes A\op\to \End(A)$ is injective. Now, the two sides have the same (finite) dimension as $K_*$-modules, so it follows that this map is an isomorphism, which is what was to be proved.  
\end{proof}
To prove the general case of a commutative ring spectrum for which $R\otimes\mathbb F_p=0$ for all $p$, we use the nilpotence theorem \cite{HS}. Let us recall an important consequence of it:
\begin{prop}\label{prop:nilpotence}
    Let $R$ be a commutative ring spectrum and $P$ a dualizable $R$-module. Suppose that for all implicit primes $p$ and and all $0\leq n\leq \infty$, $L_{K(n)}P= 0$. In this case, $P=0$. 

    Here, $K(0)=\mathbb Q,K(\infty) = \mathbb F_p$.
In particular, if $R\otimes\mathbb F_p= 0$ for all primes $p$, then it suffices to check that $L_{K(n)}P= 0$ for all $0\leq n<\infty$.
\end{prop}
\begin{proof}
    By definition, $L_{K(n)}P=0$ if and only if $K(n)\otimes P = 0$. As $P$ is dualizable and $K(n)$ admits a ring structure, $K(n)\otimes P = 0$ if and only if $K(n)\otimes \End(P)= 0$: one direction is always true, as $K(n)\otimes P$ is a module over $K(n)\otimes\End(P)$. For the other direction, note that $\End(P)\simeq P\otimes_R P^\vee \simeq \colim_{\Delta\op}P\otimes R^{\otimes n}\otimes P^\vee$. 

    Similarly, $P=0$ if and only if $\End(P)=0$.

    Now, $\End(P)$ is an $\mathbb E_1$-ring, so the result follows from \cite[Theorem 3]{HS}.

    The ``in particular'' part follows from the fact that if $R\otimes\mathbb F_p=0$, then $P\otimes\mathbb F_p=0$ too. 
\end{proof}
We also recall an important consequence of the Chromatic Nulstellensatz \cite{ChroNS}. 
\begin{prop}\label{prop:NS}
Fix an implicit prime $p$.
    Let $R$ be a $K(n)$-local commutative ring spectrum, and $P$ a nonzero dualizable $R$-module. There exists a field $L$ as well as a map of commutative ring spectra $R\to E(L)$ such that $E(L)\otimes_R P \neq 0$. 
\end{prop}
To prove this from the results of \cite{ChroNS}, we need a bit of work. Before doing so, let us deduce the desired result from this. 

\begin{cor}\label{cor:Brchro}
    Let $R$ be a commutative ring spectrum such that $R\otimes\mathbb F_p=0$ for all $p$. \Cref{question : centralazumaya} has a positive answer in $\Mod_R(\Sp)$, that is, every dualizable central separable algebra is Azumaya.  
\end{cor}
\begin{proof}
    Let $A$ be a dualizable central separable algebra over $R$, and let $P$ denote the cofiber of $A\otimes A\op\to \End(A)$. We aim to prove that $P=0$. To reach a contradiction, we assume $P\neq 0$. 
    
    As $A$ is dualizable, $P$ is dualizable too. By \Cref{prop:nilpotence}, there exists a prime $p$ and an $n$ such that $L_{K(n)}P\neq 0$. By \Cref{prop:NS}, there exists a field and a map of commutative ring spectra $L_{K(n)}R\to E(L)$ such that $E(L)\otimes_{L_{K(n)}R}L_{K(n)}P\neq 0$. Note that $P$ is dualizable over $R$, so that $L_{K(n)}R\otimes_R P\simeq L_{K(n)}P$. 
    
    Therefore, $E(L)\otimes_R P\neq 0$. As $E(L)\otimes_R -$ is symmetric monoidal, we find that $E(L)\otimes_R A$ is not Azumaya. This contradicts \Cref{prop:BrEthy}. 
    \end{proof}
We now explain how to deduce \Cref{prop:NS} from \cite{ChroNS}.

First, a definition \cite[Definition 4.4.1]{AmbiChro}: 
\begin{defn}
    A monoidal functor $f:\D\to\E$ between stably monoidal \categories{} is said to be \emph{nil-conservative} if for all $R\in \Alg(\D)$, $f(R) = 0$ implies that $R=0$. 
\pend \end{defn}
\begin{lm}[{\cite[Lemma 4.32]{ChroNS}}]
    Let $\C \in \CAlg(\PrL)$ be compactly generated, with the property that every compact in $\C$ is dualizable, and let $A\to B$ a morphism in $\CAlg(\C)$. If it detects nilpotence, then $B\otimes_A -: \Mod_A(\C)\to \Mod_B(\C)$ is nil-conservative.  
\end{lm}
\begin{lm}[{\cite[Proposition 4.4.4]{AmbiChro}}]
    A nil-conservative monoidal exact functor between stably monoidal categories is conservative when restricted to dualizable objects. 
\end{lm}
\begin{cor}\label{cor:nildetcons}
    Let $\C \in \CAlg(\PrL)$ be compactly generated, with the property that every compact in $\C$ is dualizable, and $A\to B$ a morphism in $\CAlg(\C)$. If it detects nilpotence, then $B\otimes_A -: \Mod_A(\C)\to \Mod_B(\C)$ is conservatie when restricted to dualizable objects. 
\end{cor}
One of the main results of \cite{ChroNS} is: 
\begin{thm}[{\cite[Theorem 5.1]{ChroNS}}]\label{thm:chroNS}
    Let $R$ be a nonzero $T(n)$-local ring. There exists a perfect $\mathbb F_p$-algebra $A$ of Krull dimension $0$ and a nilpotence detecing map $R\to E(A)$ in $\Sp_{T(n)}$. 
\end{thm}
If $R$ is $K(n)$-local, then it is also $T(n)$-local. If $P$ is furthermore dualizable over $R$, then for any map of commutative algebras $R\to S$ to a $K(n)$-local ring $S$, $E(A)\otimes_R P$ is already $K(n)$-local.
\begin{cor}
    In order to prove \Cref{prop:NS}, it suffices to prove the special case where $R= E(A)$ for $A$ a perfect $\mathbb F_p$-algebra of Krull dimension $0$. 
\end{cor}
\begin{proof}
Suppose \Cref{prop:NS} holds whenever $R=E(A)$, $A$ a perfect $\mathbb F_p$-algebra of Krull dimension $0$, and let $R$ be an arbitrary $K(n)$-local commutative ring spectrum, and $P$ a nonzero dualizable $R$-module. 

By \Cref{thm:chroNS} (\cite[Theorem 5.1]{ChroNS}), we can find a nilpotence detecting map $R\to E(A)$ in $\Sp_{T(n)}$ for some perfect $\mathbb F_p$-algeba of Krull dimension $0$, $A$.  By \Cref{cor:nildetcons}, the $T(n)$-local tensor product with $E(A)$ over $R$ is conservative on dualizable objects, hence $E(A)\otimes_R P$ is nonzero, since its $T(n)$-localization is nonzero (note that $P$ is dualizable over $R$, so it is already $T(n)$-local, and hence it is nonzero as a $T(n)$-local $R$-module), and dualizable over $E(A)$, thus \Cref{prop:NS} follows for $R$. 
\end{proof}
The proof of this special case is in fact implicit in the proof of \cite[Theorem 4.47]{ChroNS} - we reproduce the proof nonetheless, for the convenience of the reader, as it is not explicitly spelled out:
\begin{proof}[Proof of \Cref{prop:NS}]
    By the previous corollary, we may assume $R= E(A)$ for some perfect $\mathbb F_p$-algebra $A$ of Krull dimension $0$. 

    Let $P$ be a dualizable $E(A)$-module. For any field $k$ and any map $A\to k$, $E(k)\otimes_{E(A)}P$ is $K(n)$-local, and thus equivalent to its $K(n)$-localization. 

Assume $E(k)\otimes_{E(A)}P= 0$ for all such $A\to k$, we wish to prove that $P= 0$. The \category{} $L_{K(n)}\Mod_{E(A)}$ is compactly generated so it suffices to show that $[c,P]_{E(A)} = 0$ for any compact $c$. As $c$ is compact in a $p$-complete \category{}, $p$ acts nilpotently on it. It follows that if $[c,P]_{E(A)}\otimes_{W(A)}A \cong [c,P]_{E(A)} \otimes_{\mathbb Z}\mathbb F_p$ is zero, then so is $[c,P]_{E(A)}$. Here, $W(A)$ is the ring of Witt vectors of $A$. 

Now, by \cite[Lemma 4.45]{ChroNS}, if $[c,P]_{E(A)}\otimes_{W(A)}A $  is nonzero, there is a perfect field $k$ and a map $A\to k$ such that $[c,P]_{E(A)}\otimes_{W(A)}A\otimes_A k \neq 0$. By \cite[Lemma 4.46]{ChroNS}, this tensor product is $[c\otimes_{E(A)}E(k), P\otimes_{E(A)}E(k)]_{E(k)} $, and this is $0$ by assumption. 

Here, we have used \cite[Lemma 4.46]{ChroNS} with $\mathcal C = L_{K(n)}\Mod_{E(\mathbb F_p)}$ so that, by \cite[Lemma 2.37]{ChroNS}, $\mathbb W_{\mathcal C}(B) \simeq E(B)$ for any perfect $\mathbb F_p$-algebra $B$ (in particular $B=\mathbb F_p,A$), and so that this really is an application of \cite[Lemma 4.46]{ChroNS}. 
\end{proof}
This concludes the proof of \Cref{thm:Brauer}. As is clear from the proof, if one wants to answer \Cref{question : centralazumaya} positively for $\Mod_R(\Sp)$ for an arbitrary commutative ring spectrum $R$, one may without loss of generality assume $R$ is an $\mathbb F_p$-algebra for some prime $p$. In this case, residue fields are harder to come by, and are the subject of ongoing work. 

A first issue is that, away from characteristic $2$, one cannot hope for graded fields, cf. \cite[Example 3.9]{mathewresidue}. One could still try to find enough ``residue fields'' and analyze their homotopy categories in enough detail to answer the question there. A good test-case would be to start with $R= k^{tC_p}$, where $k$ is a field of characteristic $p$ and the Tate construction is taken with respect to the trivial action. In this case, $\Mod_{k^{tC_p}}\simeq \mathrm{StMod}_{kC_p}$, the stable module \category, and it seems possible to study the separable algebras therein an try to prove that they are Azumaya - for instance, the commutative case was studied in \cite{sepstmod} (of course, the commutative case is orthogonal to our discussion, but Balmer and Carlson's result shows that such an analysis is not completely impossible).  

We note that \Cref{question : centralazumaya} (both its inputs and its answers, positive or negative) can be phrased in the homotopy category $\ho(\C)$, and so one can also try to approach it using the homological residue fields of Balmer \cite{balmerfields}. Thus the question becomes completely about (graded) abelian symmetric monoidal $1$-categories, over $\mathbb F_p$. It is not clear to the author whether one can say anything in this generality. 

\subsection{Centers of separable algebras}\label{section:center}
In this subsection, we study \Cref{question : sepcenter}. Just as in the previous subsection, our approach is via descent. As the center of an algebra is $\mathbb E_2$, and in particular homotopy commutative, \Cref{cor:descenthocomm} tells us that separability can be tested locally. 

In the previous subsection however, we used a much weaker notion of ``local'', namely, we tested Azumaya-ness against \emph{conservative functors}, of which there is a larger supply than ``descendable'' functors. We were not able to phrase separability in terms of certain maps being equivalences, and so we are not able to use this technique.

For this reason, our positive answer is in a more restricted generality. The goal of this section is to prove: 
\begin{thm}\label{thm:center}
Let $A \in\Alg(\C)$ be a separable algebra. \Cref{question : sepcenter} has a positive answer, i.e. the center $Z(A)$ is separable, in the following cases:
\begin{enumerate}
    \item\label{item:sepcenterring} If $\C=\Mod_R(\Sp)$ for some connective commutative ring spectrum $R$, and $A$ is almost perfect \cite[Definition 7.2.4.10]{HA}.
 More generally, this holds if $\C= \QCoh(X)$ for some (connective) Deligne-Mumford stack $X$ and if $A$ is locally almost perfect.  
    \item\label{item:sepcentermorava} If $\C= \Mod_E(\Sp_{K(n)})$ is the \category{} of $K(n)$-local $E$-modules, where $E$ is Morava $E$-theory at height $n$, for some height $n$ and some odd implicit prime $p$. In particular, the same is true if $\C=\Sp_{K(n)}$. 
\end{enumerate}
\end{thm}
\begin{rmk}
    If we have some \emph{a priori} control over $Z(A)$, one can get sometimes phrase separability in terms of certain maps being equivalences, and then get a more general positive answer. This is the case if, for instance, we assume that $A$ is sufficiently finite over its center. For instance, in the ordinary category of (discrete) $R$-modules for some (discrete) commutative ring $R$, Auslander and Goldman prove in \cite[Theorem 2.1]{AuslanderGoldman} that a separable algebra is always dualizable over its center. We do not know in what generality this can be expected, and as we were not able to formulate general criteria for this to happen, we did not include results along these lines here.  
\pend \end{rmk}
\begin{rmk}
We note that if $A$ is separable and almost perfect over a connective commutative ring spectrum $R$, $Z(A)$ is also almost perfect, and thus, by \Cref{prop:etalesepaperf}, it is separable if and only if it is étale. Under these finiteness assumptions, étaleness can be checked ``conservative locally'', and this is how we will be able to actually prove \Cref{item:sepcenterring}. In an earlier draft, the assumptions on $R$ were more restrictive, and we are grateful to Niko Naumann and Luca Pol for sharing a draft of their work which allowed us to prove \Cref{prop:etalesepaperf}, and subsequently, this version of the above theorem. 
\end{rmk}

Before moving on to the proof of \Cref{thm:center}, we note that under a positive answer to \Cref{question : sepcenter}, we can somewhat recreate the picture from \cite{AuslanderGoldman}: 
\begin{lm}\label{lm : centercenter}
Suppose $A\in\Alg(\C)$ is separable, and that $Z(A)$ is also separable. 

In this case, the center of $A$ \emph{as a $Z(A)$-algebra}, $C$, is equivalent to $Z(A)$, i.e. $A$ is a central $Z(A)$-algebra. 
\end{lm}
\begin{proof}
Note that \Cref{thm : commlift} implies, together with the homotopy commutativity of $Z(A)$, that $Z(A)$ has an essentially unique commutative algebra structure extending its ($\mathbb E_2$-)algebra structure. 

Furthermore, by \Cref{thm : homod=mod} and \Cref{prop : tensorsep}, we have that $\ho(\Mod_{Z(A)}(\C))\simeq \Mod_{hZ(A)}(\ho(\C))$ as symmetric monoidal categories, compatibly with the lax symmetric monoidal functor to $\ho(\C)$.

As $A$ is separable over $Z(A)$ by \Cref{prop:conversetowersep}, its center in $\Mod_{Z(A)}(\C)$ can be computed in $\ho(\Mod_{Z(A)}(\C))$ by \Cref{cor:centerabsolute}, and thus we may assume that $\C$ is a $1$-category. 

In particular, $A\otimes A\op\to A\otimes_{Z(A)}A\op $ is then an epimorphism, as it is split, so that $\hom_{A\otimes_{Z(A)} A\op}(A,A)\to \hom_{A\otimes A\op}(A,A)$ is a a monomorphism, compatible with the forgetful map to $A$. 

But the first one receives a map from $Z(A)$, as $Z(A)$ is commutative, also compatible with the forgetful map to $A$, and so, because all these maps to $A$ are monomorphisms (as they admit retractions and we are in a $1$-category), this implies the claim.
\end{proof}

We now move on to \Cref{thm:center}. The descent method here is based on:
\begin{prop}\label{prop:centerdescent}
    Let $A$ be a homotopy commutative algebra in $\C$. Assume that there is fully faithful symmetric monoidal functor $\C\to \lim_I \D_i$, where $i\mapsto \D_i$ is a diagram of additively symmetric monoidal \categories{}, and where $I$ has a weakly initial set of objects $I_0$. 

    In this case, if the projection $p_{i_0}(A)$ is separable in $\D_{i_0}$ for all $i_0\in I_0$, then $A$ is separable in $\C$.
\end{prop}
Here, a set of objects $I_0$ in $I$ is weakly initial if any object in $I$ receives a map from some object in $I_0$.
\begin{proof}
    As $I_0$ is a weakly initial set of objects, the assumption on $p_{i_0}(A)$ implies that $p_i(A)$ is separable in every object $i$, and so by \Cref{cor:descenthocomm}, the image of $A$ in $\lim_I\D_i$ is separable. By fully faithfulness, it follows that $A$ is also separable in $\C$. 
\end{proof}
This explains the second half of \Cref{item:sepcenterring} in \Cref{thm:center}: for any (connective) Deligne Mumford stack $X$, $\QCoh(X)$ can be expressed as a limit of \categories{} of the form $\Mod_R(\Sp)$, so if one  can prove the result for those ones, it follows automatically for $\QCoh(X)$. So we will prove \Cref{item:sepcenterring} from \Cref{thm:center} only in the affine case. We begin with:
\begin{lm}\label{lm:centerfield}
Let $A$ be a separable algebra in $\Mod_k(\Sp)$, where $k$ is a field. In this case, the center of $A$ is separable.    
\end{lm}
\begin{proof}
    The homotopy groups functor induces a symmetric monoidal equivalence $\ho(\Mod_k(\Sp))\simeq \mathbf{GrVect}_k$ with the category of graded $k$-vector spaces, so by \Cref{cor:centerabsolute} and \Cref{prop:homsepimpliessep}, it suffices to prove the result in $\mathbf{GrVect}_k$, so let $A$ be a separable algebra therein. 

    We note that by \Cref{rmk:centerretractlinear}, $Z(A)$ is a $Z(A)$-linear retract of $A$. It follows that for any ideal $I$ in $Z(A)$, we have $IA\cap Z(A) = I$. But now, by \Cref{lm:sepimpliessemisimp}, because $\mathbf{GrVect}_k$ is semisimple, $IA$ must be principal, generated by a (graded) central idempotent $e$. In particular, $e\in IA\cap Z(A)$, and so $e\in I$. Thus, $Z(A)$ is semi-simple. 

    It follows that any module over $Z(A)$ is projective, and in particular $A$ is projective over $Z(A)$. Thus, as $Z(A)$-bimodules, we have that $Z(A)$ is a retract of $A$, which is a retract of $A\otimes_k A\op$, which is projective over $Z(A)\otimes_k Z(A)$. Hence $Z(A)$ is projective over $Z(A)\otimes_k Z(A)$, which implies that it is separable. 
\end{proof}
Recall that by \cite[Proposition 1.6]{Neeman}, this means in particular that $Z(A)$ is discrete and an étale algebra overr the field $k$, in the usual sense. 

We then reduce the general case to the discrete case:
\begin{prop}
To prove \Cref{item:sepcenterring} from \Cref{thm:center}, it suffices to prove it in the case where $R$ is discrete. 
\end{prop}
\begin{proof}
    Note that the canonical functor $\Mod_R\to \lim_n \Mod_{R_{\leq n}}$ is fully faithful when restricted to bounded below objects by \cite[19.2.1.5]{SAG}, so it suffices to prove the result for each $R_{\leq n}$, by \Cref{prop:centerdescent}. 

    In particular, as $R_{\leq n+1}\to R_{\leq n}$ is a square zero extension by a connective spectrum, it suffices to prove that the result is stable under such, namely, that the result for $R_{\leq n}$ implies that for $R_{\leq n+1}$. This follows from \cite[Theorem 16.2.0.2]{SAG} and \Cref{prop:centerdescent}: the \category{} of bounded below $R_{\leq n+1}$-modules can be expressed as a pullback where the two corners are the \category{} of bounded below $R_{\leq n}$-modules. 
\end{proof}
\begin{rmk}
    This proof in fact shows that to provide a positive answer to \Cref{question : sepcenter} for a connective $R$, and in the bounded below case, it suffices to provide one for $\pi_0(R)$. 
\pend \end{rmk}
\begin{proof}[Proof of \Cref{item:sepcenterring} from \Cref{thm:center}]
As explained above, we may assume $R$ is a discrete commutative ring. We fix an almost perfect separable algebra $A$ over $R$. We first aim to prove that its center $Z(A)$ is a flat $R$-module. 

By \Cref{cor:centerabsolute}, $Z(A)$ is a retract of $A$ and thus is also almost perfect. By \cite[\href{https://stacks.math.columbia.edu/tag/068V}{Tag 068V}]{stacks-project}, to prove that it is flat, we may therefore basechange to any field and check that the result is in degree $0$. 

Using again \Cref{cor:centerabsolute}, we find that for any map $R\to k$ to a field, $Z(A)\otimes_R k\simeq Z(A\otimes_R k)$. Now $A\otimes_R k$ is a separable algebra over a field, so by the case of fields, i.e. \Cref{lm:centerfield} , $Z(A\otimes_R k)$ is separable. By \cite[Proposition 1.6]{Neeman}, it follows that it is discrete. Thus, $Z(A)$ is indeed flat over $R$, as claimed. 

In particular, it is also discrete. Because it is almost perfect, it follows that it is also finitely presented, as a module over $R$. It also follows that it is finitely presented as an algebra over $R$, and thus  \cite[\href{https://stacks.math.columbia.edu/tag/02GM}{Tag 02GM}]{stacks-project} implies that, to prove that it is étale over $R$, we may check after basechange along any map $R\to k$, $k$ a field. But there \Cref{lm:centerfield} kicks in again: $Z(A)\otimes_R k\simeq Z(A\otimes_R k)$ is separable and hence étale over $k$. 

It follows that $Z(A)$ is étale over $R$. By \Cref{prop:etalesep}, $Z(A)$ is separable.
\end{proof}
We now move on to \Cref{item:sepcentermorava} from \Cref{thm:center}. The proof of this will rely, as in \Cref{section:brauer}, on Milnor modules. However, because the center of an algebra is a notion that really relies on the symmetric monoidal structure of the ambient category, this time we were not able to use a trick as in the proof of \Cref{prop:BrEthy} to use the (not-necessarily-symmetric) monoidal equivalence $\Mil_E\simeq \coMod_{K_*^EK}(\Mod_{K_*}((\mathbf{GrVect}_k))$, so we are only able to give a proof at odd primes, where this equivalence \emph{can} be made symmetric monoidal. 

We will need a bit more about Milnor modules, so we recommend the reader have a deeper look at \cite[Section 6]{hopkinsluriebrauer}. What we called $\Mil_E$ in \Cref{thm:Milnor} is denoted $\Syn_E^\heart$ in \cite{hopkinsluriebrauer}, but there is also a larger \category{} $\Syn_E$\footnote{cf. \Cref{warn:synsyn}} and a fully faithful (Proposition 4.2.5 in \textit{loc. cit.}), symmetric monoidal (Variant 4.4.11 in \textit{loc. cit.}) embedding $\mathrm{Sy}[-]: \Mod_E(\Sp_{K(n)})\to \Syn_E$. We let $\one$ denote the unit of $\Syn_E$, and $\one^{\leq n}$ its truncations. The following is implicit in \cite{hopkinsluriebrauer}:
\begin{lm}\label{lm:convergencesyn}
    Let $X\in \Syn_E$. The canonical map $X\to \lim_n \one^{\leq n}\otimes X$ is an equivalence. In particular, the canonical symmetric monoidal functor $\Syn_E\to\lim_n \Mod_{\one^{\leq n}}(\Syn_E)$ is fully faithful. 
\end{lm}
\begin{proof}
    The second part of the statement follows from the first, as the canonical map $X\to \lim_n \one^{\leq n}\otimes X$ is the unit of the adjunction $\Syn_E\rightleftarrows\lim_n \Mod_{\one^{\leq n}}(\Syn_E)$. 

    For the first part, we simply note that the canonical map $X\to \one^{\leq n}\otimes X$ induces an equivalence upon $n$-truncation, and therefore so do the morphisms $\one^{\leq m}\otimes X\to \one^{\leq n}\otimes X$. Because limits and truncations in $\Syn_E=\Fun^\times(\Mod_E^{\mathrm{mol}},\Ss)$ are pointwise, the claim follows. 
\end{proof}
The following will allow us to reduce to $\Syn_E^\heart$:
\begin{lm}[{\cite[Proposition 7.3.6]{hopkinsluriebrauer}}]\label{lm:cohesivesyn}
    For every $n\geq 0$, basechange along $\one^{\leq n+1}\to \one^{\leq n}$ fits in a pullback square of additively symmetric monoidal \categories{} of the form:
    $$\xymatrix{\Mod_{\one^{\leq n+1}}(\Syn_E) \ar[r] \ar[d] & \Mod_{\one^{\leq n}}(\Syn_E) \ar[d] \\ \Mod_{\one^{\leq n}}(\Syn_E)\ar[r] & \mathcal C}$$
\end{lm}
\begin{proof}[Proof of \Cref{item:sepcentermorava} from \Cref{thm:center}]
We begin by proving the case of $\C= \Mod_E(\Sp_{K(n)})$.

  Let $A\in\Mod_E(\Sp_{K(n)})$ be a separable algebra. By \cite[Proposition 4.2.5, Variant 4.4.11]{hopkinsluriebrauer} and \Cref{cor:centerabsolute}, to prove that its center is separable, it suffices to prove that its image $\mathrm{Sy}[A]\in\Syn_E$  has the same property, and by \Cref{lm:convergencesyn} and \Cref{prop:centerdescent}, it suffices to prove the same result for each $\one^{\leq n}\otimes\mathrm{Sy}[A]\in Mod_{\one^{\leq n}}(\Syn_E)$. 

  By induction, \Cref{prop:centerdescent} and by \Cref{lm:cohesivesyn}, it suffices to prove it for $\one^{\leq 0}\otimes\mathrm{Sy}[A]\in\Mod_{\one^{\leq 0}}(\Syn_E)$. By \cite[Lemma 7.1.1]{hopkinsluriebrauer}, the latter is equivalent to $\pi_0\mathrm{Sy}[A]$\footnote{Denoted $\mathrm{Sy}^\heart[A]$ in \textit{loc. cit.}.}\footnote{This result can be seen as a version of the statement ``$\mathrm{Sy}[A]$ is flat'', see also \cite[Proposition 2.16]{PVK}. Thus in a sense the beginning of this proof is very similar to the proof of \Cref{item:sepcenterring}.}, and the same is true for $\mathrm{Sy}[A\hat \otimes A]$. In other words, $\pi_0\mathrm{Sy}[A]$ is a separable algebra in $\Syn_E^\heart$, so we are reduced to the case of separable algebras in $\Syn_E^\heart$. 

  It is in this last analysis, i.e. that of separable algebras in $\Syn_E^\heart$, that we really use that we were working with $\Mod_E(\Sp_{K(n)})$ and the precise $\Syn_E$ from \cite{hopkinsluriebrauer}. Namely, \cite[Proposition 6.9.1]{hopkinsluriebrauer} states that, at an odd prime, there is a symmetric monoidal equivalence between $\Syn_E^\heart$ and the category of graded modules over a (finite dimensional) cocommutative Hopf algebra over $K_*\cong k[t^{\pm 1}], |t| = 2$. The latter is equivalently described as a category of algebraic representations of an algebraic group, and so, by \Cref{cor:descentcommsep}, one can check that an algebra is separable on underlying objects (this is similar to the proof of \Cref{lm:idempotentcomod}). 

As the center is preserved by this forgetful functor, we are reduced to the case of the category of graded modules over a graded field, where the proof is essentially the same as that of \Cref{lm:centerfield}. 

This concludes the proof for $\C=\Mod_E(\Sp_{K(n)})$. The case of $\C=\Sp_{K(n)}$ follows from this, together with \Cref{prop:centerdescent} and Galois descent for the $K(n)$-local Galois extension $\Sph_{K(n)}\to E$ (in more detail, see \cite[Proposition 10.10]{Akhilgalois}). 
\end{proof}

As for \Cref{question : centralazumaya} we note that all the parts involved in \Cref{question : sepcenter} (its inputs and its answers, positive or negative) only depend on the homotopy category of $\C$. In particular, if one tries to answer the question in full generality, one can try to consider abelian categories, e.g. via Balmer's homological residue fields -- this idea is very clearly apparent in the proof of \Cref{item:sepcentermorava}. 

Ultimately, the questions in \Cref{section:brauer}, in full generality, are questions about symmetric monoidal abelian categories.

\section{Separable algebras and descent in Hochschild homology}\label{section:descent}
Separable algebras are defined in of the bimodule structure of $A$, which is also relevant to the definition of the Hochschild homology of $A$. This connection was already explored in \cite[Section 9]{rognes} and \cite[Section 1]{BRS}.

In this section, we study one aspect of this connection. To explain it, fix some base stably, presentably symmetric monoidal \category{} $\C$, and a commutative algebra $A$ in $\C$. Our goal is to explain in what way, when $B$ is an $A$-algebra which is \emph{absolutely separable}, the Hochschild homology of $B$ can be made to depend ``linearly'' on $B$. This linear dependence will imply that in the absolutely separable case, one can prove strong descent results for $\HH_\C(A)\to \HH_\C(B)$ from similar results for $A\to B$, cf. \Cref{cor:HHdescent} for a general statement and \Cref{cor:Galoisdescent} for a specific example of interest. 
\begin{cons}
Let $A$ be a commutative algebra in $\C$.
Recall from \cite[Corollary 4.2.3.7]{HA} that there is a presentable fibration $\BiMod^A(\C) \to \Alg_A(\C)$ which classifies the functor $B\mapsto \BiMod_B(\C)$\footnote{Note that we take $B$-bimodules in $\C$, even though we view $B$ as an algebra in $\Mod_A(\C)$.}. 

As $A$ is initial in $\Alg_A$, we further get a natural transformation $(A,M)\to (B,M)$ on $\BiMod^A(\C)$. Here, we abuse notation and write $M$ for the restriction of scalars to $A$ of a $B$-bimodule $M$. This induces a natural transformation $\HH_\C(A,M)\to \HH_\C(B,M)$, where $\HH_\C$ is Hochschild homology relative to $\C$, see \cite[Section 1]{horel}. 
\pend \end{cons}
Taking $A = \one $, we recover the usual map $M\to \HH_\C(B,M)$. In terms of relative tensor products, we have $\HH_\C(B,M)\simeq B\otimes_{B\otimes B\op}M$, and this map is induced by the right-$B\otimes B\op$-module map given by $\mu:B\otimes B\op\to B$. In particular, we have, as an immediate consequence of the definition:
\begin{prop}
Let $\C$ be presentably symmetric monoidal and $B$ a separable algebra in $\C$. The natural map $M\to \HH_\C(B,M)$ admits a natural section. 
\end{prop}
As a consequence of \Cref{cor : commsepsplit} and \Cref{rmk:diagprop}, we also have (compare \cite[Proposition 9.2.5]{rognes}):
\begin{prop}\label{prop:HHisyourself}
Let $\C$ be additively presentably symmetric monoidal and $B$ a separable commutative algebra in $B$. The natural map $M\to \HH_\C(B,M)$ admits a natural section, and it is an equivalence if and only if $M$ is a diagonal bimodule, i.e. of the form $\mu^*N$ for some $B$-module $N$.  
\end{prop}
Our next goal is to give a similar condition for the map $\HH_\C(A,M)\to \HH_\C(B,M)$ to be an equivalence. Specializing to $M=B$, we will find that in that situation, $\HH_\C(B)$ is equivalent to $\HH_\C(A,B)$, the latter being ``linear in $B$''. 

The key to finding this condition is the following:
\begin{thm}\label{thm:HHtrans}
    Let $\C$ be presentably symmetric monoidal, $A$ a commutative algebra in $\C$, and $B$ an $A$-algebra. In this case, the functor $M\mapsto \HH_\C(A,M)$, defined on \emph{$B$-bimodules}, admits a canonical lift to $B$-bimodules \emph{in $\Mod_A(\C)$}, or equivalently, to $B\otimes_A B\op$-bimodules. 
    
    With this lift abusively denoted $\HH_\C(A,-)$, we have a natural equivalence $$\HH_A(B,\HH_\C(A,-))\simeq \HH_\C(B,-)$$ 
    under which the canonical map $\HH_\C(A,M)\to \HH_A(B,\HH_\C(A,M))$ is identified with the canonical map $\HH_\C(A,M)\to \HH_\C(B,M)$.
\end{thm}
\begin{proof}
    On $A$-bimodules, the functor $\HH_\C(A,-)$ can equivalently be described as basechange along $\mu: A\otimes A\op\to A$. 
 Consider then the following commutative diagram of algebras in $\C$: 
 \[\begin{tikzcd}
	{A\otimes A} & {B\otimes B\op} \\
	A & {B\otimes_A B\op}
	\arrow[from=1-1, to=2-1]
	\arrow[from=1-1, to=1-2]
	\arrow["p", from=1-2, to=2-2]
	\arrow[from=2-1, to=2-2]
\end{tikzcd}\]  
Upon applying the functor $\LMod_\bullet(\C)$, it induces a commutative diagram (where we leave $\C$ implicit in the notation):
\[\begin{tikzcd}
	{\LMod_{A\otimes A}} & {\LMod_{B\otimes B\op}} \\
	{\LMod_A} & {\LMod_{B\otimes_A B\op}}
	\arrow[from=1-1, to=2-1]
	\arrow[from=1-1, to=1-2]
	\arrow[from=1-2, to=2-2]
	\arrow[from=2-1, to=2-2]
\end{tikzcd}\]
The claim is that this square is horizontally right adjointable, i.e. that the canonical natural transformation in the following square: 
\[\begin{tikzcd}
	{\LMod_{A\otimes A}} & {\LMod_{B\otimes B\op}} \\
	{\LMod_A} & {\LMod_{B\otimes_A B\op}}
	\arrow[from=1-1, to=2-1]
	\arrow[from=1-2, to=1-1]
	\arrow[from=1-2, to=2-2]
	\arrow[from=2-2, to=2-1]
	\arrow[shorten <=5pt, shorten >=5pt, Rightarrow, from=1-1, to=2-2]
\end{tikzcd}\]
is an equivalence. 

All functors involved, as well as this transformation, are canonically $\C$-linear, and this transformation is an equivalence at $B\otimes B\op$. This suffices, as $\LMod_{B\otimes B\op}$ is generated under colimits and tensors with $\C$ by $B\otimes B\op$. 

It follows that the basechange functor $\LMod_{B\otimes B\op}\to \LMod_{B\otimes_A B\op}$ is the desired lift: forgetting down to $\LMod_A$, we see that it sends $M$ to $\HH_\C(A,M)$. 

Furthermore, the diagonal $B\otimes B\op$-module $B$ can be viewed as $p^*B$, where $p:B\otimes B\op\to B\otimes_A B\op$ is the canonical map and we abuse notation by denoting $B$ also the diagonal $B\otimes_AB\op$-module.

It follows that $\HH_\C(B,M)\simeq p^*B\otimes_{B\otimes B\op}M\simeq B\otimes_{B\otimes_A B\op} p_!M$, where $p_!$ is the basechange functor, so that this last term is $\HH_A(B,\HH_\C(A,M))$ by definition.

The identification of the canonical maps comes from making this last equivalence explicit using the projection formula - the details are left to the reader. 
\end{proof}
\begin{rmk}
If we apply this to, e.g. $\cat C = \Sp, A= \mathbb Z$ and $B$ any $\mathbb Z$-algebra, then this says in particular that $\HH_\mathbb Z(B; \THH(\mathbb Z; B))\simeq \THH(B)$. Using a Postnikov filtration, this yields the $\HH$ to $\THH$ spectral sequence \cite[Theorem 4.1]{PirWald}. We chose $\mathbb Z$, but of course any base ring would give a similar formula and spectral sequence. 
\pend \end{rmk}

In particular, when $B$ is separable as an $A$-algebra, $\HH(A;M)$ retracts onto $\HH(B;M)$. Taking $M=B$, $\HH(A;B)$ retracts onto $\HH(B)$. To identify when this retraction is an equivalence, we recall the following definition from \Cref{subsection:galoisex}:
\begin{defn}
Suppose $\C$ admits geometric realizations compatible with the tensor product. 
Let $A\in\CAlg(\C)$, and let $B\in \Alg(\Mod_A(\C))$ be an $A$-algebra. We say $B$ is \emph{absolutely separable} if it is separable over $A$, and furthermore it has a separability idempotent which factors as $\one\to B\otimes B\op\to B\otimes_A B\op$. 
\pend \end{defn}
\begin{rmk}
    We do not require the lift $\one\to B\otimes B\op$ to be an idempotent. 
\pend \end{rmk}
The key point is that, in the commutative case, absolute separability is going to guarantee that $\HH_\C(A,M)$ is a diagonal bimodule whenever $M$ is a diagonal $B$-bimodule. In more detail: 
\begin{prop}\label{prop:absimpliesdiagpreserved}
    Let Let $\C$ be additively presentably symmetric monoidal, $A$ a commutative algebra in $\C$, and $B$ a commutative $A$-algebra. Suppose $B$ is an absolutely separable $A$-algebra. Let $\mu: B\otimes B\op\to B$ be the multiplication map (which is canonically an algebra map, as $B$ is a commutative $A$-algebra).

    In this case, the functor $\HH_\C(A,-): \Mod_{B\otimes B\op}(\C)\to \Mod_{B\otimes_A B\op}(\C)$ defined in \Cref{thm:HHtrans}, precomposed with restriction of scalars $\mu^*:\Mod_B(\C)\to \Mod_{B\otimes B\op}(\C)$, factors through the full subcategory $\Mod_B(\C)\subset \Mod_{B\otimes_A B}(\C)$ of diagonal $B$-bimodules. 
    \end{prop}
\begin{proof}
    The full subcategory $\Mod_B(\C)\subset \Mod_{B\otimes_A B}(\C)$ is closed under colimits and tensors with $\C$, as the inclusion is given by restriction of scalars along the (relative) multiplication map $B\otimes_A B\op\to B$. 
    
    As the composite $\Mod_B(\C)\xrightarrow{\mu^*}\Mod_{B\otimes B\op}(\C)\xrightarrow{\HH_\C(A,-)} \Mod_{B\otimes_A B\op}(\C)$ is also colimit-preserving and $\C$-linear, it suffices to prove that it sends $B$ to a diagonal bimodule.  It thus suffices to prove that $\mu^*B\otimes_{B\otimes B\op}(B\otimes_A B\op)\to  \mu^*B\otimes_{B\otimes B\op}B$ is an equivalence
    
    Following \Cref{cor : commsepsplit}, write $B\otimes_A B\op\simeq B\times C$ for some $B\otimes_A B\op$-algebra $C$. It now suffices to prove that $\mu^*B\otimes_{B\otimes B\op}C = 0$.

 Let $e: 1\to B\otimes B\op$ denote a lift of the separability idempotent (we say ``the'' because of \Cref{cor:uniqueidem}), which exists by assumption. 
Then, on $B\otimes_{B\otimes B\op}C$, $e$ acts as the identity, because it does so on $B$; and as $0$ because it does so on $C$. It follows that $B\otimes_{B\otimes B}C = 0$, so we are done. 
\end{proof}
\begin{rmk}
$e$ acts as the separability idempotent on $B$ and $C$, $B\otimes B\op$-linearly because the map $B\otimes B\op\to B\otimes_A B\op$ is a map of commutative algebras; if it were only a map of associative algebras, it would not necessarily be true.
\pend \end{rmk}
We thus obtain:
\begin{cor}\label{cor : hhlin}
Let $\C$ be additively presentably symmetric monoidal
Let $A,B\in \CAlg(\C)$, and let $f: A\to B$ be a map of commutative algebras. Assume $B$ is absolutely separable over $A$. 

For any $B$-module $M$, viewed as a $B$-bimodule via restriction of scalars along $B\otimes B\op\to B$, the canonical map $\HH_\C(A;M)\to \HH_\C(B;M)$ is an equivalence.

In particular, the canonical map $\HH_\C(A;B)\to \HH_\C(B)$ is an equivalence. 
\end{cor}
\begin{proof}
    By \Cref{thm:HHtrans}, the canonical map $\HH_\C(A,M)\to \HH_\C(B,M)$ is equivalent to the canonical map $\HH_\C(A,M)\to \HH_A(B,\HH_\C(A,M))$, it thus suffices to prove that the latter is an equivalence. By \Cref{prop:HHisyourself}, this is equivalent to $\HH_\C(A,M)$ being a diagonal $B\otimes_A B\op$-module, which follows from the assumption that $B$ is absolutely separable together with \Cref{prop:absimpliesdiagpreserved}. 
\end{proof}
\begin{ex}
Suppose $A$ is a separable algebra in $\C$. For any commutative algebra $C$, $C\otimes A$ is a separable $C$-algebra, and it is in fact absolutely separable. Indeed, the separability idempotent lifts as $1\to A\otimes A\op\to C \otimes A\otimes C\otimes A\op\to (C\otimes A)\otimes_C (C\otimes A\op)$.
\pend \end{ex}
\begin{ex}\label{ex : flatsepabs}
Let $\C = \Mod_R(\Sp)$ for some commutative ring spectrum $R$, and assume $A\to B$ is a flat map of commutative $R$-algebras. We have a map, natural in the $A$-module $M$, of the form $\pi_0(B)\otimes_{\pi_0(A)}\pi_*(M)\to \pi_*(B\otimes_A M)$. By flatness, both sides are homology theories, and using flatness again, this map is an isomorphism when $M=A$, so it is an isomorphism for all $M$. 

In particular if $M=B$, we have an isomorphism $\pi_0(B)\otimes_{\pi_0(A)}\pi_*(B)\cong \pi_*(B\otimes_A B)$. So if $B$ is further separable over $A$, the separability idempotent lifts to $B\otimes B$, i.e. $B$ is absolutely separable. 
\pend \end{ex}
\begin{ex}
Any separable extension of connective commutative ring spectra is absolutely separable, because then $B\otimes B\to B\otimes_A B$ is surjective on $\pi_0$. 
\pend \end{ex}
\begin{ex}
We saw in \Cref{subsection:galoisex} that the Galois extensions $L_p\to \KU_p$ and $\KO\to \KU$ were absolutely separable.   
\pend \end{ex}
\begin{cor}\label{cor:HHdescent}
    Let $\C$ be stably presentably symmetric monoidal
Let $A,B\in \CAlg(\C)$, and let $f: A\to B$ be a map of commutative algebras. Assume $B$ is absolutely separable over $A$. 

If $B$ \emph{is descendable} over $A$ in the sense of \cite[Definition 3.18]{Akhilgalois}, then $\HH_\C(B)$ is descendable over $\HH_\C(A)$.

In particular, $\HH_\C(A)$ is the limit of the Amitsur complex $\lim_\Delta \HH_\C(B)^{\otimes_{\HH(A)}n}$, or equivalently $\lim_\Delta \HH_\C(B^{\otimes_A n})$. 

If $B$ is furthermore faithful $G$-Galois for some discrete group $G$, we have $\HH_\C(A)\simeq \HH_\C(B)^{hG}$. 
\end{cor}
\begin{proof}
For the first part, we note that the functor $\HH_\C(A,-): \Mod_A(\C)\xrightarrow{\mu^*} \BiMod_A(\C)\xrightarrow{\mu_!} \Mod_A(\C)$ is exact and $\Mod_A(\C)$-linear. Furthermore, it sends $A$ to $\HH_\C(A)$. In particular, if $A$ is in the thick tensor ideal of $\Mod_A(\C)$ generated by $B$, then $\HH_\C(A)$ is in the thick tensor ideal of $\Mod_{\HH_\C(A)}(\C)$ generated by $\HH_\C(A,B)\simeq \HH_\C(B)$ (\Cref{cor : hhlin}). 

The second part follows from the first part by \cite[Proposition 3.20]{Akhilgalois}, and the third part by identifying the limit of the Amitsur complex with the homotopy fixed points, using e.g. \cite[Proposition 6.28]{nilpotencedescent}\footnote{See the proof of \cite[Theorem 6.27]{nilpotencedescent} in the case where $\mathcal F$ is the family of free finite $G$-sets for the use of \cite[Proposition 6.28]{nilpotencedescent}.}. 
\end{proof}
\begin{cor}\label{cor:Galoisdescent}
The Galois extensions $L_p\to \KU_p$ and $\KO\to \KU$ have descent in $\THH$ (for the former, we mean $p$-complete $\THH$). 
\end{cor}
\begin{proof}
This follows from \Cref{cor:HHdescent} together with \Cref{ex:Adamssummand}, \Cref{ex:KOKU}. 
\end{proof}
\begin{rmk}
These descent results were already known, cf. \cite[Example 4.6]{Akhildescent}. However, the techniques for this are more sophisticated: here we only use nilpotence techniques, while in the cited \cite{CMNN}, the authors use some chromatically localized variant of nilpotence, called $\epsilon$-nilpotence. 
\pend \end{rmk}
\begin{rmk}
    In the special case of Galois extensions, one could rephrase \Cref{cor:HHdescent} in terms of Mathew's work in \cite{Akhildescent}. Indeed, in the absolutely separable case, one can prove that the map $\HH_\C(A)\otimes_A B\to \HH_\C(B)$ is an equivalence, which by \cite[Proposition 4.3]{Akhildescent} is equivalent to Galois descent. 

    Note that Mathew's criterion \cite[Theorem 1.3]{Akhildescent} for Galois descent in $\THH$ is a special case of \Cref{cor:HHdescent}: by \Cref{prop:etalesep} and \Cref{ex : flatsepabs}, the assumptions therein imply absolute separability. 
\pend \end{rmk}

\appendix
\section{On the cyclic invariance of the trace}\label{app:cyctrace}
The goal of this appendix is to provide a short proof of:
\begin{prop}\label{prop:C2trace}
    Let $\C$ be a symmetric monoidal \category{}, and $A\in\Alg(\C)$ an algebra whose underlying object is dualizable. Let $t:A\to \one$ be the trace form of $A$, defined as $A\to A\otimes A\otimes A^\vee \to A\otimes A^\vee\to \one$, where the first map is the coevalution of $A$, the second is the multiplication of $A$, and the final map is the evaluation of $A$. 

    The trace pairing of $A$, defined as the composite $A\otimes A\to A\to \one$, has a canonical $C_2$-equivariant structure, where the source has the swap action and the target has the trivial action.
\end{prop}
\begin{rmk}
    Informally, the trace pairing is given by $a\mapsto$ the trace on $A$ of the endomorphism $L_a$ given by left multiplication by $a$, i.e. $a\mapsto \tr(L_a\mid A)$. The $C_2$-invariance then comes from the cyclic invariance of traces, namely $\tr(L_aL_b\mid A) \simeq \tr(L_bL_a\mid A)$. In fact, this will be apparent from the proof. 
\pend \end{rmk}
We start by simply giving a quick construction of this $C_2$-equivariant structures, and we conclude the appendix by sketching two other proofs which are slightly more conceptual but also more involved. For this reason, we will not go into the details of these other proofs. We are grateful to Jan Steinebrunner for suggesting the following simple proof:
\begin{proof}
    We observe that the map $A\to A\otimes A^\vee\simeq \End(A)$ is an algebra map, and that therefore, the claim for $A$ reduces to the same claim for $\End(A)$ instead, and thus, to the more general case of $A=\End(x)$ for any dualizable object $x\in\C$. 

    In this situation, we write the trace pairing of $A$ as the composite: $$\End(x)\otimes \End(x)\simeq (x\otimes x^\vee)\otimes (x\otimes x^\vee)\simeq (x\otimes x)\otimes (x^\vee \otimes x^\vee) \simeq (x\otimes x)\otimes (x^\vee \otimes x^\vee)\simeq \End(x)\otimes \End(x)\to \one \otimes \one\simeq \one$$
    where: the first map is the equivalence $\End(x)\simeq x\otimes x^\vee$; the second swaps the middle term $x^\vee\otimes x\simeq x\otimes x^\vee$; the third, importantly, swaps the first term $x\otimes x\simeq x\otimes x$ but leaves the second unchanged; the fourth reorders the terms as in the first two steps; and the fifth is simply $t\otimes t$, where $t:\End(x)\to \one$ is the evaluation $x\otimes x^\vee\to \one$. 

    One easily verifies that this coincides with $t\circ \mu$, where $\mu$ is the multiplication map of $\End(x)$ and $t$ is as above. Indeed, this is a check at the level of homotopy categories. In terms of string diagrams, the long composite described above is:
\begin{equation}\label{fig1}
\begin{tikzpicture}[
    auto,
    col/.style={circle,thick, draw=white!20,
                 inner sep=1pt,minimum size=7mm}]
    \def\xspace{1.2}
    \def\yspace{.8}
    
\node[col] (00) at (-4*\xspace,{-1*\yspace})  {$x$};
    \node[col] (01) at (-4*\xspace,{-2*\yspace})  {$x^\vee$};
    \node[col] (02) at (-4*\xspace,{-3*\yspace})  {$x$};
    \node[col] (03) at (-4*\xspace,{-4*\yspace})  {$x^\vee$};
    \node[col] (10) at (-3*\xspace,{-1*\yspace})  {$x$};
    \node[col] (11) at (-3*\xspace,{-2*\yspace})  {$x$};
    \node[col] (12) at (-3*\xspace,{-3*\yspace})  {$x^\vee$};
    \node[col] (13) at (-3*\xspace,{-4*\yspace})  {$x^\vee$};
        \node[col] (20) at (-2*\xspace,{-1*\yspace})  {$x$};
    \node[col] (21) at (-2*\xspace,{-2*\yspace})  {$x$};
    \node[col] (22) at (-2*\xspace,{-3*\yspace})  {$x^\vee$};
    \node[col] (23) at (-2*\xspace,{-4*\yspace})  {$x^\vee$};
        \node[col] (30) at (-1*\xspace,{-1*\yspace})  {$x$};
    \node[col] (31) at (-1*\xspace,{-2*\yspace})  {$x^\vee$};
    \node[col] (32) at (-1*\xspace,{-3*\yspace})  {$x$};
    \node[col] (33) at (-1*\xspace,{-4*\yspace})  {$x^\vee$};
    
 \draw [-] (00) to [out=0, in = 180] (10) ;
\draw [-] (03) to [out=0, in = 180] (13);
\draw [-] (12) to [out=0, in = 180] (22);
\draw [-] (13) to [out=0, in = 180] (23);
\draw [-] (20) to [out=0, in = 180] (30);
\draw [-] (23) to [out=0, in = 180] (33);
\draw plot [smooth,tension=1.5] coordinates {(-1,-0.75) (0,-1.25) (-1,-1.75)};
\draw plot [smooth,tension=1.5] coordinates {(-1,-2.25) (0,-2.75) (-1,-3.25)};
\draw [-] (01) to (12);
\draw [-] (10) to (21);
\draw [-] (02) to (11);
\draw [-] (11) to (20);
\draw [-] (21) to (32); 
\draw [-] (22) to (31);
\end{tikzpicture}
\end{equation}
whereas the map $t\circ \mu$ is:
\begin{equation}\label{fig2} 
\begin{tikzpicture}[
    auto,
    col/.style={circle,thick, draw=white!20,
                 inner sep=1pt,minimum size=7mm}]
    \def\xspace{1.2}
    \def\yspace{.8}
    
\node[col] (00) at (-4*\xspace,{-1*\yspace})  {$x$};
    \node[col] (01) at (-4*\xspace,{-2*\yspace})  {$x^\vee$};
    \node[col] (02) at (-4*\xspace,{-3*\yspace})  {$x$};
    \node[col] (03) at (-4*\xspace,{-4*\yspace})  {$x^\vee$};

 \draw plot [smooth,tension=1.5] coordinates {(-4,-0.75) (-2,-2) (-4,-3.25)};
  \draw plot [smooth,tension=1.75] coordinates {(-4.25,-1.75) (-3.5,-2) (-4.25,-2.25)};
\end{tikzpicture}
\end{equation}

And one can see how to deform the top one into the bottom one using elementary string diagram operations. 

    The advantage of spelling the second, much simpler map in terms of the long composite is that every step in the long composite is now $C_2$-equivariant. 

    The first step is $C_2$-equivariant because there is a functor $y\mapsto y\otimes y$ from $\C$ to $\C^{BC_2}$. The second step is $C_2$-equivariant because this functor has a canonical symmetric monoidal structure. 

    The third step is the more interesting one (again): we are claiming that the swap map $x\otimes x\to x\otimes x$ is $C_2$-equivariant, where both sides have the ``swap'' $C_2$-stucture. But this is a general fact about objects with an $A$-action, for $A$ an abelian group. More precisely,  let $A$ be an $\mathbb E_2$-monoid, and consider the composite $BA\times BA\times \C^{BA}\xrightarrow{\mu\times \C^{BA}} BA\times \C^{BA}\xrightarrow{ev} \C$, where $\mu:BA\times BA\to BA$ is the mutliplication map remaining on $BA$ from the $\mathbb E_2$-structure on $A$. This map has a mate given by $BA\times \C^{BA}\to \C^{BA}$, and this further has a mate given by $BA\to \Fun(\C^{BA},\C^{BA})$.

    In particular, evaluation at a given $a\in A$ induces an endomorphism of $\id_{\C^{BA}}$, in particular, for each $y\in\C^{BA}$, an $A$-equivariant endomorphism $\rho_a: y\to y$, whose underlying morphism is simply the action of $a$. 

    Plugging in $A=C_2$, we find that the swap map on $x\otimes x$ is $C_2$-equivariant. The rest of the steps is as elementary as the first two steps. It is worth pointing out that the equivalence $\one\otimes\one\simeq \one$ is in fact $C_2$-equivariant for the trivial action on the target. 
\end{proof}

We now give two incomplete sketches of other proofs. They are more involved, but they give more information; in particular while one ``can easily see'' that the above generalizes to $C_n$-equivariance of an $n$-fold trace pairing, for these other ones, it really is immediate. 
\subsubsection{Using the cobordism hypothesis}
The string diagrams \ref{fig1} and \ref{fig2} can be interpreted as barely symbolic tools to convey a proof, or as one dimensional framed manifolds, i.e. as living in the one dimensional cobordism \category{}. By the cobordism hypothesis\footnote{While a proof of the general cobordism hypothesis has only been sketched by Lurie in \cite{luriecob}, in dimension $1$, it has been completely proved, see \cite{harpazcob}.}, the positively framed $0$-dimensional manifold ``point'' is the universal dualizable object, and so to prove that the trace pairing $\End(x)\otimes\End(x)\to\End(x)\to\one$ can be made $C_2$-equivariant for an arbitrary dualizable $x$, it suffices to prove it in the cobordism category. 

The string diagrams that we drew clearly show that, at the very least, this trace pairing admits a homotopy from itself to itself precomposed with the swap $\End(x)\otimes \End(x)\simeq \End(x)\otimes \End(x)$, but the problem lies in the higher coherence of $C_2$-equivariance. It turns out, however, that in the universal case there are no such coherence problems. 

Indeed, the group of orientation-preserving diffeomorphisms of $[0,1]$ rel $\{0,1\}$ is contractible - it is a convex subspace of the space of all maps; and it follows that the group of orientation-preserving diffeomorphisms of any compact $1$-dimensional manifold with no closed components, rel boundary is also contractible. In particular, the subspace of $\map_{\mathrm{Cob}_1^{\mathrm{fr}}}(\End(x)\otimes\End(x), \emptyset)$ spanned by the cobordisms with no ``floating circles'' is contractible - but our trace pairing lives in this subspace ! In particular, up-to-homotopy-$C_2$-equivariance guarantees actual $C_2$-equivariance. 

This proof has the advantage that it also proves that if we further require naturality in $(\C,x)$, then the $C_2$-equivariant structure is essentially unique. For the same reason, $C_n$-equivariance, which is easy $1$-categorically (in fact, follows from $C_2$-equivariance !), also follows from this proof, as the spaces are still contractible. 
\subsubsection{Using Hochschild homology}
This proof will be even more sketchy than the previous one. It relies on folklore facts about Hochschild homology that do not seem to have been completely referenced in the literature yet.  

First, without loss of generality, we assume $\C$ is presentably symmetric monoidal, so that we can use $\C$-linear Hochschild homology at our leisure. 

In this case, for a dualizable algebra $A$, the forgetful map $\LMod_A(\C)\to\C$ is an internal left adjoint in $\Mod_\C(\PrL)$ and thus induces a map $\HH_\C(A)\to \HH_\C(\one)\simeq \one$, cf. \cite{HSS}.

Now, $\HH_\C$ has a canonical $S^1$-action constructed in \cite{HSS} using the cobordism hypothesis, and in \cite{NS} using a construction of $\HH_\C$ as a cyclic bar construction. The latter, however, has only been constructed functorially in morphisms of algebras, and so cannot be used to construct/analyze the transfer $\HH_\C(A)\to \HH_\C(\one)$. It is however a folklore fact that these two constructions agree, with their $S^1$-actions, and if we use this, the rest of the proof is relatively clear: for any $\C$-atomic $x\in \LMod_A(\C)$ (cf. \Cref{defn:atomic}), we have a map $\hom_A(x,x)\to \HH_\C(A)$, and a commutative diagram in $\C$:
\[\begin{tikzcd}
	{\hom_A(x,x)} & {\hom(x,x)} \\
	{\HH_\C(A)} & {\HH_\C(\one)}
	\arrow[from=1-1, to=2-1]
	\arrow[from=2-1, to=2-2]
	\arrow[from=1-2, to=2-2]
	\arrow[from=1-1, to=1-2]
\end{tikzcd}\]
where the right vertical map is the trace map. In particular, for $x= A$ itself, we get:\[\begin{tikzcd}
	A & {\End(A)} \\
	{\HH_\C(A)} & {\HH_\C(\one)}
	\arrow[from=1-1, to=2-1]
	\arrow[from=2-1, to=2-2]
	\arrow[from=1-2, to=2-2]
	\arrow[from=1-1, to=1-2]
\end{tikzcd}\]
so that our trace pairing can be written as $A\otimes A\to A\to\HH_\C(A)\to \HH_\C(\one)\simeq \one$. The map $\HH_\C(A)\to \one$ is $S^1$-equivariant, and the map $A\otimes A\to A\to \HH_\C(A)$ is equivalent to the inclusion of the $1$-simplices of the cyclic bar construction. As the cyclic bar construction extends to a cyclic object, the inclusion of these $1$-simplices is canonically $C_2$-equivariant. 

Thus the composite $A\otimes A\to \HH_\C(A)\to\one$ is $C_2$-equivariant. 

This proof also easily generalizes to the $n$-fold trace pairing and its $C_n$-equivariance. Furthermore, from this perspective, the relation between the $C_n$-equivariances, as $n$-varies, is extremely clear: we simply have a morphism of cyclic objects from the cyclic bar construction of $A$ to $\one$ (the latter as a constant cyclic objects). Thus the $S^1$-equivariant structure on $\HH_\C(A)\to \one$ corresponds exactly to a highly coherent $C_n$-equivariant structure on the $n$-fold trace pairings of $A$, as $n$ varies. 

\section{On epimorphisms of algebras}\label{app:epi}
An earlier version of \Cref{item:epi} in \Cref{lm:general} was stated in terms of epimorphisms of algebras, under the assumption that $\C$ was presentably symmetric monoidal. We later realized, however, that we needed a different condition, namely, instead of requiring that $A\to B$ be an epimorphism of algebras, we needed the canonical map $B\otimes_A B\to B$ to be an equivalence. 

In the commutative case, these two conditions are equivalent because of the following lemma, and because pushouts of commutative algebras are computed as relative tensor products:
\begin{lm}
    Assume $D$ is an \category{} with pushouts, and let $f:x\to y$ be a map in $D$. It is an epimorphism if and only if the induced map $y\coprod_x y\to y$ is an equivalence. 
\end{lm}
However, pushouts of algebras are more complicated. As a result, we do not know if these two conditions are equivalent. The goal of this appendix is to compare these two notions. We prove that the condition about tensor products always implies that the map is an epimorphism, and conversely, that if $\C$ is stably symmetric monoidal, $f$ being an epimorphism implies the condition about tensor products. So for most purposes, as we mainly care about stable \categories{} in this paper, these two notions are equivalent.  We however take this oppotunity to raise this question:
\begin{ques}
    Let $\C$ be a presentably symmetric monoidal \category{}, and let $f:A\to B$ be a morphism in $\Alg(\C)$. Suppose $f$ is an epimorphism - does it follow that the multiplication map $B\otimes_A B\to B$ is an equivalence ? 
\end{ques}
\begin{lm}\label{lm:pshoutslice}
    Let $\D$ be a \category{} with pushouts, and $d\in \D$. A morphism $f: x\to y$ in $\D_{d/}$ is an epimorphism if and only if the underlying morphism in $\D$ is one.
\end{lm}
\begin{proof}
    Let $\E$ be a \category with pushouts. A mophism $f:x\to y$ is an epimorphism if and only if the induced codiagonal $y\coprod_x y\to y$ is an equivalence. 

    If $\D$ admits pushouts, then so does $\D_{d/}$, and they are preserved and reflected by the forgetful functor $\D_{d/}\to \D$. Applying the previous paragraph to $\E= \D_{d/}$ and $\D$ thus implies the desired statement.
\end{proof}
\begin{lm}\label{lm:monopres}
   Let $\C$ be a presentably symmetric monoidal \category, and $f:\D\to \E$ a morphism in $\Mod_\C(\PrL)$. Suppose it is a monomorphism. Denoting its right adjoint by $g$, the unit map $\id_\D\to gf$ is then a (pointwise) monomophism.

   In particular, if $\D$ is stable, it is a pointwise equivalence, and hence $f$ is fully faithful.
\end{lm}
\begin{proof}
Since it is a monomorphism, $\D\to \D\times_\E\D$ is an equivalence. Looking at mapping spaces, it follows that $\map_\D(x,y)\to \map_\D(x,y)\times_{\map_\E(f(x),f(y))}\map_\D(x,y)$ is also an equivalence for all $x,y\in\D$ and hence, using the adjunction, that $$\map_\D(x,y)\to \map_\D(x,y)\times_{\map_\D(x,gf(y))}\map_\D(x,y)$$ is an equivalence for all $x,y\in\D$. 

By the Yoneda lemma, it follows that for all $y\in\D$, $y\to y\times_{gf(y)} y$ is an equivalence for all $y$, i.e. that $y\to gf(y)$ is a monomorphism, as was claimed.

The ``in particular'' follows from the fact that in a stable \category, any monomorphism is an equivalence: indeed, if $y\to z$ is a monomorphism, the pullback square \[\begin{tikzcd}
	y & y \\
	y & z
	\arrow[from=1-1, to=2-1]
	\arrow[from=1-1, to=1-2]
	\arrow[from=2-1, to=2-2]
	\arrow[from=1-2, to=2-2]
\end{tikzcd}\]
is also a pushout square, thus proving the claim. 
\end{proof}
The following is the main result of this appendix:
\begin{lm}\label{lm:episofrings}
    Let $\C$ be a presentably symmetric monoidal \category, and $f:A\to B$ a morphism in $\Alg(\C)$. The morphism $f$ is an epimorphism in $\Alg(\C)$ if and only if the induced functor $\LMod_A(\C)\to \LMod_B(\C)$ is an epimorphism in $\Mod_\C(\PrL)$. In particular, this is so if it is a localization, equivalently if $B\otimes_A B\to B$ is an equivalence.

    If $\C$ is stable, the converse is true. 
\end{lm}
\begin{proof}
    By \cite[Theorem 4.8.5.11]{HA}, the functor $\Alg(\C)\to \Mod_\C(\PrL)_{\C/}, A\mapsto (\LMod_A(\C),A)$ is fully faithful and colimit-preserving. By \Cref{lm:pshoutslice}, $f:A\to B$ is an epimorphism if and only if the induced functor $\LMod_A(\C)\to \LMod_B(\C)$ is an epimorphism in $\Mod_\C(\PrL)$. 

  Examining the co-unit of the extension-restriction adjunction $$\LMod_A(\C)\rightleftarrows\LMod_B(\C)$$
    we see that it is a localization if and only if the map $B\otimes_A B\to B$ is an equivalence. 

    Conversely, assume $\C$ is stable and $f$ is an epimophism. It follows by the same argument that $\RMod_A(\C)\to \RMod_B(\C)$ is an epimorphism. Dualizing, and using that $\Fun^L_\C(\RMod_A(\C),\C)\simeq \LMod_A(\C)$ \cite[Remak 4.8.4.8]{HA}, it follows that $\LMod_B(\C)\to \LMod_A(\C)$ is a monomorphism. As $\C$ is stable, we are in the stable situation of \Cref{lm:monopres}, and can thus conclude that it is fully faithful. It follows that $\LMod_A(\C)\to \LMod_B(\C)$ is a localization, and hence, that $B\otimes_A B\to B$ is an equivalence, as claimed. 
\end{proof}

\bibliographystyle{alpha}
\bibliography{Biblio.bib}

\begin{thebibliography}{CMNN20}

\bibitem[AG60]{AuslanderGoldman}
Maurice Auslander and Oscar Goldman.
\newblock The {B}rauer group of a commutative ring.
\newblock {\em Transactions of the American Mathematical Society}, 97(3):367--409, 1960.

\bibitem[Bal11]{balmerseparability}
Paul Balmer.
\newblock Separability and triangulated categories.
\newblock {\em Advances in Mathematics}, 226(5):4352--4372, 2011.

\bibitem[Bal14]{balmerdegree}
Paul Balmer.
\newblock Splitting tower and degree of tt-rings.
\newblock {\em Algebra Number Theory}, 8(3):767--779, 2014.

\bibitem[Bal16]{balmerNT}
Paul Balmer.
\newblock The derived category of an {\'e}tale extension and the separable {Neeman--Thomason} theorem.
\newblock {\em Journal of the Institute of Mathematics of Jussieu}, 15(3):613--623, 2016.

\bibitem[Bal20]{balmerfields}
Paul Balmer.
\newblock Nilpotence theorems via homological residue fields.
\newblock {\em Tunisian Journal of Mathematics}, 2:359--378, 2020.

\bibitem[Bar17]{barwick}
Clark Barwick.
\newblock Spectral {M}ackey functors and equivariant algebraic {K-theory (I)}.
\newblock {\em Advances in Mathematics}, 304:646--727, 2017.

\bibitem[BC18]{sepstmod}
Paul Balmer and Jon~F Carlson.
\newblock Separable commutative rings in the stable module category of cyclic groups.
\newblock {\em Algebras and Representation Theory}, 21:399--417, 2018.

\bibitem[BCSY22]{chromaticFT}
Tobias Barthel, Shachar Carmeli, Tomer~M Schlank, and Lior Yanovski.
\newblock The {Chromatic Fourier Transform}.
\newblock {\em arXiv preprint arXiv:2210.12822}, 2022.

\bibitem[BDS14]{restrictionétale}
Paul Balmer, Ivo Dell’Ambrogio, and Beren Sanders.
\newblock Restriction to subgroups as {\'e}tale extensions, in topology, {KK}-theory and geometry.
\newblock {\em Algebraic and Geometric Topology}, 2014.

\bibitem[BGS19]{mackeyII}
Clark Barwick, Saul Glasman, and Jay Shah.
\newblock Spectral {M}ackey functors and equivariant algebraic {K-theory, II}.
\newblock {\em Tunisian Journal of Mathematics}, 2(1):97--146, 2019.

\bibitem[BNT18]{bunke2018beilinson}
Ulrich Bunke, Thomas Nikolaus, and Georg Tamme.
\newblock The beilinson regulator is a map of ring spectra.
\newblock {\em Advances in Mathematics}, 333:41--86, 2018.

\bibitem[BP21]{piotrtobi}
Tobias Barthel and Piotr Pstr{\k{a}}gowski.
\newblock Morava {K}-theory and {F}iltrations by {P}owers.
\newblock {\em arXiv preprint arXiv:2111.06379}, 2021.

\bibitem[BRS12]{BRS}
Andrew Baker, Birgit Richter, and Markus Szymik.
\newblock Brauer groups for commutative {S}-algebras.
\newblock {\em Journal of Pure and Applied Algebra}, 216(11):2361--2376, 2012.

\bibitem[BSY22]{ChroNS}
Robert Burklund, Tomer~M Schlank, and Allen Yuan.
\newblock The chromatic nullstellensatz.
\newblock {\em arXiv preprint arXiv:2207.09929}, 2022.

\bibitem[CMNN20]{CMNN}
Dustin Clausen, Akhil Mathew, Niko Naumann, and Justin Noel.
\newblock Descent in algebraic {$ K $}-theory and a conjecture of {Ausoni--Rognes}.
\newblock {\em Journal of the European Mathematical Society}, 22(4):1149--1200, 2020.

\bibitem[CSY18]{AmbiChro}
Shachar Carmeli, Tomer~M Schlank, and Lior Yanovski.
\newblock Ambidexterity in {$T(n)$}-local stable homotopy theory.
\newblock {\em arXiv preprint arXiv:1811.02057}, 2018.

\bibitem[CSY21a]{AmbiHeight}
Shachar Carmeli, Tomer~M Schlank, and Lior Yanovski.
\newblock Ambidexterity and height.
\newblock {\em Advances in Mathematics}, 385:107763, 2021.

\bibitem[CSY21b]{CycloChro}
Shachar Carmeli, Tomer~M Schlank, and Lior Yanovski.
\newblock Chromatic cyclotomic extensions.
\newblock {\em arXiv preprint arXiv:2103.02471}, 2021.

\bibitem[Dev07]{devinatzfinite}
Ethan~S Devinatz.
\newblock Homotopy groups of homotopy fixed point spectra associated to {$E_n$}.
\newblock {\em Proceedings of the Nishida Fest (Kinosaki, 2003), Geometry \& Topology Monographs}, 10:131--145, 2007.

\bibitem[DH04a]{devinatzhopkins}
Ethan~S Devinatz and Michael~J Hopkins.
\newblock Homotopy fixed point spectra for closed subgroups of the {M}orava stabilizer groups.
\newblock {\em Topology}, 43(1):1--47, 2004.

\bibitem[DH04b]{DH}
Ethan~S Devinatz and Michael~J Hopkins.
\newblock Homotopy fixed point spectra for closed subgroups of the {M}orava stabilizer groups.
\newblock {\em Topology}, 43(1):1--47, 2004.

\bibitem[DS18]{dellambrogiosanders}
Ivo Dell'Ambrogio and Beren Sanders.
\newblock A note on triangulated monads and categories of module spectra.
\newblock {\em Comptes Rendus. Math{\'e}matique}, 356(8):839--842, 2018.

\bibitem[GH05]{Goerss-Hopkins}
Paul Goerss and Michael Hopkins.
\newblock Moduli spaces of commutative ring spectra.
\newblock {\em London Mathematical Society Lecture Note Series}, 315:151, 2005.

\bibitem[GHN15]{gepnerhaugsengnikolaus}
David Gepner, Rune Haugseng, and Thomas Nikolaus.
\newblock Lax colimits and free fibrations in {$\infty$-categories}.
\newblock {\em arXiv preprint arXiv:1501.02161}, 2015.

\bibitem[GL21]{GL}
David Gepner and Tyler Lawson.
\newblock Brauer groups and {G}alois cohomology of commutative ring spectra.
\newblock {\em Compositio Mathematica}, 157(6):1211--1264, 2021.

\bibitem[Goe08]{goerssMFG}
Paul~G Goerss.
\newblock Quasi-coherent sheaves on the moduli stack of formal groups.
\newblock {\em arXiv preprint arXiv:0802.0996}, 2008.

\bibitem[Har12]{harpazcob}
Yonatan Harpaz.
\newblock The cobordism hypothesis in dimension 1.
\newblock {\em arXiv preprint arXiv:1210.0229}, 2012.

\bibitem[Har20]{yonatan}
Yonatan Harpaz.
\newblock Ambidexterity and the universality of finite spans.
\newblock {\em Proceedings of the London Mathematical Society}, 121(5):1121--1170, 2020.

\bibitem[HHLN20]{HHLN}
Rune Haugseng, Fabian Hebestreit, Sil Linskens, and Joost Nuiten.
\newblock Lax monoidal adjunctions, two-variable fibrations and the calculus of mates.
\newblock {\em arXiv preprint arXiv:2011.08808}, 2020.

\bibitem[HL13]{hopkinslurieambi}
Michael Hopkins and Jacob Lurie.
\newblock Ambidexterity in {K(n)-}local stable homotopy theory.
\newblock {\em preprint}, 2013.

\bibitem[HL17]{hopkinsluriebrauer}
Michael~J Hopkins and Jacob Lurie.
\newblock On {Brauer Groups of Lubin-Tate Spectra I}.
\newblock {\em preprint available at http://www. math. harvard. edu/\~{} lurie}, 2017.

\bibitem[HNP20]{harpaznuitenprasma}
YONATAN HARPAZ, JOOST NUITEN, and MATAN PRASMA.
\newblock On {$k$}-invariants for {$(\infty,n)$}-categories.
\newblock {\em arXiv preprint arXiv:2011.12723}, 2020.

\bibitem[Hov04]{hovey}
Mark Hovey.
\newblock Operations and co-operations in {M}orava {$ E $}-theory.
\newblock {\em Homology, Homotopy and Applications}, 6(1):201--236, 2004.

\bibitem[Hoy]{hoyoisBrown}
Marc Hoyois.
\newblock Finite {Brown} representability.

\bibitem[HR21]{horel}
Geoffroy Horel and Maxime Ramzi.
\newblock A multiplicative comparison of {Mac Lane} homology and topological {H}ochschild homology.
\newblock {\em Annals of K-Theory}, 6(3):571--605, 2021.

\bibitem[HS98]{HS}
Michael~J Hopkins and Jeffrey~H Smith.
\newblock Nilpotence and stable homotopy theory {II}.
\newblock {\em Annals of mathematics}, 148(1):1--49, 1998.

\bibitem[HSS17]{HSS}
Marc Hoyois, Sarah Scherotzke, and Nicolo Sibilla.
\newblock Higher traces, noncommutative motives, and the categorified {Chern} character.
\newblock {\em Advances in Mathematics}, 309:97--154, 2017.

\bibitem[KT17]{kleintilson}
John~R Klein and Sean Tilson.
\newblock On the moduli space of {$A_\infty$}-structures.
\newblock {\em Manifolds and $ K $-Theory}, 682:141, 2017.

\bibitem[Law20]{lawsonroots}
Tyler Lawson.
\newblock Adjoining roots in homotopy theory.
\newblock {\em arXiv preprint arXiv:2002.01997}, 2020.

\bibitem[LP22]{lueckepeterson}
Kiran Luecke and Eric Peterson.
\newblock There aren't that many {Morava E}-theories.
\newblock {\em arXiv preprint arXiv:2202.03485}, 2022.

\bibitem[Lur08]{luriecob}
Jacob Lurie.
\newblock On the classification of topological field theories.
\newblock {\em Current developments in mathematics}, 2008(1):129--280, 2008.

\bibitem[Lur09]{HTT}
Jacob Lurie.
\newblock {\em Higher topos theory}.
\newblock Princeton University Press, 2009.

\bibitem[Lur10]{lurienotes}
Jacob Lurie.
\newblock Lecture notes on chromatic homotopy theory, 2010.

\bibitem[Lur12]{HA}
Jacob Lurie.
\newblock Higher algebra, 2012.

\bibitem[Lur18a]{lurie2018elliptic}
Jacob Lurie.
\newblock Elliptic cohomology ii: orientations.
\newblock {\em preprint, available at author's webpage}, 2018.

\bibitem[Lur18b]{SAG}
Jacob Lurie.
\newblock Spectral algebraic geometry.
\newblock {\em preprint}, 2018.

\bibitem[Lur22]{kerodon}
Jacob Lurie.
\newblock Kerodon.
\newblock \url{https://kerodon.net}, 2022.

\bibitem[Mac22]{macpherson}
Andrew~W Macpherson.
\newblock A bivariant yoneda lemma and {$(\infty,2)$--}categories of correspondences.
\newblock {\em Algebraic \& Geometric Topology}, 22(6):2689--2774, 2022.

\bibitem[Mat15]{mathewthick}
Akhil Mathew.
\newblock A thick subcategory theorem for modules over certain ring spectra.
\newblock {\em Geometry \& Topology}, 19(4):2359--2392, 2015.

\bibitem[Mat16]{Akhilgalois}
Akhil Mathew.
\newblock The {G}alois group of a stable homotopy theory.
\newblock {\em Advances in Mathematics}, 291:403--541, 2016.

\bibitem[Mat17a]{mathewresidue}
Akhil Mathew.
\newblock Residue fields for a class of rational {$\mathbf E_\infty$}-rings and applications.
\newblock {\em Journal of Pure and Applied Algebra}, 221(3):707--748, 2017.

\bibitem[Mat17b]{Akhildescent}
Akhil Mathew.
\newblock {THH} and base-change for {G}alois extensions of ring spectra.
\newblock {\em Algebraic \& Geometric Topology}, 17(2):693--704, 2017.

\bibitem[MNN17]{nilpotencedescent}
Akhil Mathew, Niko Naumann, and Justin Noel.
\newblock Nilpotence and descent in equivariant stable homotopy theory.
\newblock {\em Advances in Mathematics}, 305:994--1084, 2017.

\bibitem[MS21]{BMS}
Shay~Ben Moshe and Tomer~M Schlank.
\newblock Higher semiadditive algebraic {K}-theory and redshift.
\newblock {\em arXiv preprint arXiv:2111.10203}, 2021.

\bibitem[Nee18]{Neeman}
Amnon Neeman.
\newblock Separable monoids in {$\mathbf{D_{qc}}(X)$}.
\newblock {\em Journal f{\"u}r die reine und angewandte Mathematik (Crelles Journal)}, 2018(738):237--280, 2018.

\bibitem[NP23]{NikoLuca}
Niko Naumann and Luca Pol.
\newblock Separable commutative algebras and {G}alois theory in stable homotopy theories.
\newblock {\em arXiv preprint, arXiv:2305.01259}, 2023.

\bibitem[NS18]{NS}
Thomas Nikolaus and Peter Scholze.
\newblock On topological cyclic homology.
\newblock {\em Acta Math.}, 221(2):203--409, 2018.

\bibitem[PV22]{PVK}
Piotr Pstr{\k{a}}gowski and Paul VanKoughnett.
\newblock Abstract {G}oerss-{H}opkins theory.
\newblock {\em Advances in Mathematics}, 395:108098, 2022.

\bibitem[PW92]{PirWald}
Teimuraz Pirashvili and Friedhelm Waldhausen.
\newblock Mac {L}ane homology and topological {H}ochschild homology.
\newblock {\em J. Pure Appl. Algebra}, 82(1):81--98, 1992.

\bibitem[Rav93]{ravenel1993life}
Douglas~C Ravenel.
\newblock Life after the telescope conjecture.
\newblock In {\em Algebraic K-Theory and Algebraic Topology}, pages 205--222. Springer, 1993.

\bibitem[Rez98]{rezk1998notes}
Charles Rezk.
\newblock Notes on the {Hopkins-Miller} theorem.
\newblock {\em Contemporary Mathematics}, 220:313--366, 1998.

\bibitem[Rog08]{rognes}
John Rognes.
\newblock {\em Galois Extensions of Structured Ring Spectra}, volume 192.
\newblock American Mathematical Soc., 2008.

\bibitem[RW80]{ravenelwilson}
Douglas~C Ravenel and W~Stephen Wilson.
\newblock The {Morava K-theories of Eilenberg-MacLane spaces and the Conner-Floyd} conjecture.
\newblock {\em American Journal of Mathematics}, 102(4):691--748, 1980.

\bibitem[San22]{sandersetale}
Beren Sanders.
\newblock A characterization of finite {\'e}tale morphisms in tensor triangular geometry.
\newblock {\em {\'E}pijournal de G{\'e}om{\'e}trie Alg{\'e}brique}, 6, 2022.

\bibitem[{Sta}23]{stacks-project}
The {Stacks project authors}.
\newblock The stacks project.
\newblock \url{https://stacks.math.columbia.edu}, 2023.

\bibitem[Ste20]{stefanich}
Germ{\'a}n Stefanich.
\newblock Higher sheaf theory {I: C}orrespondences.
\newblock {\em arXiv preprint arXiv:2011.03027}, 2020.

\bibitem[Str99]{strickland}
Neil Strickland.
\newblock Products on {$\mathrm{MU}$}-modules.
\newblock {\em Transactions of the American Mathematical Society}, 351(7):2569--2606, 1999.

\bibitem[SW15]{satiwesterland}
Hisham Sati and Craig Westerland.
\newblock Twisted {Morava K-theory and E}-theory.
\newblock {\em Journal of Topology}, 8(4):887--916, 2015.

\bibitem[W{\"u}r86]{wurgler}
Urs W{\"u}rgler.
\newblock Commutative ring-spectra of characteristic 2.
\newblock {\em Commentarii Mathematici Helvetici}, 61(1):33--45, 1986.

\bibitem[Yua22]{yuan}
Allen Yuan.
\newblock The sphere of semiadditive height 1.
\newblock {\em arXiv preprint arXiv:2208.12844}, 2022.

\end{thebibliography}

\textsc{Institut for Matematiske Fag, K\o benhavns Universitet, Danmark}

\textit{Email adress : }\texttt{maxime.ramzi@math.ku.dk}
\end{document}